\documentclass[10pt]{article}

\usepackage{hyperref}
\hypersetup{
	pdfpagelabels,
	plainpages = false,
    urlcolor = red,
	colorlinks = true,
	linkcolor = blue,
	citecolor = blue,
    linktocpage = true,
	pdftitle = {Lifting Obstructions and Period Homomorphisms},
    pdfauthor = {Karl-Hermann Neeb, Friedrich Wagemann and Christoph Wockel},
    bookmarksopen = true,
    bookmarksopenlevel = 1,
    hypertexnames = false
  }

\usepackage[a4paper,hscale=0.7,vscale=0.8,hmarginratio=1:1,includehead,headheight=14pt]{geometry}

\usepackage{amssymb,amsmath,amscd,dsfont}
\usepackage[all,cmtip,line]{xy}
\CompileMatrices
\newcommand{\blabel}[1]{\setlength\fboxrule{.0em}\fbox{\ensuremath{\scriptstyle #1}}}
\newenvironment{tabsection}{}{}
\newcommand{\op}[1]{\ensuremath{\operatorname{#1}}}
\newcommand{\wt}[1]{\ensuremath{\widetilde{#1}}}
\newcommand{\wh}[1]{\ensuremath{\widehat{#1}}}
\newcommand{\ol}[1]{\ensuremath{\overline{#1}}}
\newcommand{\ul}[1]{\ensuremath{\underline{#1}}}

\newcommand{\cC}{\ensuremath{\mathcal{C}}}
\newcommand{\cD}{\ensuremath{\mathcal{D}}}

\newcommand{\cF}{\ensuremath{\mathcal{F}}}
\newcommand{\cG}{\ensuremath{\mathcal{G}}}
\newcommand{\cH}{\ensuremath{\mathcal{H}}}

\newcommand{\cL}{\ensuremath{\mathcal{L}}}
\newcommand{\cM}{\ensuremath{\mathcal{M}}}

\newcommand{\cP}{\ensuremath{\mathcal{P}}}

\newcommand{\cT}{\ensuremath{\mathcal{T}}}
\newcommand{\cU}{\ensuremath{\mathcal{U}}}
\newcommand{\cV}{\ensuremath{\mathcal{V}}}
\newcommand{\cW}{\ensuremath{\mathcal{W}}}

\newcommand{\cZ}{\ensuremath{\mathcal{Z}}}
\newcommand{\mc}[1]{\ensuremath{\mathcal{#1}}}

\newcommand{\fz}{\ensuremath{\mathfrak{z}}}

\newcommand{\fk}{\ensuremath{\mathfrak{k}}}

\newcommand{\fg}{\ensuremath{\mathfrak{g}}}

\newcommand{\dd}{\op{\tt{d}}}

 \newcommand{\C}{\ensuremath{\mathbb{C}}}
 
 \newcommand{\R}{\ensuremath{\mathbb{R}}}
 
 \newcommand{\N}{\ensuremath{\mathbb{N}}}
 \newcommand{\Z}{\ensuremath{\mathbb{Z}}}

\newcommand{\id}{\ensuremath{\operatorname{id}}}
\newcommand{\pr}{\ensuremath{\operatorname{pr}}}

\newcommand{\ad}{\ensuremath{\operatorname{ad}}}

\newcommand{\Obj}{\ensuremath{\operatorname{Obj}}}
\newcommand{\Mor}{\ensuremath{\operatorname{Mor}}}

\newcommand{\Aut}{\ensuremath{\operatorname{Aut}}}

\newcommand{\Hom}{\ensuremath{\operatorname{Hom}}}

\newcommand{\im}{\ensuremath{\operatorname{im}}}

\newcommand{\tx}[1]{\ensuremath{\text{#1}}}

\newcommand{\SU}{\ensuremath{\operatorname{SU}}}

\newcommand{\End}{\ensuremath{\operatorname{End}}}
\newcommand{\su}{\ensuremath{\operatorname{\mathfrak{su}}}}

\newcommand{\per}{\ensuremath{\operatorname{per}}}

\newcommand{\se}{\ensuremath{\nobreak\subseteq\nobreak}}
\newcommand{\from}{\ensuremath{\nobreak\colon\nobreak}}
\renewcommand{\to}{\ensuremath{\nobreak\rightarrow\nobreak}}

\newcommand{\sprod}{\ensuremath{\mathopen{\langle}\mathinner{\cdot},\mathinner{\cdot}\mathclose{\rangle}}}

\newcommand{\cat}[1]{\ensuremath{\boldsymbol{\operatorname{#1}}}}
\newcommand{\cq}[1]{\ensuremath{q}}
\newcommand{\nelt}{\ensuremath{\mathds{1}}}

\usepackage{graphicx}
\usepackage{calc}
\usepackage{wrapfig}

\usepackage{amssymb,amsmath,amscd}
\usepackage[amsmath,thmmarks]{ntheorem}
\theoremseparator{.}
\theoremsymbol{\rule{1ex}{1ex}}
\theorembodyfont{\upshape}
\newtheorem{definition}{Definition}[section]
\newtheorem{remark}[definition]{Remark}

\newtheorem{example}[definition]{Example}

\newtheorem*{proof}{Proof}

\theorembodyfont{\itshape}

\newtheorem{Theorem}[definition]{Theorem}

\newtheorem*{nntheorem}{Theorem}
\theoremsymbol{}
\newtheorem{lemma}[definition]{Lemma}
\newtheorem{proposition}[definition]{Proposition}
\newtheorem{theorem}[definition]{Theorem}
\newtheorem{corollary}[definition]{Corollary}

\usepackage[all,cmtip,line]{xy}
\CompileMatrices

\usepackage{ulem}
\def\dotuline{\bgroup 
  \ifdim\ULdepth=\maxdimen  
   \settodepth\ULdepth{(j}\advance\ULdepth.4pt\fi
  \markoverwith{\begingroup
  \advance\ULdepth0.08ex
  \lower\ULdepth\hbox{\kern.075em .\kern.05em}%
  \endgroup}\ULon}

\def\wdotuline{\bgroup 
  \ifdim\ULdepth=\maxdimen  
   \settodepth\ULdepth{(j}\advance\ULdepth.4pt\fi
  \markoverwith{\begingroup
  \advance\ULdepth0.08ex
  \lower\ULdepth\hbox{\kern.3em .\kern.2em}%
  \endgroup}\ULon}

\def\dashuline{\bgroup 
  \ifdim\ULdepth=\maxdimen  
   \settodepth\ULdepth{(j}\advance\ULdepth.4pt\fi
  \markoverwith{\kern.1em
  \vtop{\kern\ULdepth \hrule width .2em}%
  \kern.1em}\ULon}

\def\wdashuline{\bgroup 
  \ifdim\ULdepth=\maxdimen  
   \settodepth\ULdepth{(j}\advance\ULdepth.4pt\fi
  \markoverwith{\kern.2em
  \vtop{\kern\ULdepth \hrule width .6em}%
  \kern.2em}\ULon}

\begin{document}

\normalem


\title{Categorified central extensions, \'etale Lie 2-groups
and Lie's Third Theorem for
locally exponential Lie algebras} 
\author{Christoph Wockel\footnote{Fachbereich Mathematik,
Bundesstr.\ 55, D-20146 Hamburg, Germany}\\[.5em]\texttt{\small
christoph@wockel.eu}}\date{} 
\maketitle

\begin{abstract}
 Lie's Third Theorem, asserting that each finite-dimensional Lie algebra
 is the Lie algebra of a Lie group, fails in infinite dimensions. The
 modern account on this phenomenon is the integration problem for
 central extensions of infinite-dimensional Lie algebras, which in turn
 is phrased in terms of an integration procedure for Lie algebra
 cocycles.

 This paper remedies the obstructions for integrating cocycles and
 central extensions from Lie algebras to Lie groups by generalising the
 integrating objects. Those objects obey the maximal coherence that one
 can expect. Moreover, we show that they are the universal ones for the
 integration problem.

 The main application of this result is that a Mackey-complete locally
 exponential Lie algebra (e.g., a Banach--Lie algebra) integrates to a
 Lie 2-group in the sense that there is a natural Lie functor from
 certain Lie 2-groups to Lie algebras, sending the integrating Lie
 2-group to an isomorphic Lie algebra.\\[\baselineskip] \textbf{MSC:}
 22E65, 58H05, 55N20\\[\baselineskip] \textbf{Keywords:}
 infinite-dimensioal Lie 2-group; central extension; cocycle;
 integration of cocycles; Lie's Third Theorem
\end{abstract}

\section*{Introduction}

\begin{tabsection}
 This paper sets out to resolve obstructions for integrating Lie
 algebras and central extensions of them. It is a celebrated theorem
 that each finite-dimensional Lie algebra is the Lie algebra of a Lie
 group, which is known as Lie's Third Theorem. It was proven by Lie in a
 local versions and in full strength by \'Elie Cartan (cf.\
 \cite{Cartan30Le-troisieme-theoreme-fondamental-de-Lie} and references
 therein). It has also been \'Elie Cartan who first remarked in
 \cite{Cartan36La-topologie-des-groupes-de-Lie} that one may also use
 the fact that $\pi_{2}(G)$ vanishes for any finite-dimensional Lie
 group\footnote{originally, Cartan's condition was that the first two
 Betti numbers vanish} to prove Lie's Third Theorem. If $G$ is
 infinite-dimensional, then $\pi_{2}(G)$ does not vanish any more, for
 instance for $C^{\infty}(S^{1},\SU(2))$ or $PU(\mc{H})$. This was used
 by van Est and Korthagen in
 \cite{EstKorthagen64Non-enlargible-Lie-algebras} to construct an
 example of a Banach--Lie algebra which cannot be the Lie algebra of a
 Lie group (cf.\
 \cite{DouadyLazard66Espaces-fibres-en-algebres-de-Lie-et-en-groupes}
 for the corresponding construction for $PU(\mc{H})$).

 However, there is a large class of infinite-dimensional Lie algebras
 which integrate to a \emph{local} Lie group, namely locally exponential
 Lie algebras. In particular, all Banach--Lie algebras belong to this
 class. The non-existence of a (global) Lie group integrating a locally
 exponential Lie algebra may thus be regarded as the obstruction against
 the corresponding local Lie group to enlarge to a (global) Lie group.
 This is why a Lie algebra, which is the Lie algebra of a (global) Lie
 group is often called \emph{enlargeable}, whilst a Lie algebra is
 called \emph{integrable} if it is the Lie algebra of a local Lie group
 (cf.\ \cite{Neeb06Towards-a-Lie-theory-of-locally-convex-groups}).

 The most sophisticated tool for the analysis of enlargeability of
 locally exponential Lie algebras is Neeb's machinery for integrating
 central extensions of infinite-dimensional Lie groups, developed in
 \cite{Neeb02Central-extensions-of-infinite-dimensional-Lie-groups}: if
 $\fz\to \wh{\fg}\to \fg$ is a central extension of Lie algebras and $G$
 is a 1-connected Lie group with Lie algebra $\fg$, then
 $\fz\to \wh{\fg}\to \fg$ integrates to a central extension of $G$ if
 and only if the period group $\per_{\wh{\fg}}(\pi_{2}(G))\se \fz$ is
 discrete. A variant of this theory (cf.\ \cite[Sect.\
 VI.1]{Neeb06Towards-a-Lie-theory-of-locally-convex-groups}) applies in
 particular to a locally exponential Lie algebra $\fg$, since
 $\fz(\fg)\to \fg\to \fg_{\ad}$ is a (generalised) central extension and
 there always exists a 1-connected Lie group $G_{\ad}$ with
 $L(G_{\ad})=\fg_{\ad}$. Thus the obstruction for $\fg$ to be
 non-enlargeable is the non-discreetness of
 $\per_{\fg}(\pi_{2}(G_{\ad}))\se\fz(\fg)$. If $\fg$ is
 finite-dimensional, then $\pi_{2}(G_{\ad})$ vanishes and Lie's Third
 Theorem is immediate. From this point of view the theorem seems to be
 merely a homotopy-theoretic accident.

 Enlarging local groups and integrating central extensions obey a common
 pattern. The obstruction for enlarging a local (Lie) group to a global
 one is an associativity constraint, which is coupled to topological
 properties of the local group (cf.\
 \cite{Smith50Some-Notions-Connected-with-a-Set-of-Generators},
 \cite{Est62Local-and-global-groups.-I},
 \cite{Est62Local-and-global-groups.-II},
 \cite{Olver96Non-associative-local-Lie-groups} and
 \cite{BourginRobart08An-infinite-dimensional-approach-to-the-third-fundamental-theorem-of-Lie}).
 In general, global associativity cannot be achieved as the
 counterexamples, mentioned above, show. In the integration problem for
 cocycles the obstruction for $\per_{\wh{\fg}}(\pi_{2}(G))\se \fz$ to be
 discrete ensures that a cocycle condition holds for a certain universal
 integrating cocycle.\\
 
 The upshot of this paper is that one may \emph{relax} global
 associativity and cocycle conditions at the same time by introducing
 more generalised \emph{but still coherent} objects, like generalised
 group cocycles and 2-groups. It is organised as follows. In the first
 section we line out an integration procedure for Lie algebra cocycles
 to generalised, locally smooth Lie group cocycles. This is the central
 idea of this paper, all other results will build on this. The main
 achievement of this section is the following

 \begin{nntheorem}
  If $\fz$ is Mackey-complete and $\fg$ is a Lie algebra with simply
  connected Lie group $G$, then each continuous Lie algebra cocycle
  $\omega\from\fg\times\fg\to\fz$ integrates to a generalised cocycle on
  $G$. Moreover, the generalised cocycle that we construct is universal
  for generalised cocycles integrating $\omega$.
 \end{nntheorem}
 
 The remaining sections describe interpretations of the results of the
 first section. The second describes an interpretation in the language
 of loop prolongations. It is discussed which aspects cannot be covered
 by loop prolongations, which then leads to an interpretation in the
 language of 2-groups. This is done in sections three and four and the
 corresponding extension theory is introduced in section five. It is
 described which r\^ole \'etalness plays in this setting, and this
 section eventually results in the second main result of this paper.

 \begin{nntheorem}
  If $\fg$ is the Lie algebra of the simply connected Lie group $G$,
  then each topologically split central extension $\fz\to\wh{\fg}\to\fg$
  with Mackey-complete $\fz$ integrates to a smooth generalised central
  extension of \'etale Lie 2-groups.
 \end{nntheorem}
 
 After having worked this out we apply the previous results to the
 (generalised) central extension $\fz(\fg)\to \fg\to \fg_{\ad}$ for
 $\fg$ locally exponential in order to obtain our version of Lie's Third
 Theorem in the next section:

 \begin{nntheorem}
  If $\fg$ is a Mackey-complete locally exponential Lie algebra, then
  there exists an \'etale Lie 2-group $\cG$ such that $\cL(\cG)$ is
  isomorphic to $\fg$.
 \end{nntheorem}
 
 In the end we indicate some directions for further research and give
 some details on locally convex Lie groups in an appendix.

 There exist many links to neighbouring topics, which we will mention
 throughout the text. Amongst those are integrability questions for Lie
 algebroids (Remark \ref{rem:integrating-Lie-algebriods}), String group
 models (Example \ref{ex:string-2-group}), diffeological Lie groups
 (Remark \ref{rem:diffeological-lie-groups}) and connections on
 categorified bundles and $n$-plectic geometry (Remark
 \ref{rem:prospect:differential_geometry_of_generalised_central_extensions}).
 Since many of them need concepts and notation that we provide in the
 text we refrain from summarising them here.
\end{tabsection}


\section*{Acknowledgements}

\begin{tabsection}
 The author enjoyed a scholarship in the Graduiertenkolleg 1493
 ``Mathematische Strukturen in der modernen Quantenphysik'' (G\"ottingen)
 and was supported by the SFB 676 ``Particles, Strings and the Early
 Universe'' (Hamburg) while carrying out the work on this paper.
 He is grateful to Chenchang Zhu for many discussions on the integration
 procedure for Lie algebras and Lie algebroids. He also wishes to thank
 Karl-Hermann Neeb for providing references concerning the classical
 version of Lie's Third Theorem and pointing out Example
 \ref{ex:topologically_non-split_adjoint_algebra}. Thanks go also to
 Sven S. Porst for countless discussions and proof-reading parts of the
 manuscript. Last, but not least, the author is grateful to Urs
 Schreiber for conversations on 2-groups and related topics.

 The author also wants to thank the referee for a thorough job, helping
 in particular to improve the presentation of the paper.
\end{tabsection}

\section*{Conventions}

\begin{tabsection}
 For us a manifold is a Hausdorff space, which is locally homeomorphic
 to open subsets of some locally convex space such that the coordinate
 changes are diffeomorphisms. A Lie group is a group, which also is a
 manifold such that the group operations are smooth (cf.\ Definition
 \ref{def:diffcalcOnLocallyConvexSpaces} for details on this). For $M,N$
 smooth manifolds and $f\from M\to N$ smooth, $Tf\from TM\to TN$ denotes
 the tangent map of $f$. If $M,N$ and $f$ are pointed, then
 $df\from T_{*}M\to T_{*}N$ denotes the differential in the base point.
 Moreover, if $df$ vanishes, then one may define
 $d^{2}f\from T_{*}M\times T_{*}M\to T_{*}N$ in terms of local
 coordinates, where the vanishing of $df$ implies the independence of
 the choice of a chart.
 For us, a locally smooth map on a pointed manifold is a map which is
 smooth on some open neighbourhood of the base-point (and \emph{not} on
 an open neighbourhood of each point).

 Unless stated otherwise, $G$ shall always be a 1-connected Lie group
 with Lie algebra $\fg$, which we usually identify with $T_{e}G$.
 Moreover, $\fz$ shall always denote a Mackey-complete locally convex
 space (in particular, integrals of smooth functions from standard
 simplices to $\fz$ always exist, cf.\ Remark
 \ref{rem:mackey-complete-space}) and $Z$ will denote the abelian Lie
 group $\fz/ \Gamma$ for some discrete subgroup $\Gamma$ (in some
 situations we will choose $\Gamma$ explicitly, but in general any
 discrete subgroup is fine). Unless stated otherwise,
 $\cq{Z}\from \fz\to Z$ will denote the canonical quotient map.

 If $A$ is an abelian Lie group where $G$ acts on (trivially if nothing
 else is said), then we define
 \begin{align*}
  C^{n}(G,A):=\{f\from G^{n}\to A\colon&	f(g_{1},...,g_{n})=0 \tx{ if
  }g_{i}=e\tx{ for some }i\tx{ and }\\ &f\tx{ is smooth on some neighbourhood of }(e,...,e)\in G^{n}\},
 \end{align*}
 the group of normalised locally smooth $A$-valued $n$-cochains on $G$.
 Note that this implies in particular
 \begin{equation*}
 df(e,...,e)(v_{1},...,v_{n})=\sum_{i=1}^{n}df(e,...,e)(0,...,v_{i},...,0)=0.
 \end{equation*}
 On $C^{n}(G,A)$ we denote by
 \begin{multline*}
  \dd_{\op{gp}}\from C^{n}(G,A)\to
  C^{n+1}(G,A),\quad \dd_{\op{gp}}f(g_{0},...g_{n})=\\ g_{0}.f(g_{1},...,g_{n})
  -\sum _{i=0}^{n-1}(-1)^{i}f(g_{0},...,g_{i}g_{i+1},...,g_{n})-
  (-1)^{n}f(g_{0},...,g_{n-1})
 \end{multline*}
 the ordinary group differential (we will also use this formula for
 $\dd_{\op{gp}}f$ in more general situation, where $A$ does not carry a
 Lie group structure and $f$ does not obey any smoothness condition). If
 $f\in C^{2}(G,A)$, $\dd_{\op{gp}}f=0$, and $G$ acts trivially, then
 \begin{equation}\label{eqn:multiplication_on_central_extension}
  (a,g)\cdot
  (b,h)=(a+b+f(g,h),gh)
 \end{equation}
 defines a group structure on $A\times G$, which we denote by
 $A\times _{f}G$.

 We denote by $\Delta^{(n)}\se \R^{n}$ the standard $n$-simplex, which
 we view as a manifold with corners. For a Hausdorff space $X$,
 $C_{n}(X)=\langle C(\Delta^{(n)},X)\rangle_{\Z}$ denotes the group of
 singular $n$-chains in $X$ over $\Z$ and
 $\partial \from C_{n}(X)\to C_{n-1}(X)$ the corresponding singular
 differential. Moreover, $Z_{n}(X)$ and $B_{n}(X)$ denote the
 corresponding cycles and boundaries and $H_{n}(X)$ the singular
 homology of $X$. For $\alpha,\alpha'\in C(\Delta^{(n)},G)$,
 $\alpha+ \alpha'$ and $-\alpha$ always refer to the additive structure
 in $C_{n}(G)$ whilst $\alpha\cdot \alpha'$ and $\alpha^{-1}$ always
 refer to the (point-wise) group structures on the Lie group
 $C(\Delta^{(n)},G)$. Moreover, $G$ acts by left multiplication on
 $C_{n}(G)$ and we take this module structure into account when using
 $\dd_{\op{gp}}$ for $C_{n}(G)$-valued mappings.

 If $\cC$ is a small category, then $\cC_{0}$ and $\cC_{1}$ are the sets
 of objects and morphisms. The structure maps of $\cC$ are always
 denoted by $s,t,\id$ and $\circ$. If $\cF\from \cC \to \cD$ is a
 functor, then $\cF_{0}\from \cC_{0}\to \cD_{0}$ and
 $\cF_{1}\from \cC_{1}\to \cD_{1}$ are the corresponding maps on the set
 of objects and morphisms. Likewise, if
 $\alpha \from \cF\Rightarrow \cF'$ is a natural transformation, then we
 use the same letter to denote the corresponding map
 $\alpha \from \cC_{0}\to \cD_{1}$. The set of isomorphism classes of
 $\cC$ is denoted by $\pi_{0}(\cC)$ and $\pi_{0}(\cF)$ is the induced
 map $\pi_{0}(\cC)\to \pi_{0}(\cD)$. Almost all categories that we will
 encounter in this article will be groupoids, i.e., categories in which
 each morphism is invertible.
\end{tabsection}

\section{Integrating cocycles}\label{sect:integrating_cocycles}

\begin{tabsection}
 This section describes the principal construction of this paper. It is
 an integration procedure for Lie algebra cocycles and generalises the
 approach from
 \cite{Neeb02Central-extensions-of-infinite-dimensional-Lie-groups}. The
 main achievement will be to overcome the obstruction from
 \cite{Neeb02Central-extensions-of-infinite-dimensional-Lie-groups} for
 the aforementioned integration procedure by passing from group cocycles
 with coefficients in an abelian Lie group to group cocycles with
 coefficients in a complex of abelian Lie groups. We first recall the
 setting and the results from
 \cite{Neeb02Central-extensions-of-infinite-dimensional-Lie-groups}.
\end{tabsection}

\begin{definition}
 Let $\fg$ be a topological Lie algebra and $\fz$ be a topological
 vector space. A Lie algebra \emph{cocycle} is a continuous bilinear map
 $\omega \from\fg\times\fg\to\fz$ satisfying
 $\omega (x,y)=-\omega (y,x)$ and
 \[
  \omega ([x,y],z)+ \omega ([y,z],x)+ \omega ([z,x],y)=0.
 \]
 The cocycle $\omega$ is said to be a \emph{coboundary} if there exists
 a continuous linear map $b\from \fg\to\fz$ with
 $\omega (x,y)=b([x,y])$. The vector space of cocycles is denoted by
 $Z^{2}_{c}(\fg,\fz)$ and the space of coboundaries
 $B^{2}_{c}(\fg ,\fz)$ is a sub space of $Z^{2}_{c} (\fg,\fz)$.
 The vector space
 $H^{2}_{c} (\fg,\fz):=Z^{2}_{c}(\fg ,\fz)/B^{2}_{c}(\fg ,\fz)$ is
 called the (second continuous) \emph{Lie algebra cohomology} of $\fg$
 with coefficients in $\fz$. Two cocycles $\omega$ and $\omega '$
 are called \emph{equivalent} if $[\omega]=[\omega ']$ in
 $H^{2}_{c} (\fg,\fz)$.
\end{definition}

\begin{remark}
 Lie algebra cohomology is a concept that unfolds its importance in
 particular when considering infinite-dimensional Lie algebras. For
 instance, Whitehead's lemma \cite[Thm.\ III.13]{Jacobson62Lie-algebras}
 asserts that $H^{2}_{c}(\fg ,\fz)$ vanishes if $\fg$ and $\fz$ are
 finite-dimensional and $\fg$ is semi-simple.

 The importance of $H^{2}_{c}(\fg,\fz)$ comes from the fact that it
 classifies (topologically split)\footnote{Our central extensions of Lie
 algebras are always assumed to be topologically split, but different
 authors follow different conventions for this.} \emph{central
 extensions} of topological Lie algebras, i.e.\ short exact sequences
 $\fz\to\wh{\fg}\xrightarrow{q}\fg$ for which there exists a continuous
 and linear right inverse of $q$ \cite[Prop.\
 V.2.10]{Neeb06Towards-a-Lie-theory-of-locally-convex-groups}. In
 infinite dimensions a prominent example for a non-trivial
 $H^{2}_{c}(\fg,\fz)$ comes from $\fg=C^{\infty}(S^{1},\fk)$, $\fz=\R$
 and the Kac-Moody cocycle
 \begin{equation}\label{eqn:Kac-Moody-cocycle}
  \omega_{\sprod} \from \fg\times\fg \to \R, \quad (f,g)\mapsto \int_{S^{1}} 
  \langle f(t),g'(t) \rangle dt,
 \end{equation}
 where $\langle \cdot,\cdot \rangle$ is the Killing form of the
 finite-dimensional simple Lie algebra $\fk$.
\end{remark}

\begin{tabsection}
 In \cite{Neeb02Central-extensions-of-infinite-dimensional-Lie-groups}
 it is described how Lie algebra cocycles may be integrated to (locally
 smooth) group cocycles. We shall now introduce slightly more general
 objects (cf.\ \cite[Sect.\ 2]{Breen92Theorie-de-Schreier-superieure})
 covering in particular the (locally smooth) group cocycles from
 \cite{Neeb02Central-extensions-of-infinite-dimensional-Lie-groups} (see
 also
 \cite{WeinsteinXu91Extensions-of-symplectic-groupoids-and-quantization}
 or \cite{TuynmanWiegerinck87Central-extensions-and-physics} for other
 occurrences of this concept).
\end{tabsection}

\begin{definition}
 Let $G$ be an arbitrary Lie group and $A\xrightarrow{\tau} Z$ be a
 morphism of abelian Lie groups. Then a \emph{generalised group cocycle}
 on $G$ with coefficients $A\xrightarrow{\tau}Z$ (shortly called
 generalised cocycle if the setting is understood) consists of two maps
 $F\in C^{2}(G,Z)$ and $\Theta \in C^{3}(G,A)$ such that
 \begin{gather}
  \dd_{\op{gp}}F =\tau \circ\Theta
  \label{eqn:cocycle-identity-for-generalised-cocycle-1}\\
  \dd_{\op{gp}} \Theta=0.
  \label{eqn:cocycle-identity-for-generalised-cocycle-2}
 \end{gather}
 A \emph{morphism} of generalised cocycles
 $(\varphi ,\psi)\from (F,\Theta)\to (F',\Theta ')$ consists of two maps
 $\varphi \in C^{1}(G,Z)$ and $\psi \in C^{2}(G,A)$ such that
 $F=F'+\dd_{\op{gp}}\varphi +\tau \circ \psi$ and
 $\Theta=\Theta'+ \dd_{\op{gp}}\psi$. Furthermore, a \emph{2-morphism}
 $\gamma \from (\varphi ,\psi)\Rightarrow (\varphi' ,\psi')$ between two
 morphisms of generalised cocycles is given by a map
 $\gamma \in C^{1}( G,A)$ such that $\psi=\psi'+\dd_{\op{gp}}\gamma$.
\end{definition}

\begin{tabsection}
 If we view discrete groups as zero-dimensional Lie groups, then the
 preceding definition also yields the concept of generalised cocycles
 without any smoothness assumptions. That is why we do not explicitly
 distinguish between cocycles with or without smoothness conditions.

 In this paper we shall mostly deal with the case that $A$ is a discrete
 group. This implies that $\Theta$ vanishes on some
 identity neighbourhood for smooth maps are in particular continuous.
\end{tabsection}

\begin{remark}
 The previous definition reduces to the case of locally smooth
 cohomology $H^{2}(G,Z)$ (cf.\ \cite[Def.\
 4.4]{Neeb02Central-extensions-of-infinite-dimensional-Lie-groups}) if
 we consider generalised cocycles modulo morphisms with coefficients
 $0\to Z$. Generalised cocycles with $(0\to Z)$-coefficients will
 sometimes be called \emph{2-cocycles} (or simply cocycles if the
 dimension is understood) with coefficients (or values in) $Z$. Similar
 to the case of Lie algebras, $H^{2}(G,Z)$ classifies \emph{central
 extensions of Lie groups}, i.e., short exact sequences
 $Z\to \wh{G}\to G$ possessing smooth local sections\footnote{This is
 equivalent to demanding that $Z\to \wh{G}\to G$ is a locally trivial
 principal bundle. Our central extensions of Lie groups are always
 assumed to be locally trivial principal bundles, but as above,
 different authors follow different conventions for this} (see
 \cite[Prop.\
 4.2]{Neeb02Central-extensions-of-infinite-dimensional-Lie-groups} and
 Example \ref{ex:cocycle-for-the-universal-covering}). Note also that a
 generalised cocycle $(F,\Theta)$ yields a 2-cocycle
 $\cq{(Z/\im (\tau))} \op{\circ} F$ with values in $Z/\im (\tau)$
 provided $\im(\tau)$ is discrete. In this case,	we call
 $\cq{(Z/\im (\tau))} \circ F$ the \emph{band} of $(F,\Theta)$.

 If we take coefficients $A\to 0$ and consider generalised cocycles
 modulo morphisms, then this yields the corresponding higher locally
 smooth cohomology $H^{3}(G,A)$ (cf.\ \cite[Def.\
 V.2.5]{Neeb06Towards-a-Lie-theory-of-locally-convex-groups}).
 Generalised cocycles with $(A\to 0)$-coefficients will sometimes also
 be called \emph{3-cocycles} (or simply cocycles if the dimension is
 understood) with coefficients (or values in) $A$.
\end{remark}

\begin{tabsection}
 We are now heading for a description of the integration procedure from
 \cite{Neeb02Central-extensions-of-infinite-dimensional-Lie-groups}. In
 order to do so, we give a slightly more conceptual construction in the
 following two lemmata that we will use later on in our generalised
 construction. They describe the simplicial part of the procedure for
 enlarging local group cocycles to global ones (cf.\
 \cite{EstKorthagen64Non-enlargible-Lie-algebras}). Variants of this
 construction are implicitly hidden in
 \cite{Iglesias95La-trilogie-du-moment} and
 \cite{BrylinskiMcLaughlin93A-geometric-construction-of-the-first-Pontryagin-class}.
 However, none of the above authors relates those cocycles to (locally
 smooth) group cohomology.

 Recall that our assumption is that $G$ is a 1-connected Lie group with Lie algebra $\fg$.
\end{tabsection}

\begin{lemma}\label{lem:homology-cocycle}
 Assume  that there exist maps
 $\alpha \from G\to     C^{\infty}(\Delta^{(1)},G)$ and
 $\beta  \from G^{2}\to C^{\infty}(\Delta^{(2)},G)$
 such that
 \begin{gather} \label{eqn:coherent-choice-of-paths-1}
  \alpha_{e}\equiv e,\quad\alpha_{g}(0)=e,\quad\alpha_{g}(1)=g, \quad
  \beta_{g,g}(s,t)=\alpha_{g}(s+t)\quad\tx{ and}\\ %
  \label{eqn:coherent-choice-of-paths-2}
  \partial\beta_{g,h}=\alpha_{g}+g.
  \alpha_{h}-\alpha_{gh}.
 \end{gather}
 Then $\dd_{\op{gp}}\beta \from G^{3}\to C_{2}(G)$ takes values in
 $Z_{2}(G)$ and we have $\dd_{\op{gp}}\Theta_{\beta}=0$ if we set
 $\Theta_{\beta}:=q \op{\circ} \dd_{\op{gp}}\beta$ with
 $q\from Z_{2}(G)\to H_{2}(G)\cong \pi_{2}(G)$ the canonical quotient
 map.
\end{lemma}

We took $\beta$ as sole subscript, indicating the dependence of
$\Theta$ on $\alpha$ and $\beta$, for $\alpha$ is
completely determined by $\beta$.

\begin{proof}
 From \eqref{eqn:coherent-choice-of-paths-2} it follows directly that
 \begin{equation*}
  \partial (\dd_{\op{gp}}\beta)(g,h,k)=\partial (g.\beta_{h,k})-\partial \beta_{gh,k}+
  \partial \beta_{g,hk}-\partial \beta_{g,h}
 \end{equation*}
 vanishes and thus $\Theta_{\beta}$ takes values in $Z_{2}(G)$. That
 $\Theta_{\beta}\in C^{3}(G,\pi_{2}(G))$ follows from
 $\beta_{g,g}(s,t)=\alpha_{g}(s+t)$, for then $\Theta_{\beta} (g,h,k)$
 is null-homotopic if one of $g,h$ or $k$ equals $e$. Moreover,
 $\dd_{\op{gp}}\Theta_{\beta}=0$ follows from $\dd_{\op{gp}}^{2}=0$.
\end{proof}

\begin{lemma}
 If $\alpha',\beta'$ is another pair of maps satisfying
 \eqref{eqn:coherent-choice-of-paths-1} and
 \eqref{eqn:coherent-choice-of-paths-2}, then there exists a map
 $\gamma\from G\to C^{\infty}(\Delta^{(2)},G)$ with
 $\partial \gamma_{g}=\alpha_{g}+g.\alpha_{e}-\alpha'_{g}$. Moreover,
 $b_{\gamma}(g,h)=\beta_{g,h}-\beta'_{g,h}-\dd_{\op{gp}}\gamma$ takes
 values in $Z_{2}(G)$ and satisfies
 $\Theta_{\beta}=\Theta_{\beta'}+ q \circ\dd_{\op{gp}} b_{\gamma}$.
\end{lemma}

\begin{proof}
 Since $G$ is simply connected, the map $\gamma$ exists by \cite[Prop.\
 5.6]{Neeb02Central-extensions-of-infinite-dimensional-Lie-groups}. Just
 as above it is checked that $b_{\gamma}$ takes values in $Z_{2}(G)$.
 Moreover, $\dd_{\op{gp}}^{2}=0$ yields
 $q \circ\dd_{\op{gp}}b_{\gamma}=q \circ\dd_{\op{gp}}(\beta-\beta')=\Theta_{\beta}-\Theta_{\beta'}$.
\end{proof}

\begin{tabsection}
 That $\Theta_{\beta}$ is a cocycle is actually trivial since we wrote
 is as a coboundary of the group cochain $\beta$. The point here is that
 it takes values in the much smaller subgroup $Z_{2}(G)$ and as cocycle
 with values in this group it is \emph{not} trivial. Its projection to
 $\pi_{2}(G)$ is even the other extreme, namely \emph{universal}, at
 least for discrete abelian groups (see
 \cite{PorstWockel08Higher-conneced-covers-of-topological-groups-via-categorified-central-extensions}
 and Example \ref{ex:cocycle-for-the-universal-covering}).

 In general, the maps $\alpha$ and $\beta$ that we are going to choose
 for our construction are pretty arbitrary. However, when fixing a chart
 around the identity then there exists a canonical choice for
 $\alpha_{g}$ and $\beta_{g,h}$ if $g$ and $h$ are ``close'' to the
 identity:
\end{tabsection}

\begin{lemma}\label{lem:homology-cocycle-actual-choice-for-a-chart}
 Let $\varphi \from U\to \wt{U} \se \fg$ be a chart with
 $\varphi (e)=0$, $\varphi (U)$ convex and $\wt{V}\se \wt{U}$ open and
 convex such that $e\in V:=\varphi^{-1}(\wt{V})$ and
 $V\cdot V\se U$. For $g\in U$ we set
 $\tilde{g}:=\varphi(g)$ and set
 $\tilde{g}\mathinner{*}\tilde{h}=\wt{gh}$ for $g,h\in V$. If we define
 \begin{align}
  \label{eqn:cocycle-from-chart-1}
  \alpha_{g}(t) & =\varphi^{-1} (t\cdot\tilde{g}),\\
  \label{eqn:cocycle-from-chart-2}
  \beta_{g,h}(t,s) & = \varphi ^{-1}\left(t(\tilde{g}\mathinner{*} s\tilde{h})+ 
  s(\tilde{g}\mathinner{*} (1-t)\tilde{h})\right)
 \end{align}
 for $g,h\in V$, then these assignments can be extended to mappings
 $\alpha$ and $\beta$ satisfying \eqref{eqn:coherent-choice-of-paths-1}
 and \eqref{eqn:coherent-choice-of-paths-2}. Moreover, if for a different
 chart $\varphi'$ we set $\bar{g}:=\varphi'(g)$ and
 \begin{equation}
  \label{eqn:cocycle-from-chart-3}
  \gamma_{g}(s,t)=
  \varphi^{-1}\left(\frac{s(1-t)}{t+s}\varphi(\varphi'^{-1}((t+s)\bar{{g}})))+t(1+s)\tilde{g}\right)
 \end{equation}
 for $g\in V\cap V'$, then this assignment can be extended to a map
 $\gamma\from G\to C^{\infty}_{*}(\Delta^{(2)},G)$, satisfying
 $\partial \gamma_{g}=\alpha_{g}-\alpha'_{g}$. In addition, if $W\se V$ with
 $W\cdot W\se V$, then $\left.\Theta_{\beta}\right|_{W\times W\times W}$
 and $\left.b_{\gamma}\right|_{W\times W}$ are smooth.
\end{lemma}

\begin{proof}
 It is easily checked that $\alpha_{g}$ and $\beta_{g,h}$ defined as in
 \eqref{eqn:cocycle-from-chart-1} and \eqref{eqn:cocycle-from-chart-2}
 satisfy \eqref{eqn:coherent-choice-of-paths-1} and
 \eqref{eqn:coherent-choice-of-paths-2}. Since $G$ is connected, we may
 choose for each $g\in G\backslash U$ some
 ${\alpha}_{g}\in C^{\infty}(\Delta^{(1)},G)$ with ${\alpha}_{g}(0)=e$
 and ${\alpha}_{g}(1)=g$.

 For $g,h\in G$ with $g\notin V$ or $h\notin V$,
 $\alpha_{g}+g.\alpha_{h}-\alpha_{gh}$ is in $Z_{2}(G)$, and thus there
 exists some ${\beta}_{g,h}\in C^{\infty}(\Delta^{(2)},G)$ with
 $\partial {\beta}_{g,h}=\alpha_{g}+g.\alpha_{h}-\alpha_{gh}$ because $G$ is
 simply connected (cf.\ \cite[Prop.\
 5.6]{Neeb02Central-extensions-of-infinite-dimensional-Lie-groups}).
 Moreover, we may choose $\beta_{g,g}(s,t)=\alpha(s+t)$. It is immediate
 that for $\gamma_{g}$ as defined in \eqref{eqn:cocycle-from-chart-3} we
 have $\partial \gamma_{g}=\alpha_{g}-\alpha_{g}'$. Since $G$ is simply
 connected, we may choose for each $g\notin V\cap V'$ some $\gamma_{g}$
 with $\partial \gamma_{g}=\alpha_{g}-\alpha'_{g}$. The rest is
 immediate.
\end{proof}

We now come to the description of the integration procedure from
\cite{Neeb02Central-extensions-of-infinite-dimensional-Lie-groups}\footnote{Variants
of this procedure for the case of Kac-Moody central extensions can be found,
for instance, in \cite{PressleySegal86Loop-groups},
\cite{Mickelsson87Kac-Moody-groups-topology-of-the-Dirac-determinant-bundle-and-fermionization},
\cite{Brylinski93Loop-spaces-characteristic-classes-and-geometric-quantization}
and \cite{MurrayStevenson03Higgs-fields-bundle-gerbes-and-string-structures}.
Implicitly, the cocycles that we shall construct here are already apparent in
their constructions.}.

\begin{remark}\label{rem:integrationInTheCaseOfADiscretePeriodGroup}
 Associated to each Lie algebra cocycle $\omega\from\fg\times\fg\to\fz$
 is its \emph{period homomorphism}
 $\per_{\omega} \from \pi_{2}(G)\to \fz$. This is given on (piecewise)
 smooth representatives by $[\sigma]\mapsto \int_{\sigma}\omega^{l}$,
 where $\omega^{l}$ is the left-invariant $\fz$-valued 2-form on $G$
 with $\omega^{l}(e)=\omega$ (cf.\
 \cite{Neeb02Central-extensions-of-infinite-dimensional-Lie-groups} or
 \cite{Wockel06A-Generalisation-of-Steenrods-Approximation-Theorem} for
 the fact that each homotopy class contains a smooth representative and
 \cite{Neeb02Central-extensions-of-infinite-dimensional-Lie-groups} for
 the fact $\int_{\sigma}\omega^{l}$ is independent on the choice of a
 representative). We define $F_{\omega,\beta }\from G\times G\to \fz$ by
 \begin{equation}\label{eqn:definition-of-F}
  F_{\omega,\beta }(g,h):=\int_{\beta_{g,h}}\omega^{l},
 \end{equation}
 where $\beta \from G^{2}\to C^{\infty}_{*}(\Delta^{(2)},G)$ is the map
 from Lemma \ref{lem:homology-cocycle-actual-choice-for-a-chart} applied
 to a chart $\varphi$ with $d\varphi (e)=\id_{\fg }$. Since
 $\beta (g,g)$ and $\beta (e,g)$ are null-homotopic, it follows that
 $F_{\omega,\beta}(g,e)=F_{\omega,\beta}(e,g)=0$. That
 $(F_{\omega,\beta},\Theta_{\beta})$ is a generalised cocycle with
 coefficients $\pi_{2}(G)\xrightarrow{\per_{\omega}}\fz$ follows from
 $\dd_{\op{gp}}\Theta_{\beta}=0$, from
 \begin{align}
  \notag
  \dd_{\op{gp}}F(g,h,k)&=F_{\omega,\beta }(h,k)-F_{\omega,\beta }(gh,k)+F_{\omega,\beta }(g,hk)-F_{\omega,\beta }(g,h)\\\notag
  &=\int_{\beta_{h,k}}\omega^{l}-\int_{\beta_{gh,k}}\omega^{l}+
  \int_{\beta_{g,hk}}\omega^{l}-\int_{\beta_{g,h}}\omega^{l}\\\label{eqn:cocycle-identity-for-integrating-cocycle}
  &=\int_{g.\beta_{h,k}-\beta_{gh,k}+\beta_{g,hk}-\beta_{g,h}}\omega^{l}=\per_{\omega}(\Theta_{\beta}(g,h,k))
 \end{align}
 and from the fact that the maps $V\times V\ni (g,h)\mapsto \beta_{g,h}$
 and
 $C^{\infty}(\Delta^{(2)},G)\ni \beta \mapsto \int_{\beta}\omega^{l}$
 are smooth.

 Since Neeb only considers 2-cocycles (and no generalised cocycles), he
 is forced\footnote{From this discussion it is clear that it is
 sufficient to divide out $\Pi_{\omega}:=\per_{\omega}(\pi_{2}(G))$ in
 order to ensure the cocycle identity. From Lemma
 \ref{lem:lemma-for-universality-2} is follows that this is also
 necessary} to consider equation
 \eqref{eqn:cocycle-identity-for-integrating-cocycle} modulo
 $\Pi_{\omega}$ and thus obtains a 2-cocycle
 $f_{\omega,\beta}(g,h):=[F_{\omega,\beta}(g,h)]$ with values in
 $Z_{\omega }:=\fz/\Pi_{\omega}$. The drawback is of course that he
 needs to assume that $\Pi_{\omega}$ is discrete in order to consider
 $Z_{\omega }$ as a Lie group (see Remark
 \ref{rem:diffeological-lie-groups} for a proposal on how to use
 diffeological Lie groups in this context).
\end{remark}

We will now make precise in which sense a group cocycle may ``integrate'' a
Lie algebra cocycle.
Recall that our standing assumption is that $\fg$ is the Lie algebra of
$G$, that $\fz$ is an arbitrary Mackey-complete locally convex
space and that $Z=\fz/\Gamma$ for $\Gamma\leq\fz$ discrete.

\begin{lemma}\label{lem:defived-cocycle}
 Let $A$ be discrete and $A\xrightarrow{\tau}Z$ be a morphism of abelian
 Lie groups. If $F\from G^{2}\to Z$ and $\Theta\from G^{3}\to A$ is a
 generalised cocycle, then $dF(e,e)$ vanishes and we get a Lie algebra
 cocycle
 \[
  L(F)\from \fg\times \fg\to\fz, \quad(x,y)\mapsto
  d^{2}F((x,0),(0,y))-d^{2}F((y,0),(0,x))
 \]
\end{lemma}

\begin{proof}
 Let $U\se G$ be an identity neighbourhood such that
 $\left.F\right|_{U\times U}$ and
 $\left.\Theta\right|_{U\times U\times U}$ are smooth maps. From $F(e,g)=F(g,e)=0$
 it follows that $dF(e,e)$ vanishes. Moreover,
 $\left.\Theta\right|_{U\times U\times U}$ vanishes since it is in
 particular continuous and $A$ is discrete. Thus
 \[
  F(g,h)+F(gh,k)-F(g,hk)-F(h,k)=0
 \]
 for $g,h,k$ in $U$. Since the computation of $L(F)$ in \cite[Lem.\
 4.6]{Neeb02Central-extensions-of-infinite-dimensional-Lie-groups} only
 depends on the values of $F$ on $U\times U$, the same calculation shows
 the claim.
\end{proof}

\begin{definition}\label{def:integration_of_cocycles}
 A generalised cocycle $(F,\Theta)$ as in the previous lemma is said to
 \emph{integrate} a $\fz$-valued Lie-algebra cocycle $\omega$ if $L(F)$
 is equivalent to $\omega$.
\end{definition}

\begin{theorem}\label{thm:integration-of-cocycles}
 The generalised cocycle $(F_{\omega,\beta},\Theta_{\beta})$ from Remark
 \ref{rem:integrationInTheCaseOfADiscretePeriodGroup} integrates
 $\omega$.
\end{theorem}

\begin{proof}
 Since $\left.F\right|_{V\times V}$ coincides with the function
 $f\from V\times V\to\fz$ in \cite[Lem.\
 6.2]{Neeb02Central-extensions-of-infinite-dimensional-Lie-groups},
 associated to the cocycle $\omega$ and the smooth maps
 $\sigma_{g,h}\from \Delta^{(2)}\to G$ from \cite[Lem.\
 6.2]{Neeb02Central-extensions-of-infinite-dimensional-Lie-groups}
 coincide with $\beta_{g,h}$ as defined in Lemma
 \ref{lem:homology-cocycle-actual-choice-for-a-chart}, \cite[Lem.\
 6.2]{Neeb02Central-extensions-of-infinite-dimensional-Lie-groups} shows
 \begin{equation*}
  L(F)(x,y)=d^{2}F((x,0),(0,y))-d^{2}F((y,0),(0,x))=\omega (x,y).
 \end{equation*}
\end{proof}

\begin{tabsection}
 We now argue that the generalised cocycle
 $(F_{\Omega,\beta},\Theta_{\beta})$ does essentially not depend on the
 choices that we made.
\end{tabsection}

\begin{remark}\label{rem:dependence-on-choices}
 The construction in Remark
 \ref{rem:integrationInTheCaseOfADiscretePeriodGroup} and the preceding
 proof depends on the actual choice of the map
 $\beta\from G\times G\to C^{\infty}(\Delta^{(2)},G ) $, which in turn
 depends on the choice of a chart $\varphi$. However, for two different
 choices the resulting cocycles $\Theta_{\beta}$ and
 $\Theta_{\beta'}$ are equivalent by Lemma \ref{lem:homology-cocycle}
 and Lemma \ref{lem:homology-cocycle-actual-choice-for-a-chart}.
 Moreover, if $\gamma\from G\to C^{\infty}_{*}(\Delta^{(2)},G)$ is the
 corresponding map as defined in Lemma
 \ref{lem:homology-cocycle-actual-choice-for-a-chart}, then we obtain a
 morphism
 $(\varphi,\psi)\from (F_{\omega,\beta},\Theta_{\beta})\rightarrow (F_{\omega,\beta'},\Theta_{\beta'})$,
 given by $\psi (g,h)=[b_{\gamma }(g,h)]$ and
 \begin{equation*}
  \varphi(g)=\int _{\gamma_{g} }\omega ^{l}.
 \end{equation*}
 
 If two Lie algebra cocycles $\omega$ and $\omega'$ are equivalent, then
 $\omega(x,y)=\omega'(x,y)+b([x,y])$ for $b\from \fg\to\fz$ linear and
 continuous. This leads to
 \begin{equation*}
  \int_{\beta_{g,h}}(\omega-\omega')^l=
  \int_{\beta_{g,h}}d(b^l)=
  \int_{\partial\beta_{g,h}}b^l=
  \int_{\alpha_g}b^l+
  \int_{g.\alpha_h}b^l-		
  \int_{\alpha_{gh}}b^l		
 \end{equation*}
 by Stokes Theorem. We thus obtain a morphism
 $(\varphi,\psi)\from (F_{\omega,\beta},\Theta_{\beta})\to (F_{\omega',\beta},\Theta_{\beta})$
 with $\psi\equiv 0$ and
 \begin{equation*}
  \varphi(g)=\int_{\alpha_g}b^l.
 \end{equation*}
\end{remark}

\begin{tabsection}
 We conclude this section with showing that the cocycle
 $(F_{\omega,\beta},\Theta_{\beta})$ we constructed here is universal
 for generalised cocycles that integrate $\omega$. This may be seen as a
 substitute for the exact sequence \cite[Thm.\
 7.12]{Neeb02Central-extensions-of-infinite-dimensional-Lie-groups}
 \begin{equation*}
  0\to \op{Ext}_{\op{Lie}}(G,Z)\xrightarrow{L} {H}^{2}_{c}(\fg,\fz)\xrightarrow{P} \Hom(\pi _{2}(G),Z).
 \end{equation*}
 The next lemma is
 the generalisation of the injectivity of $L$ (cf.\ \cite[Prop.\
 7.4]{Neeb02Central-extensions-of-infinite-dimensional-Lie-groups}) for
 not necessarily discrete subgroups $\Gamma\se\fz$.
\end{tabsection}

\begin{lemma}\label{lem:lemma-for-universality-1}
 Let $F\in C^{2}(G,\fz)$ be such that $\dd_{\op{gp}}F$ vanishes on some
 identity neighbourhood of $G^{3}$, $\dd_{\op{gp}} F$ takes values in
 some subgroup $\Gamma$ of $\fz$ and $L(F)$ is
 trivial as a Lie algebra cocycle. Then there exists
 $\varphi\in C^{1}(G,\fz)$ such that $F-\dd_{\op{gp}}\varphi$ vanishes
 on some identity neighbourhood and takes values in $\Gamma$ on
 $G\times G$.
\end{lemma}

\begin{proof}
 First note that $L(F)$ actually defines a Lie algebra cocycle by the
 same argument as in Lemma \ref{lem:defived-cocycle}. Since $L(F)$ is
 trivial, there exists a continuous and linear map $\chi\from \fg\to\fz$
 such that $L(F)=\chi([\mathinner{\cdot},\mathinner{\cdot}])$.

 Let $U,V\se G$ be contractible identity neighbourhoods such that
 $\left.\dd_{\op{gp}}F\right|_{U\times U\times U}$ vanishes,
 $\left.F\right|_{U\times U}$ is smooth and $V^{2}\se U$. Then
 $(\fz\times U,\fz\times V,\mu_{F},(0,e))$ with
 \begin{equation*}
  \mu_{F}((z,g),(w,h)):=(z+w+F(g,h),gh)
 \end{equation*}
 is a local Lie group with Lie algebra $\fz\oplus_{L(F)}\fg$. Since
 $L(F)=\chi([\mathinner{\cdot},\mathinner{\cdot}])$, we have that
 \begin{equation*}
 \fz\oplus _{L(F)}\fg\ni (z,x)\mapsto z+\chi(x)\in \fz
 \end{equation*} 
defines a
 homomorphism of Lie algebras. This we may integrate to a homomorphism
 of local Lie groups, given by $(z,g)\mapsto z+\varphi(g)$. By shrinking
 $V$ if necessary we may assume that $\varphi$ is defined on $V$. That this
 map is a homomorphism implies that $F-\dd_{\op{gp}}\varphi$ vanishes on
 $V\times V$.

 Since $\dd_{\op{gp}}F$ takes values in $\Gamma$, we have that
 $f:=q \circ F\from G\times G\to \fz/\Gamma$ is a group cocycle (where
 $q\from \fz\to\fz/\Gamma$ is the canonical projection) and thus
 $(\fz / \Gamma)\times _{f}G$ is a group (which actually is topological,
 but not Hausdorff in general). Now
 \begin{equation*}
  f_{\varphi}\from V\to (\fz /\Gamma) \times _{f}G,\quad g\mapsto (q(\varphi(g)),g)
 \end{equation*}
 satisfies $f_{\varphi}(g)\cdot f_{\varphi}(h)=f_{\varphi}(g\cdot h)$
 wherever defined. Since $G$ is 1-connected, $f_{\varphi}$
 extends\footnote{The group on the right does not need to be topological
 for this, cf.\ \cite[Cor.\
 A.2.26]{HofmannMorris98The-structure-of-compact-groups}} to a unique
 group homomorphism. This extension is given by
 $g\mapsto (\varphi'(g),g)$ for some function
 $\varphi'\from G\to \fz/\Gamma$. Moreover, $\varphi'$ extends
 $q \circ\varphi$ and satisfies $f-\dd_{\op{gp}}\varphi'\equiv 0$. If we
 choose a lift $s\from \fz/\Gamma\to \fz$ with $q(0)\mapsto 0$, then
 $g\mapsto s(\varphi'(g))$ for $g\notin V$ extends $\varphi$ to all of
 $G$ with the desired properties.
\end{proof}

\begin{remark}
 The previous proof easily adapts to the case where $G$ is not simply
 connected. One may construct $\varphi$ as in the previous proof, but if
 $G$ is just connected, $f_{\varphi}$ does not necessarily extend.
 However, it determines a homomorphism
 $\wt{G}\to (\fz/\Gamma)\times_{f}G$, where $\wt{G}$ is the
 1-connected cover of $G$. Restricting this homomorphism to $\pi _{1}(G)$
 yields a homomorphism $\pi _{1}(G)\to \fz/\Gamma$. If this is trivial,
 $f_{\varphi}$ in fact extends to a homomorphism and the argument
 carries over.
\end{remark}

The following lemma is our version of \cite[Thm.\
 7.9]{Neeb02Central-extensions-of-infinite-dimensional-Lie-groups} for non-discrete $\Gamma$.

\begin{lemma}\label{lem:lemma-for-universality-2}
 If $F\in C^{2}(G,\fz)$ is such that $\dd_{\op{gp}}F$ vanishes on some
 identity neighbourhood, $\dd_{\op{gp}}F$ takes values in
 some subgroup $\Gamma$ of $\fz$ and $L(F)$ is equivalent to $\omega$ as
 a Lie algebra cocycle, then $\per_{\omega}(\pi_{2}(G))\se \Gamma$.
\end{lemma}

\begin{proof}
 Since $\per_{\omega}$ does only depend on the cohomology class of
 $\omega$ (cf.\ \cite[Rem.\
 5.9]{Neeb02Central-extensions-of-infinite-dimensional-Lie-groups}), we
 may assume that $L(F)= \omega$. Set $\Theta:=\dd_{\op{gp}}F$ and let
 $U,V\se G$ be open and contractible identity neighbourhoods such that
 $\left.F\right|_{U\times U}$ is smooth,
 $\left. \Theta\right|_{U\times U\times U}$ vanishes and
 $V\cdot V\se U$. For each $g\in G$ we define
 $\kappa_{g}\in \Omega^{1}(g V,\fz)$ by
 \begin{equation*}
  \kappa_{g}(w_{x})=d_{2}F_{g^{-1}\cdot x}(x^{-1}.w_{x}) \tx{ for }w_{x}\in T_{x}\,gV,
 \end{equation*}
 where $d_{2}F_{g}(w_{h}):=dF(0_{g},w_{h})$ for $g,h\in U$ and
 $w_{h}\in T_{h}U$. This is smooth for $\left.F\right|_{U\times U}$ is
 smooth and a straight forward computation shows
 $d \kappa_{g}=\left.\omega^{l}\right|_{gV}$. For $g,h \in G$ with
 $gV\cap hV\neq \emptyset$ we have $g^{-1}h\in U$. Thus
 $\dd_{\op{gp}}F(g^{-1}h,h^{-1}x,x^{-1}\eta(t))$  vanishes for
 $\eta(t)\in gV\cap hV$ and this implies
 \begin{equation*}
  (\kappa_{g}-\kappa_{h})(w_{x})= d_{2}F_{g^{-1}\cdot h}(h^{-1}.w_{x}).
 \end{equation*}
 If $\alpha\from [0,1]\to gV\cap hV$ is smooth, then this in turn yields
 \begin{align*}
  \int_{\alpha}\kappa_{g}-\kappa_{h}=&\int_{0}^{1}d_{2}F_{g^{-1}h}(h^{-1}.\dot \alpha(t))dt\\
  =&F(g^{-1}h,h^{-1}\alpha(1))-F(g^{-1}h,h^{-1}\alpha(0))\\
  =&F(h,h^{-1}\alpha(1))-F(g,g^{-1}\alpha(1))+\Theta(g,g^{-1}h,h^{-1}\alpha(1))-\\
  &F(h,h^{-1}\alpha(0))+F(g,g^{-1}\alpha(0))-\Theta(g,g^{-1}h,h^{-1}\alpha(0)).
 \end{align*}
 Now let $[\sigma]\in \pi _{2}(G)$ be represented by a smooth map
 $\sigma\from [0,1]^{2}\to G$ such that $\sigma$ maps a neighbourhood of
 $\partial [0,1]^{2}$ to $\{e\}$. Then there exists some $n\in \N$ such
 that for $i,j\in\{0,...,n-1\}$
 \begin{equation*}
  \sigma\left(\left[\frac{i}{n},\frac{i+1}{n}\right]\times 
  \left[\frac{j}{n},\frac{j+1}{n}\right]\right)\se g_{ij} V
 \end{equation*}
 for some $g_{ij}\in G$. We denote by $\sigma_{ij}$ the restriction of
 $\sigma$ to
 $[\frac{i}{n},\frac{i+1}{n}]\times [\frac{j}{n},\frac{j+1}{n}]$. Then
 \begin{equation}\label{eqn:lemma-for-universality-1}
  \per_{\omega}([\sigma])=\int _{\sigma}\omega^{l}=\sum_{i,j=0}^{n-1}\int_{\sigma_{ij}}\omega^{l}=
  \sum_{i,j=0}^{n-1}\int_{\sigma_{ij}} d \kappa_{g_{ij}}=\sum_{i,j=0}^{n-1}\int_{\partial\sigma_{ij}} \kappa_{g_{ij}}
 \end{equation}
 by Stokes Theorem. We parametrise the intersection
 $\sigma_{ij}\cap \sigma_{i+1\, j}$ by
 $\mu_{ij}(t):=\sigma(\frac{i+1}{n},\frac{j+t}{n})$ and
 $\sigma_{ij}\cap \sigma_{ij+1}$ by
 $\nu_{ij}(t)=\sigma(\frac{i+t}{n},\frac{j+1}{n})$. In particular, we
 have the identities
 \begin{align*}
  \mu_{i0}(0)=\mu_{i\,n-1}(1)&=e&\mu_{ij}(1)& =\nu_{ij}(1)& \mu_{ij}(1)&=\nu_{i+1\,j}(0)\\
  \nu_{0j}(0)=\nu_{n-1\,j}(1)&=e&\mu_{i\,j+1}(0)&=\nu_{ij}(1)& \mu_{i\,j+1}(0)&=\nu_{i+1\,j}(0).
 \end{align*}
 Since $\left.\sigma\right|_{\partial [0,1]^{2}}$ vanishes, the
 integrals along $\partial \sigma_{ij}\cap \partial \sigma$ in
 \eqref{eqn:lemma-for-universality-1} vanish and we thus have
 \begin{align*}
  \per_{\omega}([\sigma])+\Gamma=&
  \left(\sum_{i,j=0}^{n-1}\int_{\mu_{ij}} \kappa_{g_{ij}}-\kappa_{g_{i+1\,j}}-
  \sum_{i,j=0}^{n-1}\int_{\nu_{ij}} \kappa_{g_{ij}}-\kappa_{g_{ij+1}}\right)+\Gamma=\\&
  \Bigg(\sum_{i,j=0}^{n-1} \dotuline{F(g_{i+1\,j},g_{i+1\,j}^{-1}\mu_{ij}(1))}-\wdotuline{F(g_{ij},g_{ij}^{-1}\mu_{ij}(1))}+\\&
  \sum_{i,j=0}^{n-1}
  -\wdashuline{F(g_{i+1\,j},g_{i+1\,j}^{-1}\mu_{ij}(0))}+\dashuline{F(g_{ij},g_{ij}^{-1}\mu_{ij}(0))}+\\&
  \sum_{i,j=0}^{n-1}
  -\dashuline{F(g_{ij+1},g_{ij+1}^{-1}\nu_{ij}(1))}+\wdotuline{F(g_{ij},g_{ij}^{-1}\nu_{ij}(1))}+\\&
  \sum_{i,j=0}^{n-1}
  \wdashuline{F(g_{ij+1},g_{ij+1}^{-1}\nu_{ij}(0))}-\dotuline{F(g_{ij},g_{ij}^{-1}\nu_{ij}(0))}\Bigg)+\Gamma.
 \end{align*}
 From the above identities it follows that the correspondingly
 underlined terms cancel out. Thus $\per_{\omega}([\sigma])$ is
 contained in $\Gamma$.
\end{proof}

\begin{lemma}\label{lem:equivalences_of_universality}
 Let $(F',\Theta')$ be a generalised cocycle on $G$ that integrates
 $\omega$ (cf.\ Definition \ref{def:integration_of_cocycles}). Assume
 that $p\in\Hom(\pi_{2}(G),A)$ and $\psi\in C^{2}(G,A)$ are such that
 $p \op{\circ} \Theta_{\beta}= \Theta'+\dd_{\op{gp}} \psi$. Then the
 following are equivalent.
 \begin{enumerate}
        \renewcommand{\labelenumi}{\theenumi}
        \renewcommand{\theenumi}{\roman{enumi})}
  \item \label{lem:equivalences_of_universality:item1}
        $  \tau \op{\circ} p= \cq{Z}\op{\circ} \per_{\omega}$.
  \item \label{lem:equivalences_of_universality:item2}
        $\dd_{\op{gp}}(\cq{Z}\op{\circ} F_{\omega,\beta}-F'-\tau \op{\circ} \psi)=0$.
  \item \label{lem:equivalences_of_universality:item3} 			 There exists
        $\varphi\in C^{1}(G,Z)$ such that
        $\cq{Z} \op{\circ} F_{\omega,\beta}=F'+\dd_{\op{gp}}\varphi +\tau \op{\circ} \psi$.
 \end{enumerate}
\end{lemma}

\begin{proof}
 We fist note that $\cq{Z}\from \fz\to Z$ is a covering map and thus there
 exists a section $s\from Z\to \fz$ such that $s(0)=0$ and $s$ is smooth on some zero neighbourhood.

 $\tx{\ref{lem:equivalences_of_universality:item1}}
 \Rightarrow\tx{\ref{lem:equivalences_of_universality:item2}}$:  We set
 $F^{\sharp}:=s \circ (F'+\tau \op{\circ} \psi)$. By Lemma
 \ref{lem:lemma-for-universality-1} there exists
 $\varphi\in C^{1}(G,\fz)$
 such that
 \begin{equation*}
  \xi:= F_{\omega,\beta}- F^{\sharp} - \dd_{\op{gp}} \varphi
 \end{equation*}
 vanishes on some identity neighbourhood and takes values in
 $\cq{Z}^{-1}(A)$. To show the assertion it suffices to show that
 $\dd_{\op{gp}}(\cq{Z}\op{\circ} \xi)=0$. This follows from
 \begin{align*}
  \tau \op{\circ} p \op{\circ} \Theta_{\beta}&=\cq{Z}\op{\circ} \per_{\omega} \op{\circ} \Theta_{\beta}\\
  \Rightarrow \tau \op{\circ} (\Theta' +\dd_{\op{gp}}\psi)&=\cq{Z}\op{\circ} \dd_{\op{gp}}F_{\omega,\beta}\\
  \Rightarrow \dd_{\op{gp}} F'+\tau \op{\circ}\dd_{\op{gp}} \psi&=\dd_{\op{gp}}(\cq{Z}\op{\circ} \xi+F'+\tau \op{\circ} \psi).
 \end{align*}
 
 $\tx{\ref{lem:equivalences_of_universality:item2}}
 \Rightarrow \tx{\ref{lem:equivalences_of_universality:item3}}$: Applying
 Lemma \ref{lem:lemma-for-universality-1} to
 $s \op{\circ}( \cq{Z}\op{\circ} F_{\omega,\beta}-F'-\tau \op{\circ} \psi )$
 gives $\varphi'\in C^{1}(G,\fz)$ such that
 \begin{equation*}
  s \op{\circ}( \cq{Z}\op{\circ} F_{\omega,\beta}-F'-\tau \op{\circ} \psi ) -\dd_{\op{gp}}\varphi'
 \end{equation*}
 has values in $\Gamma$. This implies
 $\cq{Z}\op{\circ} F_{\omega,\beta}-F'-\tau \op{\circ} \psi -\dd_{\op{gp}}\varphi=0 $
 if we set $\varphi:= \cq{Z}\op{\circ} \varphi'$.

 $\tx{\ref{lem:equivalences_of_universality:item3}}
 \Rightarrow \tx{\ref{lem:equivalences_of_universality:item1}}$: By
 \cite{PorstWockel08Higher-conneced-covers-of-topological-groups-via-categorified-central-extensions},
 $\im(\Theta_{\beta})$ generates $\pi_{2}(G)$. Thus the claim
 follows from
 \begin{equation*}
  \cq{Z}\op{\circ} \per_{\omega} \op{\circ} \Theta_{\beta}= \cq{Z}\op{\circ} \dd_{\op{gp}}F_{\omega,\beta}  =
  \dd_{\op{gp}}(F'+\tau \op{\circ} \psi)=\tau \op{\circ} (\Theta'+\dd_{\op{gp}}\psi)=
  \tau \op{\circ} p \op{\circ} \Theta_{\beta}.
 \end{equation*}
\end{proof}

\begin{proposition}\label{prop:universality}
 Let $(F',\Theta')$ be a generalised cocycle on $G$ that integrates
 $\omega$ (cf.\ Definition \ref{def:integration_of_cocycles}). Then
 there exists a unique $p\from \pi_{2}(G)\to A$ such that
 $p \op{\circ} \Theta_{\beta}=\Theta'+\dd_{\op{gp}}\psi$ for some $\psi\in C^{2}(G,A)$. Moreover,
 $\tau \op{\circ} p= \cq{Z}\op{\circ} \per_{\omega}$.
\end{proposition}

\begin{proof}
 Since $\Theta_{\beta}$ is universal for discrete groups, there exists a
 unique $p_{\Theta'}\in\Hom(\pi_{2}(G),A)$ such that
 $[p_{\Theta'} \op{\circ} \Theta_{\beta}]=[\Theta']$ (cf.\
 \cite{PorstWockel08Higher-conneced-covers-of-topological-groups-via-categorified-central-extensions}),
 which is equivalent to
 $p \op{\circ} \Theta_{\beta}=\Theta'+ \dd_{\op{gp}}\psi$ for some
 $\psi\in C^{2}(G,A)$. Set $p:=p_{\Theta'}$. We will show that
 $\tau \op{\circ} p= \cq{Z}\op{\circ} \per_{\omega}$ and recall the
 construction of $p_{\Theta'}$ for this sake.

 For $H$ an abelian group and an arbitrary cocycle $f\from G^{3}\to H$,
 vanishing on some identity neighbourhood $U\times U\times U$, we
 construct an $H$-valued \v{C}ech-2 cocycle as follows. Take $V\se U$ a
 symmetric open identity neighbourhood such that $V^{2}\se U$. Then the
 sets $V_{g}:=g V$ for $g\in G$ form an open cover of $G$ and
 \begin{equation*}
  \eta(f,V) _{g,h,k}\from V_{g}\cap V_{h}\cap V_{k}\to H,\quad x\mapsto
  -f(g,g^{-1}h,g^{-1}k)-f(g^{-1}h,h^{-1}k,k^{-1}x)
 \end{equation*}
 is smooth since $g^{-1}h$, $g^{-1}k$, $h^{-1}k$ are elements of $U$ if
 $V_{g}\cap V_{h}\cap V_{k}\neq \emptyset$. Moreover, it follows from
 $\dd_{\op{gp}}f(g,g^{-1}h,h^{-1}k,k^{-1}x)=0$ that
 \begin{equation*}
  \eta(f,V)_{g,h,k}(x)=-f(g,g^{-1}h,h^{-1}x)-f(h,h^{-1}k,k^{-1}x)+f(g,g^{-1}k,k^{-1}x),
 \end{equation*}
 showing that $\eta(f,V)_{g,h,k}$ constitutes a \v{C}ech 2-cocycle on
 $G$. The class $[\eta(f)]\in \check{H}^{2}(G,A)$ of this cocycle only
 depends on the equivalence class of $f$ in $H^{3}(G,A)$. Since $PG$ is
 contractible, this transgresses to an $H$-valued \v{C}ech 1-cocycle on
 $\Omega G$, giving rise to a covering $H\to \wh{\Omega G}\to \Omega G$
 of $\Omega G$. Choosing a base point in $\wh{\Omega G}$ turns its
 connected component $\wh{\Omega G}_{0}$ into a central extension
 $H\to \wh{\Omega G}_{0}\to \Omega G$ of Lie groups (cf.\ \cite[App.\
 2]{HofmannMorris98The-structure-of-compact-groups}). Thus there exists
 a covering morphism $P\from \wt{\Omega G}\to \wh{\Omega G}_{0}$, where
 $\wt{\Omega G}$ is the universal covering of $\Omega G$ (note that
 $\Omega G$ is connected since $G$ is assumed to be simply connected).
 The restriction of $P$ to the subgroup
 $\pi_{2}(G)\cong \pi_{1}(\Omega G)\se \wt{\Omega G}$ then gives the
 homomorphism $p_{f}$. In particular, we note that $p_{f}$ only depends
 on $[\eta(f)]$.

 From this it follows that
 $\tau \op{\circ} p= \cq{Z}\op{\circ} \per_{\omega}$ if and only if
 $[\eta(\tau \op{\circ} p \op{\circ} \Theta_{\beta})]=[\eta( \cq{Z}\op{\circ} \per_{\omega} \op{\circ} \Theta_{\beta})]$.
 In order to show the latter we assume that
 $\tau \op{\circ} p \op{\circ} \Theta_{\beta}$ and
 $\cq{Z}\op{\circ} \per_{\omega} \op{\circ} \Theta_{\beta}$ are smooth
 when restricted to $U\times U\times U$ and observe that
 $[p \op{\circ} \Theta_{\beta}]=[\Theta']$ implies
 $\eta(\tau \op{\circ} p \op{\circ} \Theta_{\beta},V)\sim \eta(\tau \op{\circ}\Theta',V)$.
 Now $\tau \op{\circ} \Theta'=\dd_{\op{gp}}F'$ implies
 \begin{gather*}
  \eta({\tau \op{\circ}  \Theta'},V)_{g,h,k}(x)=
  -\tau(  \Theta'(g,g^{-1}h,g^{-1}k)-
  \Theta'(g^{-1}h,h^{-1}k,k^{-1}x))=\\
  -\dd_{\op{gp}}F'(g,g^{-1}h,h^{-1}k)-
  \dd_{\op{gp}}F'(g^{-1}h,h^{-1}k,k^{-1}x)=\\
  F'(g,g^{-1}h)+F'(h,h^{-1}k)-F'(g,g^{-1}k)\\-\big(F'(g^{-1}h,h^{-1}x)+F'(h^{-1}k,k^{-1}x)-F'(g^{-1}k,k^{-1}x)\big).
 \end{gather*}
 This is equivalent to the cocycle
 \begin{equation}\label{eqn:universality1}
  V_{g}\cap V_{h}\cap V_{k}\ni x\mapsto -\big(F'(g^{-1}h,h^{-1}x)+F'(h^{-1}k,k^{-1}x)-F'(g^{-1}k,k^{-1}x)\big)
 \end{equation}
 since $V_{g}\cap V_{h}\ni x\mapsto F'(g,g^{-1}h)\in Z$ is a 1-cochain.
 Similarly,
 $\per_{\omega} \op{\circ} \Theta_{\beta}=\dd_{\op{gp}}F_{\omega,\beta}$
 implies that
 $\eta({\cq{Z}\op{\circ} \per_{\omega} \op{\circ} \Theta_{\beta}},V)$ is
 equivalent to the cocycle
 \begin{equation}\label{eqn:universality2}
  V_{g}\cap V_{h}\cap V_{k}\ni x\mapsto -\cq{Z}\big(F_{\omega,\beta}(g^{-1}h,h^{-1}x)+F_{\omega,\beta}(h^{-1}k,k^{-1}x)-F_{\omega,\beta}(g^{-1}k,k^{-1}x)\big).
 \end{equation}
 Since $F_{\omega,\beta}$ and $F'$ both integrate $\omega$, their
 restrictions to some open neighbourhood (which we may still assume to
 be $U$) are equivalent as local cocycles. Any local group coboundary
 between them gives rise to a coboundary between the \v{C}ech cocycles
 \eqref{eqn:universality1} and \eqref{eqn:universality2}.
\end{proof}

\begin{corollary}
 The generalised cocycle $(F_{\omega,\beta},\Theta_{\beta})$ from Remark
 \ref{rem:integrationInTheCaseOfADiscretePeriodGroup} is universal for
 generalised cocycles integrating $\omega$. This means that if
 $(F',\Theta')$ integrates $\omega$ (cf.\ Definition
 \ref{def:integration_of_cocycles}), then there exists some
 $p\in\Hom( \pi_{2}(G), A)$ and a morphism
 $(\varphi,\psi)\from(\cq{Z}\op{\circ} F_{\omega,\beta},p \op{\circ} \Theta_{\beta})\to (F',\Theta')$
 of generalised cocycles. Moreover, the existence of $(\varphi,\psi)$
 determines $p$ uniquely.
\end{corollary}

\section{Loop prolongations}

In this section we provide the minimal algebraic structure that the generalised
cocycles from the previous section yield. However, it will turn out that these
algebraic structures do not mix very well with the underlying smooth structures,
so we will go to slightly advanced algebraic structures in the next
section to treat smoothness issues appropriately.

\begin{definition}
 A \emph{loop} is a set $X$ together with a map
 $\mu\from X\times X\to X$, $(x,y)\mapsto \mu(x,y)=:x\cdot y$ and a
 distinguished element $e\in X$ such that $x\cdot e=e\cdot x=x$ for all
 $x\in X$ and such that the maps $\lambda_{x},\rho_{x}\from X\to X$,
 $\lambda_{x}(y)=x\cdot y$ and $\rho_{x}(y)=y\cdot x$ are bijective. A
 morphism between two loops $X$ and $Y$ is a map $\varphi\from X\to Y$
 satisfying $\varphi(x\cdot y)=\varphi(x)\cdot \varphi(y)$.
\end{definition}

\begin{remark}\label{rem:loop_prolongation}
 (cf.\ \cite{EilenbergMacLane47Algebraic-cohomology-groups-and-loops})
 Let $X$ be a loop. Then in general we do not have that
 $(x\cdot y)\cdot z$ equals $x\cdot(y\cdot z)$, but since
 $\rho_{(x\cdot y)\cdot z}$ is bijective, there exists a unique element
 $A(x,y,z)$ such that
 \begin{equation*}
  x\cdot (y\cdot z)=A(x,y,z)\cdot((x\cdot y)\cdot z).
 \end{equation*}
 We call the map $A\from X\times X\times X\to X$ the \emph{associator}
 of $X$ and sometimes refer to $A(x,y,z)$ as the associator of $(x,y,z)$.

 We have to add the following data and assumptions in order to come from
 general loops to group cohomology. Suppose that we have a homomorphism
 $\varphi\from X\to H$ for $H$ an arbitrary group. Then the kernel of
 $\varphi$ is a normal subloop\footnote{One should not get confused by
 the different notions of normal subloops, for instance in
 \cite{NagyStrambach02Loops-in-group-theory-and-Lie-theory}
 \cite{Pflugfelder90Quasigroups-and-loops:-introduction},
 \cite{Baer45The-homomorphism-theorems-for-loops} or the one used in
 \cite{EilenbergMacLane47Algebraic-cohomology-groups-and-loops}. One
 easily checks that they are all equivalent, the probably easiest one is
 presented, for instance, in
 \cite{KiechleKinyon04Infinite-simple-Bol-loops}. In particular, the
 usual kernel-epimorphism correspondence goes through, see
 \cite{Baer45The-homomorphism-theorems-for-loops} or \cite[p.\
 14]{NagyStrambach02Loops-in-group-theory-and-Lie-theory} and references
 given there.} of $X$ and since $H$ is a group, all associators are
 contained in $\ker(\varphi)$. If we choose a sub\emph{group} $A$ of
 $\ker(\varphi)$ and assume that
 \begin{align}
  &\varphi\tx{ is surjective},\label{eqn:loop_prolongation-1}\\
  &A(k,x,y)=A(x,k,y)=e\tx{ for all }x,y\in X\tx{ and }
  k\in \ker(\varphi),\label{eqn:loop_prolongation-2}\\
  &k\cdot A(x,y,z)=A(x,y,z)\cdot k\tx{ for all }x,y,z\in X\tx{ and }
  k\in\ker(\varphi),\label{eqn:loop_prolongation-3}\\
  &x\cdot a=a\cdot x\tx{ for all }x\in X\tx{ and }a\in A\label{eqn:loop_prolongation-4},
 \end{align}
 then we call $(A,\varphi\from X\to H)$ a (general) \emph{loop
 prolongation} of $H$ by $A$ (cf.\ \cite[Sect.\
 4]{EilenbergMacLane47Algebraic-cohomology-groups-and-loops}). In this
 case, one can actually show that $\ker(\varphi)$ is a sub\emph{group}
 of $X$
 \cite[(4.6)]{EilenbergMacLane47Algebraic-cohomology-groups-and-loops}
 and that $A$ and each associator is contained in $Z(\ker(\varphi))$. It
 has been shown in
 \cite{EilenbergMacLane47Algebraic-cohomology-groups-and-loops} that if
 we assume, in addition, that all associators are contained in $A$ and
 that $A(x,y,k)=e$ for each $k\in K$, then $A$ factors through a map
 from $H\times H\times H\to A$ which is in fact a cocycle. We will from
 now on assume that this is the case (and drop the adjective ``general''
 to indicate this in the notation).

 Note that the assignment of a 3-cocycles to a loop prolongation gives
 rise to a group isomorphism between (equivalence classes of) loop prolongations
 and $H^{3}(G,A)$
 (with respect to a suitably defined
 product, see
 \cite{EilenbergMacLane47Algebraic-cohomology-groups-and-loops}).
\end{remark}

\begin{lemma}\label{lem:loop_prolongation_from_generalised_cocycle}
 Let $(F,\Theta)$ be a generalised cocycle on the discrete group $H$
 with coefficients $A\se Z$, also discrete. Then
 $\mu((z,g),(w,h))=(z+w+F(g,h),gh)$ endows $Z\times H$ with the
 structure of a loop, which we denote by $Z\times_{F}H$. Moreover, if
 $q\from Z\to Z/A$ is the canonical quotient homomorphism, then
 $\varphi\from Z\times_{F}H\to (Z/A)\times_{(q \op{\circ} F)} H$,
 $(z,h)\mapsto (q(z),h)$ defines a loop prolongation of
 $(Z/A)\times_{(q \op{\circ} F)} H$ by $A$. The group 3-cocycle associated
 to this loop prolongation then equals $\Theta$.
\end{lemma}

\begin{proof}
 That $Z\times_{F}H$ defines a loop is directly checked from the
 definition, as well as condition \eqref{eqn:loop_prolongation-1},
 \eqref{eqn:loop_prolongation-3} and \eqref{eqn:loop_prolongation-4}.
 From \eqref{eqn:multiplication_on_central_extension} is follows
 immediately that $\varphi$ defines a homomorphism. Since
 \begin{multline*}
  ((z,g)(w,h))(v,k)=(z+w+v+F(g,h)+F(gh,k),ghk)=\\
  (\Theta(g,h,k),e)\big((z+w+v+F(h,k)+F(g,hk),ghk)\big)=\\
  (\Theta(g,h,k),e)\big((z,g)((w,h)(v,k))\big)
 \end{multline*}
 follows from $\dd_{\op{gp}}F=\Theta$, we have
 $A((z,g),(w,h),(v,k))=(\Theta(g,h,k),e)$. Thus
 \eqref{eqn:loop_prolongation-2} follows from the normalisation
 conditions of $(F,\Theta)$. From the above equation it is also clear
 that $\Theta$ is the group 3-cocycles associated to this loop
 prolongation.
\end{proof}

\begin{tabsection}
 The preceding lemma shows in particular that the integrating cocycle
 $(F_{\beta,\omega},\Theta_{\beta})$ from Remark
 \ref{rem:integrationInTheCaseOfADiscretePeriodGroup} gives rise to a
 loop prolongation and one might
 be tempted to incorporate smoothness assumptions into the game. But
 we shall see now that in general one does not have a smooth structure
 on $\fz\times_{F_{\omega,\beta}}G$ which has $\fz\times V$ (the subset
 on which $\mu$ already is smooth) as an open
 subset.
\end{tabsection}

\begin{example}\label{ex:Kac-Moody_group_cocycle}
 Let $G=C^{\infty}(S^{1},\SU_{2})$, $\fg=C^{\infty}(S^{1},\su_{2})$ and
 $\omega_{\sprod}$ be the Kac-Moody cocycle from
 \eqref{eqn:Kac-Moody-cocycle}. If we normalise $\sprod$ such that the
 left-invariant extension of
 $\langle[\mathinner{\cdot},\mathinner{\cdot}],\mathinner{\cdot}\rangle$
 is a generator of $H^{3}_{\op{dR}}(\SU_{2},\Z)$, then it follows from
 the calculation in the proof of \cite[Thm.\
 III.9]{MaierNeeb03Central-extensions-of-current-groups} that
 $\per_{{\omega_{\sprod}}}=\Z$. Thus
 $f:=\cq{U(1)} \op{\circ} F_{{\omega_{\sprod}},\beta}$ is a 2-cocycle
 (cf.\ Remark \ref{rem:integrationInTheCaseOfADiscretePeriodGroup}). We
 thus obtain a central extension of Lie groups $  U(1)\to \wh{G}\to G$
 with $\wh{G}:=U(1)\times_{f}G$. From \cite[Prop.\
 5.11]{Neeb02Central-extensions-of-infinite-dimensional-Lie-groups} it
 follows that the connecting homomorphism $\pi_{2}(G)\to \pi_{1}(U(1))$
 in the long exact homotopy sequence of this fibration is (up to the
 choice of a sign) the identity on $\Z$. This implies that
 $\pi_{2}(\wh{G})=0$.

 Now consider the loop $\R\times_{F} G$ with
 $F:=F_{{\omega_{\sprod}},\beta}$ from the previous lemma. If there
 existed a topology on $\R\times G$ having $\R\times V$ as an open
 subset (for $V$ as in Remark
 \ref{rem:integrationInTheCaseOfADiscretePeriodGroup}) and turning $\mu$
 into a globally smooth map, then the exact sequence
 $  \Z \to \R\times_{F}G\to \wh{G}$, induced by $\varphi$ as in the
 previous lemma, would define a locally trivial principal bundle over
 $\hat{G}$. In fact, the maps
 \begin{equation*}
  g\cdot V\ni x\mapsto (0,g)\cdot(0,g^{-1}\cdot x)\in \R\times G
 \end{equation*}
 would provide smooth local sections (it is here that we use the global smoothness of $\mu$). 
 Since this bundle is in fact a
 covering and $\wh{G}$ is simply connected, this covering would be
 trivial and $\wh{G}$ would be diffeomorphic to $\Z\times G$. But
 $\pi_{2}(G)=\Z$, a contradiction to the existence of a globally smooth
 extension of the locally smooth multiplication.
\end{example}

\begin{remark}
 It is always the case that a group which restricts to a local
 (analytic) Lie group possesses a compatible (global) smooth
 (analytic) structure, at least if the group is generated by the local
 group (see Theorem
 \ref{thm:globalisation-of-smooth-structures-on-groups} for the precise
 statement).

 It is a well-known fact that an analogous statement does not hold for
 loops. In fact, each finite-dimensional almost smooth loop restricts to
 a smooth local loop on some neighbourhood of $e$. If there existed a
 compatible smooth structure on this loop, then the identity would yield
 a diffeomorphism between this smooth structure and the almost smooth
 structure. This is due to the fact that a morphism between almost
 smooth loops is smooth if and only if it is so on an open neighbourhood
 of $e$. So each almost smooth loop which is not a smooth loop yields
 a counterexample to the generalisation of the introductory statement of
 this remark to loops. See \cite[Sect.\
 1.3]{NagyStrambach02Loops-in-group-theory-and-Lie-theory} for details.
 In fact, the situation is even worse, there exist one-dimensional
 analytic loops which may not be extended to a global analytic loop
 \cite[Lines before Rem.\
 IX.6.8]{HofmannStrambach90Topological-and-analytic-loops}.
\end{remark}

\begin{tabsection}
 As the previous discussion shows, we are forced to consider as
 integrating objects of $\omega$ loop prolongations, which are
 compatible with the smooth structure only in an identity neighbourhood:

\end{tabsection}

\begin{definition}
 A loop prolongation $(A,\varphi\from L\to H)$ is called \emph{locally
 smooth} if $L$ is endowed with the structure of a smooth local loop. With this
 we mean that there exists a subset $W$ containing the identity, which
 is endowed with some manifold structure such that
 $\left.L\right|_{W}:=(\mu^{-1}(W)\cap (W\times W),W,\mu,0)$ is a local
 loop and all structure maps are smooth (cf.\ \cite[Sect.\
 1.1]{NagyStrambach02Loops-in-group-theory-and-Lie-theory}). In
 particular, the manifold structure is part of the data. If, in
 addition, the associator of $(A,\varphi\from L\to H)$ vanishes on some
 identity neighbourhood (in $W\times W\times W$), then we call the loop
 prolongation \emph{locally associative}.
\end{definition}

\begin{tabsection}
 Of course, local smoothness and local associativity make sense for
 loops in general, but we shall use this concepts only for loops that
 are parts of a loop prolongation.
\end{tabsection}

\begin{lemma}
 If $(A,\varphi\from L\to H)$ is a locally smooth loop prolongation and
 $L$ is generated by some open $V\se W$ with $V\cap A=\{e\}$, then $L/A$
 carries a Lie group structure such that the quotient map
 $q\from L\to L/A$ restricts to a local diffeomorphism on some open
 neighbourhood of $e$.
\end{lemma}

\begin{proof}
 Since $V\cap A=\{e\}$, $q$ is injective on $V$ and we
 use it to endow $q(V)\se L/A$ with a manifold structure. Clearly, the
 group multiplication and inversion are locally smooth on $A/L$ with
 respect to this smooth structure. Since $V$ generates $L$ we have that
 $q(V)$ generates $L/A$ and the assertion follows from Theorem
 \ref{thm:globalisation-of-smooth-structures-on-groups}.
\end{proof}

\begin{remark}
 If $(A,\varphi\from L\to H)$ is a locally smooth and locally
 associative loop prolongation, then $\left.L\right|_{V}$ is a local Lie
 group for some open identity neighbourhood $V\se W$. This then gives
 rise to a Lie algebra $\cL(L)$, which is independent of the
 choice of $V$.
\end{remark}

\begin{proposition}\label{prop:integrating_lie_algebras_to_loop_prolongations}
 If $\fg$ is a Mackey-complete locally exponential Lie algebra, then
 there exists a locally smooth and locally associative loop prolongation
 such that the associated Lie algebra is isomorphic to $\fg$.
\end{proposition}

\begin{tabsection}
 Since the upcoming sections are independent of this result, there is no
 harm in postponing the proof until the end of Section
 \ref{sect:lies_third_theorem}.
\end{tabsection}

\section{2-groups}

\begin{tabsection}
 We now introduce 2-groups in oder to overcome the discrepancy between
 the globally defined algebraic structure and the locally given
 smoothness in the next section (cf.\ Remark
 \ref{rem:2-group-are-better-than-loop-prolongations}).
\end{tabsection}

\begin{definition}\label{def:2-groups}
 A (unital) \emph{2-group} is a small groupoid $\cG$, together with a
 \emph{multiplication functor} $\otimes \from \cG \times \cG\to \cG$, an
 \emph{inversion functor} $\ol{^{~}}\from \cG\to \cG$ and an object
 $\nelt$ (frequently identified with its identity morphism
 $\id_{\nelt}$), together with natural isomorphisms
 \begin{equation*}
  \alpha_{g,h,k}\from (g\otimes h)\otimes k\to g\otimes(h\otimes k)
 \end{equation*}
 for $g,h,k$ objects of $\cG$, called \emph{associators}. We require
 $\nelt=\ol{\nelt}$, $g\otimes \nelt=g=\nelt\otimes g$ and
 $g\otimes \ol{g}=\nelt=\ol{g}\otimes g$ on objects and morphisms and we
 require that
 \begin{equation*}
  \alpha_{g,h,k\otimes l} \op{\circ} \alpha_{g\otimes h,k,l}=
  (\id_{g}\otimes \alpha_{h,k,l}) \op{\circ}\alpha_{g,h\otimes k,l} \op{\circ}
  (\alpha_{g,h,k}\otimes \id_{l})
 \end{equation*}
 for all $g,h,k,l$.
 The last requirement is frequently referred to as \emph{pentagon
 identity}. Moreover, we require that $\alpha_{g,h,k}$ is an identity if
 one of $g$, $h$ or $k$ is $\nelt$ and that
 $\alpha_{g,\ol{g},g}=\id_{g}$ and
 $\alpha_{\ol{g},g,\ol{g}}=\id_{\ol{g}}$.

 A \emph{morphism} of (unital) 2-groups is a functor
 $\cF\from \cG \to\cG'$, together with natural isomorphisms
 $\beta_{g,h}\from  \cF(g)\otimes'\cF(h)\to  \cF(g\otimes h)$,
 satisfying $\cF(\nelt)=\nelt'$, $\cF(\ol{g})=\ol{\cF(g)}$,
 $\beta_{ g,\nelt}=\id_{\cF(g)}=\beta_{\nelt,g}$ and
 $\beta_{g,\ol{g}}=\nelt'=\beta_{\ol{g},g}$ for all objects $g$ of
 $\cG$. Here, we require
 \begin{equation*}
  \cF(\alpha(g,h,k)) \op{\circ} \beta_{g\otimes h,k} 
  \op{\circ} (\beta_{g, h}\otimes' \id_{k}) =
  \beta_{g,h\otimes k} \op{\circ} (\id_{\cF(g)}\otimes '\beta_{h,k} )
  \op{\circ} \alpha'(\cF(g),\cF(h),\cF(k)) 
 \end{equation*}
 for all $g,h,k$.
 Finally, a \emph{2-morphism} between morphisms $\cF$ and $\cF'$
 consists of natural isomorphisms $\gamma_{g}\from \cF(g)\to \cF'{(g)}$
 such that
 \begin{equation*}
  \gamma_{g\otimes h} \op{\circ}  \beta_{g,h}=
  \beta'_{g,h}\op{\circ} (\gamma_{g}\otimes ' \gamma_{h}).
 \end{equation*}
 The resulting 2-category is denoted by $\cat{2-Grp}$.
\end{definition}

\begin{tabsection}
 We took the clumsy notation for the inversion functor to distinguish it
 explicitly from the functor $\cG\to\cG^{\op{op}}$ that maps each
 morphism to its inverse morphism. Non-unital 2-groups involve
 additional natural isomorphisms, replacing the identities
 $g\otimes \nelt=g=\nelt\otimes g$ and
 $g\otimes \ol{g}=\nelt=\ol{g}\otimes g$, which are themselves required
 to obey certain coherence conditions. However, in our constructions
 these isomorphisms shall always be identities which is why we excluded
 them from our definition. This is also why we drop the adjective
 ``unital'' in the sequel.
\end{tabsection}

\begin{remark}
 2-groups are special kinds of monoidal categories (cf.\
 \cite{BaezLauda04Higher-dimensional-algebra.-V.-2-groups}), just as
 groups are special kinds of monoids. However, observe that a group is a
 monoid with a special \emph{property} (existence of inverses), while a
 2-group is a monoidal category with an additional \emph{structure} (an
 inversion functor). Our 2-groups are also examples of coherent
 2-groups, as considered in
 \cite{BaezLauda04Higher-dimensional-algebra.-V.-2-groups}.
\end{remark}

\begin{example}\label{ex:strict-2-groups-from-crossed-modules}
 A particular important class of examples form the so called
 \emph{strict 2-groups}. They can be characterised to be 2-groups, for
 which all natural isomorphisms in Definition \ref{def:2-groups} are the
 identities.

 Besides this description, strict 2-groups can be described by crossed
 modules as follows (in fact, the 2-category of strict 2-groups is
 2-equivalent to the 2-category of crossed modules, cf.\
 \cite{Porst08Strict-2-Groups-are-Crossed-Modules},
 \cite{Forrester-Barker02Group-Objects-and-Internal-Categories} or
 \cite{Loday82Spaces-with-finitely-many-nontrivial-homotopy-groups}). A
 \emph{crossed module} is a morphism of groups $\tau \from H\to G$ (for
 $G$ now an arbitrary group), together with an automorphic action of $G$
 on $H$ such that
 \begin{align}
  \tau (g.h)&=g\cdot \tau (h)\cdot g^{-1}\label{eqn:crossed-module-equivariance}\\
  \tau (h).h'&= h\cdot h'\cdot h^{-1}.\label{eqn:crossed-module-peiffer}
 \end{align}
 Note that these two equations force $\ker(\tau)$ to be central in $H$
 and $\im(\tau)$ to be normal in $G$. From this one can build up a
 2-group $\cG$ as follows. The objects are $G$, the morphisms are
 $H\rtimes G$ and the structure maps are given by $s(h,g)=g$,
 $t(h,g)=\tau (h)\cdot g$, $\id_{g}=(e,g)$,
 $(h',\tau (h)\cdot g)\circ (h,g)= (h'h,g)$. The identity object is $e$
 and the multiplication and inversion functor are given by
 multiplication and inversion on the groups $G$ and $H\rtimes G$. Then
 \eqref{eqn:crossed-module-equivariance} and
 \eqref{eqn:crossed-module-peiffer} ensure that this defines a functor
 with the desired properties if we set the associator from
 Definition \ref{def:2-groups} to be the identity.
\end{example}

\begin{example}
 We obtain a slightly weaker version of the previous example if we are
 given instead of a crossed module a loop prolongation
 $(A,\varphi\from L\to G)$. From this we construct an (in general
 non-strict) 2-group as follows. The set of objects is $L$, the set
 of morphisms is $A\times L$ and the structure maps are given by
 $s(a,x)=x$, $t(a,x)=a\cdot x$ and $(a,b\cdot y){\circ}(b,y)=(a+b,y)$.
 Then we set
  $x\otimes y=x\cdot y$ and $(a,x)\otimes (b,y)=(a\cdot b,x\cdot y)$,
 which clearly defines a functor. Since the loop $L$ may fail to be
 associative, this only defines a 2-group if we introduce the
 associators $\alpha_{x,y,z}=(A(x,y,z),(x\cdot y)\cdot z)$.
 One readily checks that this defines a 2-group with the aid of the
 axioms of a loop prolongation from Remark \ref{rem:loop_prolongation}.
\end{example}

\begin{tabsection}
 We finish this section with a few constructions that will be important
 later on.
\end{tabsection}

\begin{remark}\label{rem:constructions_from_2_groups}
 Each 2-group comes along with a couple of natural groups associated to
 it.
 \begin{itemize}
  \item The set of isomorphism classes $\pi_{0}(\cG)$ of $\cG$. Since
        $\otimes$ is a functor, it induces a map
        $\pi_{0}(\cG)\times \pi_{0}(\cG)\to\pi_{0}(\cG)$. This clearly
        defines a group multiplication for isomorphic objects in $\cG$
        become equal in $\pi_{0}(\cG)$.
  \item The set $\cG_{\nelt}$ of morphisms in the full subcategory of
        $\cG$, generated by $\nelt$. On $\cG_{\nelt}$, we define a map
        \begin{equation*}
         \cG_{\nelt}\times \cG_{\nelt}\to \cG_{\nelt},\quad
         (g,h)\mapsto g\otimes h.
        \end{equation*}
        If we assume that $\alpha_{g,h,k}$ is an identity
		if one of $g$, $h$ or $k$ is isomorphic to $\nelt$,%
        \footnote{This should follow from coherence, but we were not
        able to find a reference for it. However, all 2-groups that we
        encounter in this article obey this condition} then this defines
        an associative multiplication on $\cG_{\nelt}$. Since
        $f\in\cG \Leftrightarrow \ol{f}\in \cG_{\nelt}$, this is in fact
        a group.
  \item The source and target fibres $s^{-1}(\nelt)$ and $t^{-1}(\nelt)$
        are a subgroup of $\cG_{\nelt}$.
  \item The endomorphisms
        $\pi_{1}(\cG):=\End(\nelt)=s^{-1}(\nelt)\cap t^{-1}(\nelt)$ of
        $\nelt$ form a subgroup of $\cG_{\nelt}$.
 \end{itemize}
\end{remark}
\section{Lie 2-groups}

\begin{tabsection}
 In this section we shall elaborate on our concept of smoothness on
 2-groups. Note the subtlety that we call our objects of main interest
 \emph{Lie} 2-groups, so we emphasise their Lie-theoretic
 interpretation. Other authors calls their corresponding objects
 \emph{smooth} 2-groups to put an emphasis on their properties as
 generalisations of smooth manifolds.

 Since our 2-groups are internal to sets (for they are assumed to be
 small categories), it seems to be natural to work internal to manifolds
 (i.e., require sets to be manifolds and maps to be smooth), but this
 turns out to be too restrictive. The perspective to Lie groups that we
 will follow for our notion of Lie 2-groups is that a Lie group is a
 group with a locally smooth group multiplication. We make this precise
 in the following theorem. Note that Lie 2-groups make sense for
 smooth spaces in a more general setting than just locally convex
 manifolds, but to stay clear and brief we will stick to the
 manifold case.
\end{tabsection}

\begin{theorem}\label{thm:globalisation-of-smooth-structures-on-groups}
 Let $G$ be a group, $U\se G$ be a subset containing $e$  and let $U$ be
 endowed with a manifold structure. Moreover, assume that there exists
 $V\se U$ open such that
 \begin{enumerate}
        \renewcommand{\labelenumi}{\theenumi}
        \renewcommand{\theenumi}{\roman{enumi})}
  \item \label{thm:globalisation:item1} $e\in V$, $V=V^{-1}$ and
        $V\cdot V\se U$,
  \item \label{thm:globalisation:item2} the maps
        $V\times V\ni (g,h)\mapsto gh \in U$ and
        $V\ni g\mapsto g^{-1}\in V$ are smooth,
  \item \label{thm:globalisation:item3} $V$ generates $G$ (as a monoid
        or, equivalently, as a group).
 \end{enumerate}
 Then there exists a unique Lie group structure on $G$ such that the
 inclusion $U\hookrightarrow G$ is a diffeomorphism on some open identity
 neighbourhood. In particular, $V\hookrightarrow G$ is a diffeomorphism onto its
 open image and any other choice of $V$ satisfying the above conditions
 gives the same smooth structure on $G$.
\end{theorem}

\begin{proof}
 The proof is standard, see for instance \cite[Prop.\
 III.1.9.18]{Bourbaki98Lie-groups-and-Lie-algebras.-Chapters-1--3}.
 However, we shall repeat the essential parts to illustrate the general
 idea.

 Let $W\se V$ be open such that $e\in W$, $W\cdot W\se V$ and
 $W=W^{-1}$. Then we transport the smooth structure from $W$ to $gW$ by
 left translation $\lambda_{g}\from W\to gW$ (i.e. we \emph{define}
 $\lambda_{g}$ to be a diffeomorphism). This is well-defined since for
 $gW\cap hW\neq \emptyset$ we have $h^{-1}g\in V$ so that the coordinate
 change
 \begin{equation*}
  \lambda_{g}^{-1}(gW\cap hW)\ni x \mapsto \lambda_{h^{-1}g}(x)\in \lambda_{h}^{-1}(gW\cap hW)
 \end{equation*}
 is smooth by \ref{thm:globalisation:item2}. In particular, $V\se G$ is
 open and $V\hookrightarrow G$ is  a diffeomorphism onto its image.

 To verify that the group multiplication is smooth, we first observe
 that for each $h\in G$ there exists $W_{h}\se V$ open with $e\in W_{h}$
 such that $h^{-1}W_{h}h\se V$ and $x\mapsto h^{-1}xh$ is smooth. In
 fact, the set of all $h\in G$ such that $W_{h}$ exists forms a sub
 monoid containing $V$, which equals $G$ by
 \ref{thm:globalisation:item3}. Thus,
 \begin{equation*}
  gW_{h}\times hW\ni (x,y)\mapsto xy=\lambda_{gh}\big((h^{-1}\cdot \lambda_{g}^{-1}(x)\cdot h)
  \cdot \lambda_{h}^{-1}(y)\big)\in ghV
 \end{equation*}
 is smooth. A similar argument shows that inversion is also smooth.

 If $G$ is endowed with a Lie group structure such that
 $U\hookrightarrow G$ restricts to a diffeomorphism on $V'$, then the
 restriction of $\id_{G}$ is smooth on $V'\cap V$. Since a
 homomorphism between Lie groups is smooth if and only if it so on an
 identity neighbourhood, this shows that $\id_{G}$ is in fact a
 diffeomorphism. This applies in particular to a possibly different choice of $V$.
\end{proof}

\begin{tabsection}
 Note that the previous theorem tells us that the group structure
 determines the global topology, as soon as the local topology is fixed.
 It also says that a Lie group may equally well be defined as a group
 $G$, together with the smooth structure on $U$ such that $V\se U$ open
 with the corresponding properties exist. This is a very familiar
 pattern in Lie theory and we shall take this perspective when defining
 {Lie} 2-groups below. We have already seen that a statement
 corresponding to the preceding theorem is not valid for loops and we
 shall see below that non-strict 2-groups yet have a different
 behaviour.

 The following example illustrates an important application of the
 preceding theorem.
\end{tabsection}

\begin{example}\label{ex:cocycle-for-the-universal-covering}
 Let $G$ be an arbitrary connected Lie group and let
 $f\from G\times G\to Z$ be a 2-cocycle, which is smooth on some
 identity neighbourhood $U \times U$. This gives a group $Z\times_{f}G$
 as in \eqref{eqn:multiplication_on_central_extension}. Taking $V\se U$
 open with $e\in V$, $V^{-1}=V$ and $V\cdot V\se U$ shows that this
 multiplication is smooth when restricted to $(V\times G)^{2}$ (a
 similar argument works for the inversion, since
 $(a,g)^{-1}=(-a-f(g,g^{-1}),g^{-1})$). Thus Theorem
 \ref{thm:globalisation-of-smooth-structures-on-groups} yields a Lie
 group structure on $Z\times_{f}G$. This turns $Z\to Z\times_{f}G\to G$
 into a central extension of Lie groups, possessing
 $U\ni x\mapsto (0,x)\in Z\times_{f}G$ as smooth local section . This
 construction applies in particular to the 2-cocycle
 $\cq{U(1)}\op{\circ}F_{\omega_{\sprod},\beta}$ from Example
 \ref{ex:Kac-Moody_group_cocycle}.

 Another important application of this construction is the following.
 Let $PG$ be the space of continuous pointed paths in $G$ and $G\to PG$,
 $g\mapsto \alpha_{g}$ be a section of the map that evaluates in $1$
 (the compact-open topology defines in fact a Lie group topology on $PG$
 with Lie algebra $P\fg$). Moreover, assume that $\alpha$ is smooth on
 some identity neighbourhood. Then
 \begin{equation*}
  \Theta_{\alpha} (g,h)= [\alpha_{g}+g.\alpha_{h}-\alpha_{gh}]\in H_{1}(G)\cong \pi _{1}(G)
 \end{equation*}
 is a group cocycle for the universal covering group
 $\wt{G}\cong PG/(\Omega G)_{0}$. In fact, reconstructing cocycles from
 the central extension $\pi_{1}(G)\to \wt{G}\to G$ as in \cite[Prop.\
 4.2]{Neeb02Central-extensions-of-infinite-dimensional-Lie-groups} shows
 that $\Theta_{\alpha}$ is equivalent to each of those cocycles.
\end{example}

We now turn to the development of our notion of Lie 2-group. At first, 
we need the concept of a smooth 2-space.

\begin{definition}\label{def:2-spaces}
 A \emph{smooth 2-space} is a (possibly infinite-dimensional) Lie
 groupoid. This means that it is a groupoid $\cC$ such that $\cC_{0}$
 and $\cC_{1}$ are endowed with smooth manifold structures, source and
 target maps are smooth surjective submersions\footnote{Surjective
 submersion in the strong sense that it is a projection in local
 coordinates. This ensures in particular that the space of composable
 morphisms $\cC_{1}\times_{\cC_{0}}\cC_{1}$ is a smooth manifold (cf.\
 \cite[App.\
 A]{NikolausSachseWockel11A-Smooth-Model-for-the-String-Group}).} and
 the other structure maps are smooth.

 A \emph{smooth functor} between smooth 2-spaces is a functor whose
 respective maps on objects and morphisms are smooth. A \emph{smooth
 natural transformation} between smooth functors is a natural
 transformation such that the corresponding map from objects to
 morphisms is smooth. The resulting 2-category is denoted by
 $\cat{2-Man}$. Two smooth 2-spaces $\cC$ and $\cD$ are  called
 \emph{isomorphic} if there exist smooth functors $F\from \cC\to\cD$ and
 $G\from\cD\to\cC$ such that $F \op{\circ}G=\id_{\cD}$ and
 $G \op{\circ} F=\id_{\cC}$ (on the nose).
\end{definition}

\begin{tabsection}
 The correct notion of equivalence of smooth 2-spaces is \emph{Morita
 equivalence}, a more involved notion than the naive one. We shall not
 need this notion in this article. The previous definition takes Lie
 groupoids as internal categories in the category of locally convex
 manifolds. For more general purposes as we are aiming for here this
 definition is insufficient. The category of locally convex manifolds
 has bad categorical properties: it lacks pull-backs, quotients
 and internal homs, also when restricting to finite-dimensional ones.
 This can be remedied by introducing smooth 2-spaces as categories
 internal to smooth spaces (also called Chen- or diffeological spaces),
 for which we refer to
 \cite{BaezHoffnung08Convenient-Categories-of-Smooth-Spaces}.

 The following Proposition is the equivalent statement to the previous
 theorem for \emph{strict} 2-groups. In order to state it we fist have
 to introduce the following notation.
\end{tabsection}

\begin{remark}
 Let $\cC$ be a small monoidal category (e.g., a 2-group) and
 $\cV\se \cC$ be a subcategory. Then the \emph{monoidal subcategory
 generated by $\cV$} is the smallest monoidal subcategory containing
 $\cV$. Since intersections of monoidal subcategories are in turn
 monoidal subcategories, there is a unique smallest monoidal
 subcategory, which we denote by $\langle\cV\rangle$.
\end{remark}

\begin{proposition}\label{prop:strict-lie-2-groups}
 Let $\cG$ be a strict 2-group, $\cU$ be a full subcategory containing
 $\nelt$ and let $\cU$ be endowed with the structure of a smooth
 2-space. Moreover, assume that there exists a full subcategory
 $\cV\se\cU$ such that $\cV_{0}$ is open in $\cU_{0}$,
 \begin{enumerate}
        \renewcommand{\labelenumi}{\theenumi}
        \renewcommand{\theenumi}{\roman{enumi})}
  \item \label{prop:strict-lie-2-groups:item1} $\nelt\in\cV$,
        $\ol{\cV}=\cV$ and $\cV\otimes\cV\se\cU$,
  \item \label{prop:strict-lie-2-groups:item2} the functors
        $\left.\otimes\right|_{\cV\times\cV}\from\cV\times\cV\to\cU$ and
        $\left.\ol{^{~}}\right|_{\cV}\from\cV\to\cV$ are smooth,
  \item \label{prop:strict-lie-2-groups:item3} $\langle\cV\rangle=\cG$.
 \end{enumerate}
 Then there exists on $\cG$ the structure of a smooth 2-space such that
 $\ol{^{~}}$ and $\otimes$ are smooth functors and the inclusion
 $\cU\hookrightarrow \cG$ restricts to an isomorphism on some full
 subcategory $\cV'$ with $\cV'_{0}\se\cU$ open. Moreover, the smooth
 structure on $\cG$ is unique with respect to these properties.


 
%
\end{proposition}

\begin{tabsection}
 The following proof relies heavily on the fact that strict 2-groups
 are actually category objects internal to the category of groups, i.e.,
 spaces of objects, morphisms and composable morphisms are groups and
 all structure maps are group homomorphisms (cf.\
 \cite{Porst08Strict-2-Groups-are-Crossed-Modules},
 \cite{BaezLauda04Higher-dimensional-algebra.-V.-2-groups},
 \cite{Forrester-Barker02Group-Objects-and-Internal-Categories}).
\end{tabsection}

\begin{proof}
 It is clear from the assumptions that $\cU_{0}\se\cG_{0}$ is a subset
 containing the identity which is endowed with a manifold structure and
 $\cV_{0}\se\cU_{0}$ is an open subset satisfying the assumptions from
 Theorem \ref{thm:globalisation-of-smooth-structures-on-groups}. This
 yields a smooth structure on $\cG_{0}$. On morphisms, we have the
 smooth manifold $\cU_{1}=s^{-1}(\cU_{0})\cap t^{-1}(\cU_{0})$
 containing $\id_{\nelt}$ and the open subset
 $\cV_{1}=s^{-1}(\cV_{0})\cap t^{-1}(\cV_{0})$. Now $\cV_{1}$ generates
 $\cG_{1}$ by assumption, so Theorem
 \ref{thm:globalisation-of-smooth-structures-on-groups} also yields a
 smooth structure on $\cG_{1}$. Moreover, $s$ and $t$ are submersions
 since they are so on $\cU_{1}$ and the smoothness of $\ol{^{~}}$ and
 $\otimes$ is part of the conclusion of Theorem 
 \ref{thm:globalisation-of-smooth-structures-on-groups}.
 The uniqueness assertion also follows immediately from the one in 
 Theorem \ref{thm:globalisation-of-smooth-structures-on-groups}.
\end{proof}

\begin{tabsection}
 It might look quite promising to expect a similar construction of
 globally smooth 2-group structures from locally ones also in the case
 of non-strict 2-groups, but this expectation is too optimistic. In fact, the
 following lemmata show that the topology of $\cG_{1}$ splits as a product
 in this case into the part that comes from the identities and the arrow
 part.
\end{tabsection}

\begin{lemma}
 Let $\cG$ be a 2-group which is also a smooth 2-space such that the
 functors $\ol{^{~}}$, $\otimes$  and the associator are smooth. Then
 $s^{-1}(\nelt)$ is a submanifold of $\cG_{1}$ and in particular a Lie
 group. Moreover,
 \begin{equation*}
  s^{-1}(\nelt)\times \cG_{1}\to \cG_{1},\quad (a,f)\mapsto a\otimes f
 \end{equation*}
 defines a smooth action. Moreover, this action is free,
 $\cG_{1}/s^{-1}(\nelt)\cong \cG_{0}$ as manifolds and $\cG_{1}$ is
 a trivial smooth principal $s^{-1}(\nelt)$-bundle.
\end{lemma}

\begin{proof}
 Since inverse images of points under submersions are submanifolds, $s^{-1}(\nelt)$ is a Lie group.
 That the action is free follows from
 \begin{equation*}
  a\otimes f=b\otimes f\Rightarrow a\otimes \underbrace{(f\otimes 
  f^{-1})}_{=\nelt}=b\otimes\underbrace{(f\otimes f^{-1})}_{=\nelt}\Rightarrow a=b.
 \end{equation*}
 The source map $\cG_{1}\to \cG_{0}$ is $s^{-1}(\nelt)$-invariant and thus induces a smooth
 map $\cG_{1}/s^{-1}(\nelt)\to\cG_{0}$. The identity map
 $\cG_{0}\to\cG_{1}$ provides a smooth global section, proving the claim.
\end{proof}

\begin{lemma}\label{lem:triviality-of-associator-in-the-etale-case}
 Let $\cG$ be a 2-group which is also a smooth 2-space such that the
 functors $\ol{^{~}}$, $\otimes$  and the associator are smooth. If
 $s^{-1}(\nelt)$ is discrete in the induced topology then the arrow part
 \begin{equation*}
  (\cG_{0})^{3}\ni(g,h,k)\mapsto \alpha(g,h,k)\otimes \id_{\ol{(g\otimes h)\otimes k}}\in s^{-1}(\nelt)
  \se \cG_{1}
 \end{equation*}
 of $\alpha$ is locally constant.
\end{lemma}

\begin{proof}
 This is due to the fact that smooth maps between locally convex manifolds
 are in particular continuous.
\end{proof}

\begin{tabsection}
 The importance of the previous lemma is that we are forced to work with
 2-groups with $s^{-1}(\nelt)$ discrete if we want a reasonable
 interpretation of a Lie 2-group integrating an ordinary Lie algebra
 (cf.\ Section \ref{sect:generalised-central-extensions}). Thus it
 illustrates the limitation on building 2-groups with too many
 smoothness conditions. However, locally smoothness of the group
 multiplication is essential for passing from Lie 2-groups to Lie
 2-algebras. In view of Theorem
 \ref{thm:globalisation-of-smooth-structures-on-groups} and Proposition
 \ref{prop:strict-lie-2-groups}, the following definition seems to be
 natural.
\end{tabsection}

\begin{definition}\label{def:Lie-2-group}
 A \emph{Lie 2-group} is a tuple $(\cG,\mc{U})$ such that $\cG$ is a
 2-group and $\mc{U}$ is a full subcategory containing $\nelt$, which is
 endowed with the structure of a smooth 2-space. Moreover, there has to
 exist a full subcategory $\cV\se \cU$ with $\cV_{0}\se \cU_{0}$ open
 such that
 \begin{enumerate}
        \renewcommand{\labelenumi}{\theenumi}
        \renewcommand{\theenumi}{\roman{enumi})}
  \item $\nelt\in \cV$, $\ol{\cV}=\cV$ and $\cV\otimes\cV\se\cU$,
  \item  the functors
        $\left.\otimes\right|_{\cV\times\cV}\from\cV\times\cV\to\cU$ and
        $\left.\ol{^{~}}\right|_{\cV}\from\cV\to\cV$ are smooth,
  \item $\langle\cV\rangle=\cG$.
 \end{enumerate}
 When working with Lie 2-groups we will sometimes not mention $\cU$
 explicitly if it is understood. A \emph{morphism} of Lie 2-groups is a
 morphism of the underlying 2-groups such that the constituting functors
 and natural transformations restrict (and co-restrict) to smooth
 functors and natural transformations on some of the subcategories $\cV$
 from above. Likewise, \emph{2-morphisms} between morphisms of Lie 2-groups are
 defined.
\end{definition}

\begin{tabsection}
 For the case of a non-strict 2-group our notion of a Lie 2-group does
 \emph{not} fit with the notion used in \cite[Def.\
 27]{BaezLauda04Higher-dimensional-algebra.-V.-2-groups}, where the
 functors and natural transformation defining the 2-group structure are
 required to be globally smooth. For the reasons explained above we find
 our notion more natural in the non-strict case (see also \cite[Thm.\
 59]{BaezLauda04Higher-dimensional-algebra.-V.-2-groups}). The
 observation that the concept introduced in
 \cite{BaezLauda04Higher-dimensional-algebra.-V.-2-groups} is sometimes
 inadequate has also been made by Henriques in \cite[Sect.\
 9]{Henriques08Integrating-Lsb-infty-algebras} (the latter notion of
 smooth 2-groups has also been used in
 \cite{Schommer-Pries10Central-Extensions-of-Smooth-2-Groups-and-a-Finite-Dimensional-String-2-Group}). However,
 the previous definition covers strict Lie 2-groups by Proposition
 \ref{prop:strict-lie-2-groups} (cf.\
 \cite{BaezCrans04Higher-dimensional-algebra.-VI.-Lie-2-algebras},
 \cite{BaezLauda04Higher-dimensional-algebra.-V.-2-groups},
 \cite{Wockel09Principal-2-bundles-and-their-gauge-2-groups}). Moreover,
 it leads to a locally smooth group structure on the group
 $\pi_{0}(\cG)$ (in the appropriate category where $\pi_{0}(\cG)$ is a
 smooth space) and thus a globally smooth group structure thereon by
 Theorem \ref{thm:globalisation-of-smooth-structures-on-groups}. We thus
 may interpret our Lie 2-groups as a categorified version of a Lie
 group, much like Lie groupoids are categorified manifolds.
\end{tabsection}

\begin{remark}\label{rem:2-group-are-better-than-loop-prolongations}
 It is the fact that 2-groups form a 2-category  which makes them more
 interesting then loop prolongations. Thus 2-groups allow for the notion
 of equivalence, which is weaker than isomorphism. Moreover, the
 category of smooth 2-spaces also allows for a weaker notion of
 morphisms, the so-called Hilsum-Skandalis morphisms (also called spans
 or weak morphisms or bibundles) between Lie groupoids. Making use of
 this concept, one can show that certain Lie 2-groups in the above sense
 are in fact equivalent to weak group objects in the weak 2-category of
 Lie groupoids with morphisms the aforementioned Hilsum-Skandalis
 morphisms (the latter are called \emph{stacky Lie groups} in
 \cite{Blohmann08Stacky-Lie-groups}). This applies in particular to the
 2-groups that we will construct in Theorem
 \ref{thm:integrating-central-extensions} (cf.\ Remark
 \ref{rem:prospect:differential_geometry_of_generalised_central_extensions}
 and
 \cite{ZhuWockel10Integrating-central-extensions-of-Lie-algebras-via-group-stacks})
 and to the one from the following example (cf.\
 \cite{Schommer-Pries10Central-Extensions-of-Smooth-2-Groups-and-a-Finite-Dimensional-String-2-Group}).
\end{remark}

Although it does not play a r{\^o}le in the main theme of the paper, we
present the following example for it illustrates the use and simplicity of
our concept of Lie 2-groups.

\begin{example}\label{ex:string-2-group}
 Let $G$ be compact, simple and simply connected. Then
 $\langle[\mathinner{\cdot},\mathinner{\cdot}],\mathinner{\cdot}\rangle$
 is a Lie-algebra 3-cocycle on $\fg$, where
 $\langle\mathinner{\cdot},\mathinner{\cdot}\rangle$ denotes the Killing
 form of $\fg$. Under this assumptions the left-invariant extension
 $\langle[\mathinner{\cdot},\mathinner{\cdot}],\mathinner{\cdot}\rangle^{l}$
 is a generator of $H^{3}_{\op{dR}}(G,\Z)\cong \Z$. Consequently, the
 corresponding period homomorphism
 \begin{equation*}
  \per_{\langle[\mathinner{\cdot},\mathinner{\cdot}],\mathinner{\cdot}\rangle}\from \pi_{3}(G)\to \R,\quad
  [\sigma]\mapsto \int_{\sigma}\langle[\mathinner{\cdot},\mathinner{\cdot}],\mathinner{\cdot}\rangle^{l}
 \end{equation*}
 (cf.\ \cite[Def.\
 V.2.12]{Neeb06Towards-a-Lie-theory-of-locally-convex-groups}) has image
 $\Z$. Now the maps
 \begin{equation*}
  \alpha\from G\to C^{\infty}(\Delta^{(1)},G)\quad\text{ and }\quad\beta\from G^{2}\to C^{\infty}(\Delta^{(2)},G)
 \end{equation*}
 from Lemma \ref{lem:homology-cocycle-actual-choice-for-a-chart} are
 accompanied by an additional map
 $\gamma\from G^{3}\to C^{\infty}(\Delta^{(3)},G)$ satisfying
 \begin{equation}\label{eqn:coherent-choice-of-triangles-1}
  \partial \gamma_{g,h,k}=
  g.\beta_{h,k}-\beta_{gh,k}+\beta_{g,hk}-\beta_{g,h}
 \end{equation}
 for the above assumptions on $G$ imply that it is 2-connected.
 Moreover, one can choose $\gamma_{g,h,k}$ to depend smoothly on
 $(g,h,k)$ in some neighbourhood of $(e,e,e)$, similar to $\alpha$ and $\beta$ in equations
 \eqref{eqn:cocycle-from-chart-1} and \eqref{eqn:cocycle-from-chart-2}.
 Then we set
 \begin{equation*}
  \varphi_{\gamma}\from G^{3}\to U(1)=\R/\Z,\quad
  (g,h,k)\mapsto\exp\left( \int_{\gamma_{g,h,k}}\langle[\mathinner{\cdot},\mathinner{\cdot}],\mathinner{\cdot}\rangle^{l}\right)
 \end{equation*}
 where $\exp\from \R\to \R/\Z$ is the canonical quotient map. This
 defines a 3-cocycle since
 \begin{equation*}
  \dd_{\op{gp}}(\varphi_{\gamma})(g,h,k,l)=\int_{(\dd_{\op{gp}}\gamma) (g,h,k,l)}\langle[\mathinner{\cdot},\mathinner{\cdot}],\mathinner{\cdot}\rangle^{l}\in\Z,
 \end{equation*}
 which in turn follows from
 $(\dd_{\op{gp}}\gamma) (g,h,k,l)\in Z_{3}(G)$, similar to Remark
 \ref{rem:integrationInTheCaseOfADiscretePeriodGroup}.

 This is in fact a locally smooth 3-cocycle and by \cite[Thm.\
 V.2.6]{Neeb06Towards-a-Lie-theory-of-locally-convex-groups} we may
 differentiate this cocycle to get back the Lie algebra 3-cocycle
 $\langle[\mathinner{\cdot},\mathinner{\cdot}],\mathinner{\cdot}\rangle$.
 Similar to the argument from Remark \ref{rem:dependence-on-choices} we
 see that the cohomology class of $\varphi_{\gamma}$ does not depend on
 the choice of $\gamma$, as long as
 \eqref{eqn:coherent-choice-of-triangles-1} is fulfilled (that is why we
 drop the subscript from now on). From this cocycle we get a 2-group
 $\cG_{G}$ by setting $(\cG_{G})_{0}$ to be $G$ and $(\cG_{G})_{1}$ to
 be $U(1)\times G$ with source and target map equal to the projection to
 $G$ and composition of morphism induced by the group structure on
 $U(1)$. The monoidal structure is given by the group multiplication in
 $G$ (on objects) and in $U(1)\times G$ (on morphisms) and the
 associator is given by $\alpha_{g,h,k}=(\varphi(g,h,k),ghk)$.

 Now $\varphi$ is smooth on $U\times U\times U$ for $U\se G$ some open
 identity neighbourhood. Choosing some $V\se U$ open with $e\in V$,
 $V=V^{-1}$ and $V^{2}\se U$ one directly checks that all requirements
 from Definition \ref{def:Lie-2-group} are satisfied. This turns
 $\cG_{G}$ into a Lie 2-group.

 The natural generalisation of the differentiation process described in
 the next section enables one to differentiate $\cG_{G}$ to a Lie
 2-algebra. Since the differentiation of $\varphi$ is the 3-cocycle
 $\langle[\mathinner{\cdot},\mathinner{\cdot}],\mathinner{\cdot}\rangle$,
 this Lie 2-algebra is the non-strict Lie 2-algebra determined by the
 3-cocycle
 $\langle[\mathinner{\cdot},\mathinner{\cdot}],\mathinner{\cdot}\rangle$
 (cf.\
 \cite{BaezCrans04Higher-dimensional-algebra.-VI.-Lie-2-algebras}). This
 is (one model for) the  string Lie 2-algebra, and the Lie 2-group
 $\cG_{G}$ would thus be another model for the string 2-group (cf.\
 \cite{Stolz96A-conjecture-concerning-positive-Ricci-curvature-and-the-Witten-genus},
 \cite{StolzTeichner04What-is-an-elliptic-object},
 \cite{BaezCransStevensonSchreiber07From-loop-groups-to-2-groups},
 \cite{Henriques08Integrating-Lsb-infty-algebras},
 \cite{Schommer-Pries10Central-Extensions-of-Smooth-2-Groups-and-a-Finite-Dimensional-String-2-Group} or
 \cite{NikolausSachseWockel11A-Smooth-Model-for-the-String-Group}). There is
 certainly much more to say about this Lie 2-group (cf.\ Remark
 \ref{rem:prospect:differential_geometry_of_generalised_central_extensions}),
 but this lies beyond the scope of the present paper.
\end{example}

\section{Categorified central extensions and \'etale Lie 2-groups}
\label{sect:generalised-central-extensions}

\begin{tabsection}
 In this section we define central extensions of (Lie) 2-groups and show
 how they arise from generalised cocycles (cf.\ Remark
 \ref{rem:from-generalised-cocycles-to-2-groups}), for the more general
 setting see \cite{Breen92Theorie-de-Schreier-superieure}. In
 particular, we seek for an interpretation of the integrating cocycle
 from Theorem \ref{thm:integration-of-cocycles} in terms of central
 extensions.

 In the first part of the section we shall describe the route from
 generalised cocycles to generalised central extensions. The second part
 elaborates on the basic notions of Lie theory for Lie 2-groups and
 central extensions. Our perspective will be that central extensions are
 described by group cohomology, see  \cite[Sect.\
 IV.3]{MacLane63Homology} for ordinary groups,
 \cite{AgoreMilitaru08Crossed-product-of-groups.-Applications} for
 generalisations and
 \cite{Neeb02Central-extensions-of-infinite-dimensional-Lie-groups} for
 the specialisation to topological and Lie groups. The approach to
 Schreier-like invariants for extensions of groupoids in
 \cite{BlancoBullejosFaro05Categorical-non-abelian-cohomology-and-the-Schreier-theory-of-groupoids}
 does not fit into our situation, for our sequences of groupoids shall
 not be bijectively on objects.
\end{tabsection}

\begin{remark}
 In order to match the following definition with the situation of
 extensions of groups recall that a short exact sequence
 $A\xrightarrow{i} B\xrightarrow{j} C$ is a sequence of order two such
 that the diagram
 \begin{equation}\label{eqn:pullback-ordinary}
  \vcenter{\xymatrix{A\ar[r]^{i}\ar[d] & B\ar[d]^{j}\\ {*} \ar[r]&C}}
 \end{equation}
 is at the same time a pullback ($i$ injective and $\im(i)\se\ker(j)$)
 and a pushout ($j$ surjective and $\ker(j)\se \im(i)$).

 In the case that we are working with a strict 2-category we have to replace
 the (ordinary) pullback by a 2-pullback and likewise replace a pushout
 by a 2-pushout. If $X\xrightarrow{f} Z$ and $Y\xrightarrow{g} Z$ are
 morphisms in a 2-category, then a \emph{2-pullback} consists of an object,
 denoted $X\times _{Z}Y$, 1-morphisms $X\times_{Z}Y\xrightarrow{p} X$
 and $X\times_{Z}Y\xrightarrow{q} Y$ and a 2-isomorphism
 $\varphi\from f \op{\circ} p\Rightarrow g \op{\circ} q$ such that the
 diagram
 \begin{equation*}
  \xymatrix{X\times_{Z}Y\ar[r]^(.6){p}\ar[d]^{q}&X\ar[d]^{f}\\
  Y\ar[r]^{g}&Z\ar@{=>}_{\varphi} "1,2"+/dl 1.5em/ ; "2,1"+/ur 1.5em/}
 \end{equation*}
 has the following universal property: For any object $W$ that comes
 equipped with morphisms $W\xrightarrow{m} X$ and $W\xrightarrow{n} Y$
 and a 2-isomorphism
 $\psi\from f \op{\circ} m \Rightarrow g \op{\circ} n$ there exists a
 morphism $s\from W\to X\times_{Z}Y$ and 2-isomorphisms
 $\xi\from m\Rightarrow p \op{\circ}s$ and
 $\zeta\from q \op{\circ} s\Rightarrow n $ such that
 \begin{equation*}
  \vcenter{\xymatrix{W\ar@/^.0pt/[dr]|{\blabel{s}}="q"\ar@/^{1em}/[drr]^{m}_{~}="s"\ar@/_{1em}/[ddr]_{n}^{~}="t" \\&X\times_{Z}Y\ar[r]^(.6){p}\ar[d]^{q}&X\ar[d]^{f}\\&
  Y\ar[r]^{g}&Z\ar@{=>}_{\varphi} "2,3"+/dl 1.5em/ ; "3,2"+/ur 1.5em/
  \ar@{=>}^(.7){\zeta} "2,2" ; "t" +/dr 1.5em/
  \ar@{=>}^(.3){\xi}   "s" +/dr 1em/ ; "2,2"  
  }}=                                         
  \vcenter{\xymatrix{W\ar@/^{1em}/[drr]^{m}_{~}="s"\ar@/_{1em}/[ddr]_{n}^{~}="t" \\&&X\ar[d]^{f}\\&
  Y\ar[r]^{g}&Z
  \ar@{=>}_{\psi} "2,3"+/dl 1.5em/ ; "3,2"+/ur 1.5em/
  }}.
 \end{equation*}
 Moreover, given another morphism $W\xrightarrow{s'}X\times_{Z}Y$ and
 2-isomorphisms $\alpha\from p \op{\circ}s\Rightarrow p \op{\circ} s' $
 and $\beta\from q \op{\circ} s\Rightarrow q \op{\circ} s'$ such that
 \begin{equation*}
  \vcenter{\xymatrix{W\ar@/^.0pt/[r]^{{s}}="q"\ar[d]_{s'}&X\times_{Z}Y\ar[r]^(.6){p}\ar[d]^{q}&X\ar[d]^{f}\\X\times_{Z}Y\ar[r]^{q}&
  Y\ar[r]^{g}&Z\ar@{=>}_{\varphi} "1,3"+/dl 1.5em/; "2,2"+/ur 1.5em/
  \ar@{=>}_{\beta}   "1,2"+/va(215) 1.5em/; "2,1" +/va(35) 1.5em/
  }}
  =
  \vcenter{\xymatrix@=1.5em{W\ar@/^.0pt/[d]_{{s'}}="q"\ar[r]^{s}& X\times_{Z}Y\ar[d]^{p} \\X\times_{Z}Y\ar[r]^(.6){p}\ar[d]_{q}&X\ar[d]^{f}\\
  Y\ar[r]^{g}&Z\ar@{=>}_{\varphi} "2,2"+/va(215) 1.5em/ ; "3,1"+/va(35) 1.5em/
  \ar@{=>}_{\alpha} "1,2"+/va(215) 1.5em/; "2,1"+/va(35) 1.5em/
  }},
 \end{equation*}
 there has to be a unique 2-isomorphism $\chi\from s\Rightarrow s'$ such
 that
 \begin{equation*}
  \vcenter{\xymatrix{W \ar@/^{1em}/[r]^{s}_(.4){~}="s"\ar@/_{1em}/[r]_{s'}^(.4){~}="t"&X\times_{Z}Y\ar[r]^(.6){p}&X
  \ar@{=>}^{\chi} "s"; "t"
  }}
  \!=\!
  \vcenter{\xymatrix@=1.2em{W \ar@/^{1em}/[rr]^{p \op{\circ} s}_{~}="s"\ar@/_{1em}/[rr]_{p \op{\circ} s'}^{~}="t"&&X
  \ar@{=>}^{\alpha} "s"; "t"
  }}\tx{ and }
  \vcenter{\xymatrix{W \ar@/^{1em}/[r]^{s}_(.4){~}="s"\ar@/_{1em}/[r]_{s'}^(.4){~}="t"&X\times_{Z}Y\ar[r]^(.6){q}&X
  \ar@{=>}^{\chi} "s"; "t"
  }}
  \!=\!
  \vcenter{\xymatrix@=1.2em{W \ar@/^{1em}/[rr]^{q \op{\circ} s}_{~}="s"\ar@/_{1em}/[rr]_{q \op{\circ} s'}^{~}="t"&&X
  \ar@{=>}^{\beta} "s"; "t"
  }}.
 \end{equation*}
 Along the same lines, one defines \emph{2-pushouts}.
\end{remark}

\begin{definition}
 (cf.\ \cite[Def.\
 66]{Schommer-Pries10Central-Extensions-of-Smooth-2-Groups-and-a-Finite-Dimensional-String-2-Group}) If
 $\tau\from A\to Z$ is a morphism of abelian groups (viewed as a crossed
 module for the trivial action of $Z$ on $A$), then we denote the
 associated 2-group from Example
 \ref{ex:strict-2-groups-from-crossed-modules} by $\cZ_{\tau}$, which we
 also call a  \emph{strict abelian 2-group}.

 For an arbitrary group $G$ denote by $\ul{G}$ the 2-group with objects
 $g\in G$, only identity morphisms and the 2-group structure induced by
 multiplication in $G$. Then we define an \emph{abelian extension} of
 $\ul{G}$ by $\cZ_{\tau}$ to be a sequence of 2-groups
 $\cZ_{\tau}\xrightarrow{i}\wh{\cG}\xrightarrow{q}\ul{G}$ such that
 $q \op{\circ} i$ is the constant functor $\nelt$ and that the diagram
 \begin{equation}\label{eqn:pullback-categorical}
  \vcenter{\xymatrix{\cZ_{\tau}\ar[r]^{i}\ar[d] & \wh{\cG}\ar[d]^{q}\\ {*} \ar[r]&\ul{G}
  \ar@{=>}^{\id_{\nelt}} "1,2"+/dl 1.5em/ ; "2,1"+/ur 1.5em/}}
 \end{equation}
 is a 2-pullback and a 2-pushout in the 2-category \cat{2-Grp}. Such an
 extension is called \emph{central} if the two functors
 \begin{equation}\label{eqn:functors_defining_action_on_the_kernel}
  \cZ_{\tau}\times\wh{\cG}\to\wh{\cG},
  \quad (z,g)\mapsto g\otimes(i(z)\otimes \ol{g})
  \quad\tx{ and }\quad
  \cZ_{\tau}\times\wh{\cG}\to\wh{\cG},\quad (z,g)\mapsto i(z)
 \end{equation}
 are naturally isomorphic.
\end{definition}

\begin{tabsection}
 Note that the fact that $\ul{G}$ has only identity morphisms enforces
 us to put $\id_{\nelt}$ into the 2-cell of the above diagram.
\end{tabsection}

\begin{remark}\label{rem:from-generalised-cocycles-to-2-groups}
 Let $G$ be a discrete group and $A,Z$ be discrete abelian groups. For
 $(F,\Theta)$ a generalised group cocycle with coefficients in
 $\tau\from A\to Z$, given by $\Theta\in C^{3}(G,A)$ and
 $F\in C^{2}(G,A)$ satisfying
 \eqref{eqn:cocycle-identity-for-generalised-cocycle-1} and
 \eqref{eqn:cocycle-identity-for-generalised-cocycle-2}, the following
 assignment defines a 2-group $\wh\cG_{(F,\Theta)}$. The category
 $\wh\cG_{(F,\Theta)}$ is given by
 \begin{alignat*}{3}
  \Obj(\wh\cG_{(F,\Theta)})&= Z\times G, &\quad  &
  s(a,x,g)=(x,g), \;\; t(a,x,g)=(\tau (a)+x,g)\\
  \Mor(\wh\cG_{(F,\Theta)})&= A\times Z\times G, & \quad &
  \id_{(x,g)}=(0,x,g),\;\; (a,x,g)\circ (b,y,g)=(a+b,y,g)
 \end{alignat*}
 and the multiplication functor by
 \begin{equation*}
  (a,x,g)\otimes(b,y,h)=(a+b,x+y+F(g,h),gh).		
 \end{equation*}
 Since $F$ satisfies the cocycle identity only up to correction by
 $\Theta$, this assignment defines a monoidal category if we define $\nelt=(0,e)$ and
 \begin{equation*}
  \alpha_{(x,g),(y,h),(z,k)}=(\Theta(g,h,k),x+y+z+F(g,h)+F(gh,k),ghk).
 \end{equation*}
 We clearly have $1\otimes g=g=g\otimes \nelt$, the
 source-target matching condition of $\alpha$ is equivalent to
 $\dd_{\op{gp}}F=\tau \circ \Theta$ and the pentagon identity is
 equivalent to $\dd_{\op{gp}}\Theta=0$. Moreover,
 \begin{equation*}
  \ol{(a,x,g)}=(-a,-x-F(g,g^{-1}),g^{-1})
 \end{equation*}
 defines an inversion functor on $\wh\cG_{(F,\Theta)}$, turning it into
 a 2-group. In addition to the 2-group structure on
 $\wh\cG_{(F,\Theta)}$, we have canonical functors
 $\cZ_{\tau}\xrightarrow{i} \wh{\cG}_{(F,\Theta)}$ and
 $\wh{\cG}_{(F,\Theta)}\xrightarrow{q} \ul{G}$.
\end{remark}

\begin{lemma}
 In the setting of the previous remark,
 $ \cZ_{\tau}\xrightarrow{i}\wh{\cG}_{(F,\Theta)}\xrightarrow{q}\ul{G}$
 is a central extension of $\ul{G}$ by $\cZ_{\tau}$.
\end{lemma}

\begin{proof}
 We first notice that the functors from
 \eqref{eqn:functors_defining_action_on_the_kernel} are actually equal
 in this case, so the extension will be central.

 We abbreviate $\cZ:=\cZ_{\tau}$ and $\cG=\wh{\cG}_{(F,\Theta)}$. Assume
 that $m\from W\to \cG$ is given such that $q \op{\circ} m=\nelt$. Then
 on objects we have that
 $m_{0}(w)\in q_{0}^{-1}(\nelt)=Z\times e_{G}\cong\cZ_{0}$ and on
 morphisms we have
 $m_{1}(v)\in q_{1}^{-1}(\nelt)=A\times Z\times \{e_{G}\}\cong \cZ_{0}$.
 So $m$ factors (on the nose) through a morphism $s\from W\to \cZ$,
 i.e., we may choose $\xi\from i \op{\circ} s\Rightarrow m$ (and of
 course also $\zeta\from * \op{\circ} s\Rightarrow *$) to be the
 identity natural transformations. Moreover, if we have
 $s'\from W\to \cZ$ and 2-isomorphisms
 $\alpha\from i \op{\circ} s\Rightarrow i \op{\circ} s' $ such that
 $q_{1}(\alpha(w))=\id_{\nelt}$, then
 $\alpha(w)\in q_{1}^{-1}(\nelt)=A\times Z\times\{e_{G}\}\cong\cZ_{1}$
 so that $\alpha$ factors through a 2-isomorphism
 $\chi \from s\Rightarrow s'$ which obviously satisfies the
 requirements. This shows that
 $\cZ_{\tau}\xrightarrow{i}\wh{\cG}_{(F,\Theta)}\xrightarrow{q}\ul{G}$  
is a
 2-pull back. Along similar lines one shows that it also is a 2-pushout.
\end{proof}

\begin{tabsection}
 We will now consider extensions of Lie 2-groups. Note that in the case
 of Lie groups (in
 \cite{Neeb02Central-extensions-of-infinite-dimensional-Lie-groups}) or
 in the setting of smooth 2-groups (in
 \cite{Schommer-Pries10Central-Extensions-of-Smooth-2-Groups-and-a-Finite-Dimensional-String-2-Group}) there is an
 additional requirement on a sequence
 $A\xrightarrow{i} B\xrightarrow{q} C$ besides that the diagram from
 \eqref{eqn:pullback-ordinary}, respectively
 \eqref{eqn:pullback-categorical}, is a (2-)pullback and a (2-)pushout.
 For Lie group extensions one requires the existence of a smooth local
 section (this then implies that $B\to C$ is a locally trivial principal
 $A$-bundle) and in
 \cite{Schommer-Pries10Central-Extensions-of-Smooth-2-Groups-and-a-Finite-Dimensional-String-2-Group} it is
 required that $A\xrightarrow{i} B\xrightarrow{q} C$ is an $A$-gerbe
 over $C$.\\

 In our treatment we restrict from now on to \'etale Lie
 2-groups, a concept that we are heading for now. This concept will be
 tailored to fit our Lie theoretic needs.
\end{tabsection}

\begin{remark}\label{rem:functor-from-2-spaces-to-2-vector-spaces}
 Similarly to the concept of a smooth 2-space (and smooth functors and
 natural transformations, cf.\ Definition \ref{def:2-spaces}), one
 defines (topological) \emph{vector 2-spaces} to be internal categories
 in locally convex vector spaces, i.e., small categories such that all
 sets occurring in the definition of a small category are locally convex
 spaces, all structure maps are continuous linear maps and source and
 target are projections. Likewise, linear functors and natural
 transformations are defined internally, defining the 2-category
 $\cat{2-Vect}$.

 There is a natural functor ${T}$ from the category $\cat{Man_{pt}}$ (of
 pointed manifolds with smooth base-point preserving maps) to the
 category $\cat{Vect}$ (of topological vector spaces with continuous
 linear maps), sending manifolds to the tangent spaces at the base-point
 and smooth maps to their differentials at the base-point. Since this
 functor preserves pull-backs, it maps categories, functors and natural
 transformation in $\cat{Man_{pt}}$ to ones in $\cat{Vect}$ and thus
 defines a 2-functor $\cT \from \cat{2-Man_{pt}}\to\cat{2-Vect}$.
\end{remark}

\begin{tabsection}
 If we want to enforce $\cT$ to take values in $\cat{Vect}$ instead of
 $\cat{2-Vect}$, then we need a canonical identification of
 $\cT(\cM)_{0}$ and $\cT(\cM)_{1}$. This is the case if $\cM$ is
 \'etale, as defined below.
\end{tabsection}

\begin{definition}\label{def:etale-lie-2-group}
 A smooth 2-space is called \emph{\'etale} if all structure maps are
 local diffeomorphisms. A Lie 2-group $(\cG,\cU)$ is called \'etale if
 $\cU$ is an \'etale 2-space. Morphisms and 2-morphisms for \'etale Lie
 2-groups are defined to be morphisms and 2-morphisms of Lie 2-groups.
 The corresponding 2-category is denoted by
 $\cat{Lie 2-Grp_{\acute{e}t}}$.
\end{definition}

\begin{tabsection}
 Most of the Lie 2-groups that we shall encounter in this article are
 \'etale. Note that the differentials of local diffeomorphisms give
 canonical identifications of the tangent spaces at the base-points.
 Thus $\cT(\cM)$ is in fact a vector space for \'etale $\cM$. We shall
 make this precise for Lie 2-groups below. Note also that
 $s^{-1}(\nelt)$ is discrete in an \'etale Lie 2-group. In particular,
 Lemma \ref{lem:triviality-of-associator-in-the-etale-case} applies to
 \'etale Lie 2-groups with globally smooth group operations.
\end{tabsection}

\begin{remark}\label{rem:integrating-Lie-algebriods}
 The passage from general Lie 2-groups to \'etale ones will have an
 effect that has also been used in
 \cite{TsengZhu06Integrating-Lie-algebroids-via-stacks} for the solution
 of the integration problem of (finite-dimensional) Lie algebroids (cf.\
 \cite{CrainicFernandes03Integrability-of-Lie-brackets}). It is a
 long-standing observation that Lie algebroids integrate to local Lie
 group
 \cite{Pradines68Troisieme-theoreme-de-Lie-les-groupoi-des-differentiables},
 but in general the integrating Lie groupoid may not be enlarged to a
 global Lie groupoid. The reasons for this failure is essentially the
 non-discreteness of the image of a period map (cf.\ \cite[Sect.\ 3.2 and
 Th.\ 4.1]{CrainicFernandes03Integrability-of-Lie-brackets}) for which
 \cite[Ex.\ 3.7]{CrainicFernandes03Integrability-of-Lie-brackets} and
 \cite[Ex.\ 1]{TsengZhu06Integrating-Lie-algebroids-via-stacks} give
 examples, very close to the integration problem that this
 article deals with. In
 \cite{TsengZhu06Integrating-Lie-algebroids-via-stacks} the integrating
 objects are \emph{Weinstein groupoids}, which are also \'etale and
 categorified replacements of Lie groupoids. The result from
 \cite{TsengZhu06Integrating-Lie-algebroids-via-stacks} can also be seen
 as integrating Lie algebroids to locally defined Lie groupoids and then
 solve the associativity-constraint by passing to Weinstein groupoids.
 In the same spirit, \'etale Lie 2-groups will be the integrating
 objects for integration locally exponential Lie algebras.
\end{remark}

\begin{definition}\label{def:Generalised-central-extension-of-Lie-groups}
 Let $G$ be an arbitrary Lie group and $\tau\from A\to Z$ be a morphism
 of abelian Lie groups with \emph{discrete} $A$. Then a \emph{smooth
 generalised central extension} (s.g.c.e.) is a sequence
 $\cZ_{\tau}\xrightarrow{i} \wh{\cG}\xrightarrow{q} \ul{G}$ of \'etale
 Lie 2-groups such that $p \op{\circ} i=\nelt$ and that the diagram
 \eqref{eqn:pullback-categorical}
 is a 2-pullback and a 2-pushout in \cat{Lie 2-Grp_{\acute{e}t}}.
 Moreover, we demand that there exists a smooth functor
 $q\from \ul{U} \to \wh{\cG}$ satisfying $q \circ s =\id_{\ul{U}}$,
 where $U\se G$ is some open identity neighbourhood and that the functors
 \begin{equation*}
  \cZ_{\tau}\times\wh{\cG}\to\wh{\cG},
  \quad (z,g)\mapsto g\otimes(i(z)\otimes \ol{g})
  \quad\tx{ and }\quad
  \cZ_{\tau}\times\wh{\cG}\to\wh{\cG},\quad (z,g)\mapsto i(z)
 \end{equation*}
 are smoothly isomorphic when restricted to some neighbourhood of
 $(0,e_{G})\in(\cZ_{\tau}\times\cG)_{0}$.
\end{definition}

\begin{tabsection}
 The requirement on a s.g.c.e. to be a sequence in \emph{\'etale} Lie
 2-group will enable us to take an easy way to central extensions of Lie
 algebras (cf.\ Proposition
 \ref{prop:s.g.c.e.-yield-central-extensions-of-lie-alg}). The
 \'etalness is not crucial for the definition to make sense, the same
 definition of course also works in \cat{Lie 2-Grp}. The following lemma
 is immediate from the definitions.
\end{tabsection}

\begin{lemma}\label{lem:From-generalised-cocycles-to-smooth-central-extensions}
 If $(F,\Theta)$ is a generalised cocycle on $G$ with coefficients
 $\tau\from A\to Z$ and $A$ is discrete, then the 2-group
 $\wh\cG_{(F,\Theta)}$ from Remark
 \ref{rem:from-generalised-cocycles-to-2-groups} is canonically an
 \'etale Lie 2-group and
 $ \cZ_{\tau}\xrightarrow{i}\wh{\cG}_{(F,\Theta)}\xrightarrow{q} \ul{G}$
 is a s.g.c.e.
\end{lemma}

The following proposition describes the way back from generalised
central extensions to ordinary ones. It is the categorical version of
the discreteness condition for $\per_{\omega}(\pi_{2}(G))$ from
\cite{Neeb02Central-extensions-of-infinite-dimensional-Lie-groups}.

\begin{proposition}
 Let $\cZ_{\tau}\xrightarrow{i} \wh{\cG}\xrightarrow{q} \ul{G}$ be a
 s.g.c.e. such that $\tau(A)\se Z$ is discrete. Then
 $\pi_{0}(\cZ_{\tau})$ and $\pi_{0}(\wh{\cG})$ carry Lie group
 structures with modelling space $\fz$ and $\fz\times\fg$
 (respectively), turning
 \begin{equation}\label{eqn:band-of-s.g.c.e.}
  \pi_{0}(\cZ_{\tau})\xrightarrow{\pi_{0}(i)} \pi_{0}(\wh{\cG})\xrightarrow{\pi_{0}(q)} G  
 \end{equation}
 into a central extension of Lie groups.
\end{proposition}

\begin{proof}
 First we note that $\pi_{0}(\cZ_{\tau})\cong Z/\tau(A)$ has a natural Lie
 group structure with Lie algebra $\fz$.
 Let $s\from \ul{U}\to \wh{\cG}$ be a smooth section of $q$. Then
 $(\pi_{0}(q))^{-1}(U)\cong \pi_{0}(\cZ_{\tau})\times U$ as a set and we endow
 $(\pi_{0}(q))^{-1}(U)$ with the smooth structure making this
 identification a diffeomorphism. Since the group multiplication on
 $\wh{\cG}$ is smooth on an open subcategory containing $\nelt$, the group
 multiplication in $\pi_{0}(\wh{\cG})$ is locally smooth. Since
 $U$ generates $G$, $\pi_{0}(q)^{-1}(U)$ generates $\pi_{0}(\wh{\cG})$
 and the assertion follows from Theorem
 \ref{thm:globalisation-of-smooth-structures-on-groups}.
\end{proof}

\begin{definition}
 The induced central extension \eqref{eqn:band-of-s.g.c.e.} is
 called the \emph{band} of
 $  \cZ_{\tau}\xrightarrow{i} \wh{\cG}\xrightarrow{q} \ul{G}$.
\end{definition}

\begin{corollary}
 If $\omega\from \fg\times\fg\to\fz$ is a Lie algebra cocycle and
 $(F,\Theta)$ is a generalised cocycle which integrates $\omega$ (cf.\
 Definition \ref{def:integrating-central-extensions}) and if
 $\tau(A)\se Z$ is discrete, then the band of $\wh\cG_{(F,\Theta)}$ is a
 central extension $Z/\tau(A)\to \wh{G}\to G$
 integrating $\fz\to\fz\oplus_{\omega}\fg\to\fg$.
\end{corollary}

\begin{proof}
 To see that the band of $\wh\cG_{(F,\Theta)}$ integrates
 $\fz\to\fz\oplus_{\omega}\fg\to\fg$ we first observe that for
 $q\from Z\to Z/\tau(A)$ the canonical quotient map
 $T\cq{(Z/\tau(A))}(e)\from \fz=T_eZ\to T_e(Z/\tau(A))$ is an
 isomorphism for $\tau(A)$ is discrete. Using this to identify $\fz$
 with $T_e(Z/\tau(A))$ the claim  follows from
 $L(\cq{(Z/\tau(A))} \circ F)=T\cq{(Z/\tau(A))}(e) \circ L(F)$.
\end{proof}

\begin{remark}
 We now derive a Lie algebra canonically associated to each \'etale Lie
 group $\cG$. We first show that the associator $\alpha$ is trivial on
 some neighbourhood of $\nelt$. Since $\cG$ is \'etale, the identity map
 $\cG_{0}\to \cG_{1}$ is a local inverse around $\nelt$ for both, $s$
 and $t$. Since $\alpha_{\nelt,\nelt,\nelt}=\nelt$, we thus have
 $\id \circ t \circ\alpha =\id \circ s \circ\alpha$, which implies
 $s \circ \alpha=t \circ \alpha$ on some neighbourhood of $\nelt$. Now
 multiplying $\alpha(g,h,k)$ with $\id_{\ol{(g\otimes h)\otimes k}}$
 defines a map with values in $s^{-1}(\nelt)$, which is continuous on
 some identity neighbourhood and thus constantly $\nelt$. Since
 $\alpha_{x,\ol{x},x}$ is an identity for each $x$ and $\alpha$ is
 natural, all of this implies
 \begin{equation*}
  \id_{(g\otimes h)\otimes k}=(\alpha_{g,h,k}\otimes \id_{\ol{(g\otimes h)\otimes k}})\otimes \id_{(g\otimes h)\otimes k}
  =\alpha_{g,h,k}
 \end{equation*}
 on some identity neighbourhood, which yields
 \begin{equation*}
  (g\otimes h)\otimes k=g\otimes(h\otimes k)
 \end{equation*}
 for $g,h,k$ from some neighbourhood of $\nelt$. Thus the multiplication
 functor defines on $\cG_{0}$ the structure of a local Lie group and
  induces on $T_{\nelt}\cG_{0}$ a Lie bracket. We denote this Lie
 algebra by $\mc{L}(\cG)$.

 A similar argument as above shows that for a morphism
 $\cF\from \cG\to\cG'$ of Lie 2-groups, where $\cG$ and $\cG'$ are
 \'etale, we have $\cF(g)\otimes'\cF(h)=\cF(g\otimes h)$ for $g,h$ from
 some identity neighbourhood. Thus $\cF_{0}$ induces a morphism of local
 Lie groups and thus of Lie algebras
 $\cL(\cF)\from \mc{L}(\cG)\to\mc{L}(\cG')$. Likewise, a 2-morphisms
 $\theta$ between two such morphisms has to be the identity on some
 identity neighbourhood. Summarising,
 \begin{equation*}
  \cL\from \cat{Lie 2-Grp_{\op{\acute{e}t}}}\to \cat{Lie Alg}
 \end{equation*}
 defines a 2-functor to the category of Lie algebras, considered as a
 2-category with only identity 2-morphisms.
\end{remark}

\begin{proposition}
\label{prop:s.g.c.e.-yield-central-extensions-of-lie-alg}
Let $\cZ_{\tau}\xrightarrow{i}\wh\cG\xrightarrow{q} \ul{G}$ be a s.g.c.e.
Then 
\begin{equation}
\label{eqn:derived-central-extension}
\cL(\cZ_{\tau})\xrightarrow{\cL(i)}\cL(\wh{\cG})\xrightarrow{\cL(q)}
\cL(\ul{G})
\end{equation}
is a central extension of Lie algebras.
\end{proposition}

\begin{proof}
 That \eqref{eqn:derived-central-extension} is short exact follows the
 fact that the 2-functor $\cL$ preserves 2-limits, which turn into
 ordinary limits in $\cat{LieAlg}$. The differential of a section (on objects) of
 $\cZ_{\tau}\xrightarrow{i}\wh\cG\xrightarrow{q} G$ provides a linear and
 continuous section of \eqref{eqn:derived-central-extension}.
\end{proof}

\begin{definition}\label{def:integrating-central-extensions}
 For a s.g.c.e. $\cZ_{\tau}\xrightarrow{i}\wh{\cG}\xrightarrow{q} \ul{G}$
 its \emph{derived} central extension is the central extension
 \eqref{eqn:derived-central-extension}.
 If $\fz\to\wh{\fg}\to\fg$ is a topologically split central extension,
 then it is said to \emph{integrate} to a
 smooth generalised central extension if there exists a s.g.c.e. such that its
 derived central extension is equivalent to $\fz\to\wh{\fg}\to\fg$.
\end{definition}

\begin{theorem}\label{thm:integrating-central-extensions}
 If $\fg$ is the Lie algebra of the simply connected Lie group $G$, then
 each topologically split central extension $\fz\to\wh{\fg}\to\fg$
 integrates to a smooth generalised central extension of \'etale Lie
 2-groups.
\end{theorem}

\begin{proof}
 We may assume that $\wh{\fg}$ is equivalently given by a Lie algebra
 cocycle $\omega\from\fg\times\fg\to\fz$, which we integrate to a
 $\cZ_{\per_\omega}$-valued cocycle $(F_{\omega,\beta},\Theta_{\beta})$
 by Theorem \ref{thm:integration-of-cocycles} for some appropriate
 choice of $\beta$. Then Lemma
 \ref{lem:From-generalised-cocycles-to-smooth-central-extensions} yields
 a s.g.c.e.
 $\cZ_{\per_{\omega}}\to \wh\cG_{(F_{\omega,\beta},\Theta_{\beta})}\to \ul{G}$.

 Let $U,V\se G$ be open identity neighbourhoods such that
 $\left.F\right|_{U\times U}$ and
 $\left.\Theta\right|_{U\times U\times U}$ are smooth and
 $V\cdot V\se U$. To calculate the derived central extension we
 consider the restriction of the multiplication functor $m$ to the full
 subcategory with objects in $\fz\times U$, where it is given by
 \begin{equation*}
  m_{0}((z,g),(w,h))=(z+w+F_{\omega,\beta}(g,h),gh)
 \end{equation*}
 on objects. By the definition of the Lie bracket of a local Lie group,
 the Lie bracket on $\mc{L}(\wh\cG_{(F_{\omega,\beta},\Theta_{\beta})})$
 is given by
 \begin{equation*}
  ((z,x),(w,y))\mapsto (L(F_{\omega,\beta})(x,y),[x,y])
 \end{equation*}
 and $L(F_{\omega,\beta})=\omega$ shows the claim.
\end{proof}

We thus recover the classical case of central extensions by passing from a
generalised central extension to its band in the case that
$\per_{\omega}(\pi_{2}(G))\se\fz$ is discrete. Moreover, we can interpret the
proof of the previous theorem as first passing to a 2-connected cover of $G$
and then solve a trivial integration problem in the following sense.

\begin{remark}\label{rem:2-connected-cover}
 Let $\beta\from G^{2}\to C^{\infty}_{*}(\Delta^{(2)}, G)$ be the map
 from Lemma \ref{lem:homology-cocycle-actual-choice-for-a-chart},
 applied to a chart $\varphi$ with $d \varphi(e)=\id_{\fg}$ and
 $\Theta_{\beta}\from G^{3}\to \pi_{2}(G)$ be the corresponding group
 3-cocycle from Lemma \ref{lem:homology-cocycle}. Then $\Theta_{\beta}$
 determines an (in general non-strict) Lie 2-group
 $\cG:=\wt{\cG}_{(0,\Theta)}$ (cf.\ Lemma
 \ref{lem:From-generalised-cocycles-to-smooth-central-extensions}),
 which we interpret as an appropriate version of a 2-connected cover of
 $G$ (cf.\
 \cite{PorstWockel08Higher-conneced-covers-of-topological-groups-via-categorified-central-extensions}).
 In particular, we have a smooth generalised central extension
 \begin{equation*}
  B \pi_{2}(G)\to \cG\to\ul{G},
 \end{equation*}
 where $B \pi_{2}(G)$ is the strict Lie 2-group, associated to the
 crossed module $\pi_{2}(G)\to \{*\}$. Now
 $\wh{\cG}:=\wh{\cG}_{(F_{\omega,\beta},\Theta_{\beta})}$ can be seen as
 a central extension $\ul{\fz}\to \wh{\cG}\to \cG$ (when generalising
 central extensions of Lie 2-groups to non-\'etale ones in the obvious
 way). Summarising, we have the commutative diagram
 \begin{equation*}
  \xymatrix{
  \ul{\fz} \ar[r]\ar[d]&\ul{\fz}\ar[r]\ar[d]&\{*\}\ar[d]\\
  \cZ_{\per_{\omega}} \ar[r]\ar[d] & \wh{\cG} \ar[r]\ar[d] &\ul{G}\ar[d]\\
  B \pi_{2}(G) \ar[r]& \cG \ar[r]& \ul{G}
  }
 \end{equation*}
 with exact rows and columns. Since $\cG$ is an \'etale Lie 2-group with
 Lie algebra $\fg$, one may interpret $\ul{\fz}\to \wh{\cG}\to \cG$ also
 as a central extension of \'etale 2-groups integrating
 $\fz\to\wh{\fg}\to\fg$.
\end{remark}

\section{Lie's Third Theorem}\label{sect:lies_third_theorem}

\begin{tabsection}
 We conclude this paper with the following generalisation of Lie's Third
 Theorem. We briefly recall definitions and some basic facts.
\end{tabsection}

\begin{definition}
 A locally convex Lie algebra $\fg$ is said to be locally exponential if
 there exists a circular convex open zero neighbourhood $U\se \fg$ and
 an open subset $D\se U\times U$ on which there exists a smooth map
 \[
  m_{U}\from D\to U,\quad (x,y)\mapsto x*y
 \]
 such that $(D,U,m_{U},0)$ is a local Lie group and such that the
 following holds.
 \begin{enumerate}
        \renewcommand{\labelenumi}{\theenumi}
        \renewcommand{\theenumi}{\roman{enumi})}
  \item For $x\in U$ and $|t|, |s|, |t+s| \leq 1$, we have
        $(tx, sx) \in D$ with $tx*sx=(t+s)x$.
  \item The second order term in the Taylor expansion of $m_{U}$
        in $0$ is $b(x,y)=\frac{1}{2}[x,y]$.
 \end{enumerate}
\end{definition}

\begin{remark}
 (cf.\ \cite[Ex.\
 IV.2.4]{Neeb06Towards-a-Lie-theory-of-locally-convex-groups}) All
 Banach-Lie algebras are locally exponential, as well as all Lie
 algebras of locally exponential Lie groups.
\end{remark}

\begin{Theorem}
 (\cite[Thm.\
 IV.3.8]{Neeb06Towards-a-Lie-theory-of-locally-convex-groups})\label{thm:gAdIsEnlargible}
 Let $\fg$ be a locally exponential Lie algebra. Then the adjoint group
 $G_{\ad}\leq \Aut(\fg)$ carries the structure of a locally exponential
 Lie group whose Lie algebra is $\fg_{\ad}:=\fg /\fz(\fg)$.
\end{Theorem}

\begin{tabsection}
 The route to Lie's Third Theorem seems to be clear, simply integrate
 $\fz(\fg)\xrightarrow{\op{incl}} \fg\xrightarrow{q} \fg_{\ad}$.
 But the latter need not be topologically split, as the following example
 shows.
\end{tabsection}

\begin{example}\label{ex:topologically_non-split_adjoint_algebra}
 Let $F\leq E:=\ell^{p}(\N)$ for some $1<p<2$ be a non-complemented, in
 particular infinite-dimensional subspace, i.e., there exists no
 continuous projection $E\to F$. We choose a linearly independent
 sequence $(e_{n})_{n\in\N}$ in $F$. Moreover, we choose a linearly
 independent sequence $(a_{n})_{n\in\N}$ in $F^{\bot}$ such that
 $\op{span}\{a_{n}\}$ is dense in $F^{\bot}$ and for each $a_{n}$
 another linearly independent $b_{n}\in F^{\bot}$. Having fixed this we
 set
 \begin{equation*}
  [x,y]:=\sum_{n=1}^{\infty}\frac{1}{2^{n}}
  (a_{n}(x)b_{n}(y)-a_{n}(y)b_{n}(x))e_{n}.
 \end{equation*}
 Since $a_{n}(e_{m})=b_{n}(e_{m})=0$ we have that
 $[[x,y],z]+[[y,z],x]+[[z,x],y]=0$ and thus
 $[\mathinner{\cdot},\mathinner{\cdot}]$ defines a Lie bracket on $E$.
 An element $x\in E$ is in the centre precisely if the map
 $[x,\mathinner{\cdot}]$ is trivial. This is the case if $x\in F$. On
 the other hand, if $x\notin F$, then $a_{n}(x)\neq 0$ for at least one
 $n\in\N$. For each $0\neq y\in\ker(a_{n})$ we have $y\notin\ker(b_{n})$
 and thus $[x,y]\neq 0$. This shows
 $x\notin F\Rightarrow[x,\mathinner{\cdot}]\neq 0$ and thus $F$ is the
 centre of $E$.
\end{example}

\begin{tabsection}
 A procedure similar to considering generalised central extensions as in
 \cite[Sect.\ VI.1]{Neeb06Towards-a-Lie-theory-of-locally-convex-groups}
 now remedies this failure.
\end{tabsection}

\begin{theorem}
 If $\fg$ is a Mackey-complete locally exponential Lie algebra, then
 there exists an \'etale Lie 2-group $\cG$ such that $\cL(\cG)$ is
 isomorphic to $\fg$.
\end{theorem}

\begin{proof}
 We consider $\fg_{\ad}:=\fg /\fz(\fg)$ and the map
 $\fg\times\fg\to \fg$, $(x,y)\mapsto [x,y]$. This map vanishes if
 $x\in\fz(\fg)$ or $y\in\fz(\fg)$ and thus induces a continuous cocycle
 $\omega_{\fg}\from \fg_{\ad}\times\fg_{\ad}\to|\fg|$, where $|\fg|$
 denotes the Mackey-complete locally convex space underlying $\fg$. This
 integrates by Theorem \ref{thm:integration-of-cocycles} to a
 generalised cocycle $(F_{\omega_{\fg}},\Theta)$, which in turn gives
 rise to an \'etale Lie 2-group $\wh{\cG}_{(F_{\omega_{\fg}},\Theta)}$
 with set of objects $|\fg|\times G_{\ad}$. Moreover, we may assume that
 $F_{\omega_{\fg}}$ is smooth on $V\times V$ and $V=V^{-1}$.

 The exponential function $\exp_{\fg_{\ad}}\from\fg_{\ad}\to G_{\ad}$
 restricts to a diffeomorphism on some open zero neighbourhood
 $U\se\fg_{\ad}$ and we may assume that $\exp(U)\se V$. We now want to
 construct a local exponential function for
 $|\fg|\oplus_{\omega_{\fg}}\fg_{\ad}$ and for this first define
 $\gamma_{x}(t):=\exp_{\fg_{\ad}}(tx)$ for $x\in\fg_{\ad}$ and for
 $x\in U$ and $t\in[0,1]$ we set
 \begin{equation*}
  z_{x}(t):=-\int_{0}^{t}
  TF_{\omega_{\fg}}\left(0_{\gamma_{x}(s)^{-1}},
  \left.\frac{d}{du}\right|_{u=s}\gamma_{x}(u)\right)\;ds,
 \end{equation*}
 where $0_{\gamma_{x}(s)^{-1}}$ denotes the zero element in
 $T_{\gamma_{x}(s)^{-1}}G_{\ad}$. Note that the integral exists since
 $|\fg|$ is Mackey-complete. With this we set
 \begin{equation*}
  \eta_{(z,x)}\from[0,1]\to |\fg|\times G_{\ad},\quad \eta_{(z,x)}(t):=
  (tz+z_{x}(t),\gamma_{x}(t))
 \end{equation*}
 and observe that
 $\dot{z}_{x}(t)=-TF_{\omega_{\fg}}(0_{\gamma_{x}(t)^{-1}}, \dot{\gamma_{x}}(t))$
 implies
 \begin{equation*}
  \left.\frac{d}{dt}\right|_{t=0} \eta_{(z,x)}(t_{0})^{-1}\eta_{(z,x)}(t_{0}+t)=(z,x)
 \end{equation*}
 for $t_{0}\in (0,1)$. Thus
 \begin{equation*}
  \exp\from |\fg|\times U \to |\fg|\times \exp_{\fg_{\ad}}(U)\se|\fg|\times G_{\ad},\quad
  (z,x)\mapsto (z+z_{x}(1),\exp_{\fg_{\ad}}(x)).
 \end{equation*}
 defines a local exponential function, which is a diffeomorphism
 since $(z,x)\mapsto (z,\exp_{\fg_{\ad}}(x))$ is one.

 Now $\fg$ is isomorphic to the closed ideal $\{(x,q(x)):x\in\fg\}$ of
 $|\fg|\oplus_{\omega_{\fg}}\fg_{\ad}$, and thus $\exp$ restricts to a
 diffeomorphism of $(|\fg|\times U)\cap \fg$ onto
 $W:= (|\fg|\times \exp_{\fg_{\ad}}(U))\cap \exp(\fg)$. Note that we
 have in particular $\fz(\fg)\times \exp_{\fg_{\ad}}(U) \se W$. We
 define $\cG$ to be the monoidal subcategory of
 $\wh{\cG}_{(F_{\omega_{\fg}},\Theta)}$, generated by the full
 subcategory $\mc{W}$ determined by $W$.

 It remains to check that $\cG$ defined this way actually is an \'etale
 Lie 2-group with Lie algebra $\fg$. We have $\cW_{0}=W$ (by definition)
 and $\cW_{1}=\pi_{2}(G_{\ad})\times W$, since
 $\per_{\omega_{\fg}}\from \pi_{2}(G_{\ad})\to |\fg|$ takes values in
 $\fz(\fg)\se \fg$ by \cite[Th.\
 VI.1.6.]{Neeb06Towards-a-Lie-theory-of-locally-convex-groups} and
 $\fz(\fg)\times \exp_{\fg_{\ad}}(U) \se W$. With the restricted
 structure maps this clearly is an \'etale 2-space and thus $(\cG,\cW)$
 is an \'etale Lie 2-group. Since $\exp$ is a local exponential function
 it is also clear that the Lie algebra, associated to the local group
 $(\mu^{-1}(W),W,\mu,(0,e))$ (with
 $\mu((z,x),(w,y))=(z+w+F_{\omega_{\fg}}(x,y),xy)$) is isomorphic to
 $\fg$. Thus $\cL(\cG)\cong\fg$.
\end{proof}

\begin{proof}
 \textbf{(of Proposition
 \ref{prop:integrating_lie_algebras_to_loop_prolongations})} The set of
 objects of the \'etale Lie 2-group $\cG$ constructed in the previous
 theorem give rise to a loop, which restricts to a locally smooth loop
 on some open neighbourhood of $\nelt$. Since the Lie 2-group is
 \'etale, this locally smooth loop is also locally associative.
 Moreover, the Lie algebra associated to this locally smooth and locally
 associative loop coincides with $\cL(\cG)$ and thus is isomorphic to $\fg$.
\end{proof}

\section{Prospects}

We tried to develop a completed account on the integration of infinite-dimensional
Lie algebras to Lie 2-groups. In order to do so we dropped some topics that may be
at hand which we shortly line out in this section. Most of them deserve to be
worked out seriously.

\begin{remark}[Diffeological Lie Groups]\label{rem:diffeological-lie-groups}
 The problem that one encounters when trying to integrate central
 extensions of infinite-dimensional Lie algebras to Lie groups is that
 one has to factor out subgroups from locally convex spaces that may be
 not discrete. This has to be done to ensure that the cocycle condition
 for a certain universal cocycle holds.

 However, one may resolve this problem by enlarging the category of
 smooth manifolds to a category in which this quotient exists. For
 instance, the category of diffeological spaces (or more general smooth
 spaces, cf.\
 \cite{BaezHoffnung08Convenient-Categories-of-Smooth-Spaces}) has this
 property. From our cocycle $(F_{\omega,\beta},\Theta_{\beta})$,
 integrating a given Lie algebra cocycle $\omega$, one obtains an
 ordinary group cocycle $q \circ F_{\omega,\beta}$, which is in general
 (locally) smooth as a map between diffeological spaces, because the
 quotient map $q\from \fz\to \fz/\Pi_{\omega}$ is smooth, no matter
 whether $\Pi_{\omega}:=\per_{\omega}(\pi_{2}(G))$ is discrete or not.
 With the corresponding version of Theorem
 \ref{thm:globalisation-of-smooth-structures-on-groups} for
 diffeological spaces one thus constructs a diffeological group
 $\wh{G}_{\omega}$ and
 \begin{equation*}
  \fz/\Pi_{\omega}\to \wh{G}_{\omega}\to G
 \end{equation*}
 is a candidate for a central extension of diffeological groups,
 integrating
 \begin{equation*}
  \fz \to \fz\oplus_{\omega}\fg\to\fg.
 \end{equation*}
 The crucial point here would be to set up the notion of a Lie functor from
 diffeological spaces to vector spaces such that it takes
 $\fz/\Pi_{\omega}$ to $\fz$, even if $\Pi_{\omega}$ is not discrete
 (such a thing should exist according to \cite{Schreiberpersonal-communicaltion}).

 In \cite{Iglesias95La-trilogie-du-moment}, a similar construction has
 been done in order to obtain a prequantisation for an arbitrary
 symplectic manifold $(M,\omega)$, with not necessarily integral
 $[\omega]\in H^{2}_{\op{dR}}(M)$. In particular, prequantisation can be
 performed by directly passing to the dual of the Lie algebra, without
 constructing a Lie algebra at first\footnote{Even if there is a Lie
 algebra around, there is a priori no canonical dual space, associated
 to it, for the usual topologies on dual spaces are not good enough
 (cf.\ \cite{Neeb06Towards-a-Lie-theory-of-locally-convex-groups}). So
 it is more natural to pass directly to the dual.}
 \cite{Iglesias08The-moment-maps-in-diffeology}.
\end{remark}

\begin{remark}[Differential Geometry of Generalised Extensions]\label{rem:prospect:differential_geometry_of_generalised_central_extensions}
 One perspective to the integration procedure for central extensions of
 Lie algebras is to find a Lie group extension as a principal bundle
 with a prescribed curvature. It should be possible to develop such a
 point of view also for smooth generalised central extensions, a similar
 perspective has been taken, for instance, by Schommer-Pries
 \cite{Schommer-Pries10Central-Extensions-of-Smooth-2-Groups-and-a-Finite-Dimensional-String-2-Group}.

 On the level of cocycles, the passage is quite clear. For a cocycle
 $f\from G\times G\to Z$, smooth on $U\times U$, the central extension
 $Z\to Z\times _{f}G\to G$ is a principal bundle, described by the
 transgressed \v{C}ech cocycle
 \begin{equation*}
  \gamma_{g,h}\from gV\cap hV\to Z,\quad x\mapsto f(g,g^{-1}x)-f(h,h^{-1}x)
 \end{equation*}
 where $V\se U$ is an open identity neighbourhood with $V\cdot V\se U$.
 That $\gamma_{g,h}$ is smooth follows from
 \begin{equation*}
  f(g,g^{-1}x)-f(h,h^{-1}x)=f(g^{-1}h,h^{-1}x)-f(g,g^{-1}h)
 \end{equation*}
 and from $g^{-1}h\in U$ if $gV\cap hV\neq \emptyset$. For a generalised
 cocycle $(F,\Theta)$ the transgressed non-abelian \v{C}ech cocycle is
 accordingly given by
 \begin{equation*}
  \gamma_{g,h}\from gV\cap hV\to Z,\quad x\mapsto F(g,g^{-1}x)-F(h,h^{-1}x)-\tau(\Theta(g,g^{-1}h,h^{-1}x))
 \end{equation*}
 and
 \begin{multline*}
  \eta_{g,h,k}\from gV\cap hV\cap kV\to Z,\\ 
  x\mapsto -\Theta(g,g^{-1}h,h^{-1}x)-\Theta(h,h^{-1}k,k^{-1}x)+\Theta(g,g^{-1}k,k^{-1}x).
 \end{multline*}
 This yields a principal $\cZ$-2-bundle $\cP$ (over $\ul{G}$)
 \cite{Wockel09Principal-2-bundles-and-their-gauge-2-groups}, which is
 as a groupoid (without any additional structure) equivalent to
 $\cG_{F,\Theta}$. Applied to the string cocycle $\varphi$ from Example
 \ref{ex:string-2-group} this 2-bundle is the prequantisation  for the
 2-plectic manifold
 $(G,\langle[\mathinner{\cdot},\mathinner{\cdot}],\mathinner{\cdot}\rangle)$
 \cite{Brylinski93Loop-spaces-characteristic-classes-and-geometric-quantization}
 and
 \cite{BaezHoffnungRogers10Categorified-symplectic-geometry-and-the-classical-string}.
 In general, the interpretation of $\cP$ as a bundle with connection is be a bit
 more tricky since the principal bundle $\pi_{0}(\cP)$ should admit
 curvature (in fancy terms, we want the fake curvature \emph{not} to
 vanish). The theory of higher bundles with connection is being
 developed at the moment (cf.\
 \cite{SchreiberWaldorf08Connections-on-non-abelian-Gerbes-and-their-Holonomy},
 \cite{NikolausWaldorf11Four-Equivalent-Versions-of-Non-Abelian-Gerbes},
 \cite{Schreiber11Differential-Cohomology-in-a-Cohesive-infty-Topos},
 references therein and
 \cite{Waldorf10Multiplicative-bundle-gerbes-with-connection} for the
 case of group extensions).
\end{remark}

\begin{remark}[Non-Locally Exponential Lie Algebras]
 One may wonder whether a similar theorem as our
 version for Lie's Third Theorem is also in reach for non-locally
 exponential Lie algebras. To our best knowledge it would be unlikely to
 expect a similar result in this direction, for the algebraic properties
 of non-locally exponential Lie algebras couple very hardly to their
 local Lie groups (if they exist at all). For instance, Lempert proved
 that $\cV(M)_{\C}$ is even not \emph{integrable} for any compact
 manifold $M$ (cf.\
 \cite{Lempert97The-problem-of-complexifying-a-Lie-group}), which relies
 on more involved arguments as the
 counterexample of van Est and Korthagen in
 \cite{EstKorthagen64Non-enlargible-Lie-algebras}.
\end{remark}

\begin{remark}[Higher Lie Algebras and Lie Algebroids]
 In a sense, we performed a similar integration procedure as
 Henriques in \cite{Henriques08Integrating-Lsb-infty-algebras}. It thus
 seems to be promising to carry this analogy further to integrate even
 infinite-dimensional Lie 2-algebras or to enlarge Henriques' procedure
 beyond the Banach-case. Since the obstructions for integrating
 locally exponential Lie algebras and
 finite-dimensional Lie algebroids
 \cite{CrainicFernandes03Integrability-of-Lie-brackets} seem to be the
 same, an integration procedure for special classes of
 infinite-dimensional Lie algebroids (e.g. Banach-Lie algebroids) as
 in \cite{TsengZhu06Integrating-Lie-algebroids-via-stacks} is quite
 likely.
\end{remark}

\begin{remark}[Stacky Lie groups]
 Our definition of a Lie 2-group is somewhat weaker than one would
 expect at first. However, if one leaves the world of manifolds and
 considers Lie groupoids as presentations of differentiable stacks, then
 we expect that Lie 2-groups as defined above lead to stacky Lie groups
 in the sense of \cite{Blohmann08Stacky-Lie-groups}.\\
 
 \noindent\textbf{Problem:} If $\cG$ is a Lie 2-group (in the sense of
 Definition \ref{def:Lie-2-group}), does there exist a stacky Lie group
 (in the sense of \cite{Blohmann08Stacky-Lie-groups}) or alternatively a
 smooth group stack $\cH$ such that the underlying 2-groups are
 equivalent and the smooth stacks are equivalent in a ``neighbourhood''
 of the identity? If this is the case, can this correspondence be
 promoted to an equivalence of the corresponding 2-categories?\\
 
 One possible way to obtain this would be to follow the usage of the
 associativity of the group multiplication through Theorem
 \ref{thm:globalisation-of-smooth-structures-on-groups}. The coordinates
 of the Lie group structure on $G$ would yield a Lie groupoid, the
 multiplication on $G$ a Hilsum-Skandalis morphism describing the
 multiplication morphism between the stacky Lie groups and the usage of
 the associativity, finally, the associator 2-morphism.

 The above problem seems to be solvable since in the cases know to the
 author an ad-hoc construction yields stacky Lie groups from Lie
 2-groups. For the String 2-group from Example \ref{ex:string-2-group}
 this is the construction of Schommer-Pries in
 \cite{Schommer-Pries10Central-Extensions-of-Smooth-2-Groups-and-a-Finite-Dimensional-String-2-Group}
 and for the 2-groups $\wh{\cG}_{(F_{\omega,\beta},\Theta_{\beta})}$ and
 $\wt{\cG}_{(0,\Theta_{\beta})}$ from Theorem
 \ref{thm:integrating-central-extensions} and Remark
 \ref{rem:2-connected-cover} this is carried out in
 \cite{ZhuWockel10Integrating-central-extensions-of-Lie-algebras-via-group-stacks}.
 Moreover, it would be desirable to work out a Lie theory of stacky Lie
 groups directly in the correct categorical setup.
\end{remark}

\appendix

\section{Appendix: Differential calculus on locally convex spaces}
\label{sect:appendix}

We provide some background material on locally convex Lie groups and
their Lie algebras in this appendix.

\begin{definition}
 \label{def:diffcalcOnLocallyConvexSpaces} Let \mbox{ $X$} and \mbox{
 $Y$} be a locally convex spaces and \mbox{ $U\se X$} be open. Then
 \mbox{ $f\from U\to Y$} is \emph{differentiable}  or \emph{\mbox{
 $C^{1}$}} if it is continuous, for each \mbox{ $v\in X$} the
 differential quotient
 \[
  df (x).v:=\lim_{h\to 0}\frac{f (x+hv)-f (x)}{h}
 \]
 exists and if the map \mbox{ $df\from U\times X\to Y$} is continuous.
 If \mbox{ $n>1$} we inductively define \mbox{ $f$} to be \emph{\mbox{
 $C^{n}$}} if it is \mbox{ $C^{1}$} and \mbox{ $df$} is \mbox{
 $C^{n-1}$} and to be \mbox{ $C^{\infty}$} or \emph{smooth}  if it is
 \mbox{ $C^{n}$}. We say that \mbox{ $f$} is \mbox{ $C^{\infty}$} or
 \emph{smooth} if \mbox{ $f$} is \mbox{ $C^{n}$} for all \mbox{
 $n\in\N_{0}$}. We denote the corresponding spaces of maps by \mbox{
 $C^{n}(U,Y)$} and \mbox{ $C^{\infty}(U,Y)$}.

 A (locally convex) \textit{Lie group} is a group which is a smooth
 manifold modelled on a locally convex space such that the group
 operations are smooth. A locally convex Lie algebra is a Lie algebra,
 whose underlying vector space is locally convex and whose Lie bracket
 is continuous.
\end{definition}

\begin{remark}
 We have the chain rule
 \begin{equation*}
  d(g \op{\circ} f)(x).v=dg(f(x)).(df(x).v)
 \end{equation*}
 and the identities $d^{2}f(x)(v,w)=\pr_{2}(d(Tf)(x,v).(w,0))$ (more
 precisely
 \begin{equation*}
  d(Tf)(x,v)(w,0)=(df(x).w,d^{2}f(x)(v,w)) \,)
 \end{equation*}
 and
 \begin{equation*}
  d(Tf)(x,v)(w,w')=d(Tf)((w,0)+(0,w'))=(df(x).w,d^{2}f(x)(v,w))+(0,df(x).w').
 \end{equation*}
 This implies the ``chain rule'' for $d^{2}f$:
 \begin{equation}\label{eqn:chain_rule_for_d2}
  d^{2}(g \op{\circ} f)(x).(v,w)=d^{2}g(f(x))(df(x).v,df(x).w)+dg(f(x)).d^{2}f(x)(v,w).
 \end{equation}
 If $M$ is a manifold and we take the definition of the tangent bundle
 \begin{equation*}
  TM:=\left(\bigcup_{i\in I}\{i\}\times \varphi_{i}(U_{i})\times X\right)/\sim
 \end{equation*}
 with
 $(i,\varphi_{i}(x),v)\sim(i',\varphi_{i'}(x),d(\varphi_{i'}\op{\circ} \varphi_{i}^{-1})(\varphi _{i}(x)).v)$
 if $x\in U_{i}\cap U_{i'}$, then the map
 \begin{equation*}
  d^{2}f\from  (T_{x}M)^{2}\to T_{f(x)}N,\quad  [i,\varphi_{i}(x),v],[i,\varphi _{i}(x),w]\mapsto [j,\psi_{j}(f(x)),d^{2}(\psi_{j}\op{\circ}f \op{\circ} \varphi _{i}^{-1})(\varphi_{i}(x))(v,w)]
 \end{equation*}
 is well-defined accoring to \eqref{eqn:chain_rule_for_d2}.
\end{remark}

\begin{definition}\label{def:exponentialfunction}
 Let \mbox{\mbox{ $G$}} be a locally convex Lie group. The group
 \mbox{\mbox{ $G$}} is said to have an \emph{exponential function} if
 for each \mbox{\mbox{ $x \in \fg$}} the initial value problem
 \[
  \gamma (0)=e,\quad \gamma '(t)=T\lambda_{\gamma (t)}(e).x
 \]
 has a solution \mbox{\mbox{ $\gamma_{x}\in C^{\infty} (\R,G)$}} and the
 function
 \[
  \exp_{G}:\fg\to G,\;\;x \mapsto \gamma_x (1)
 \]
 is smooth. Furthermore, if there exists a zero neighbourhood
 \mbox{\mbox{ $W\se \fg$}} such that \mbox{\mbox{
 $\left.\exp_{G}\right|_{W}$}} is a diffeomorphism onto some open
 identity neighbourhood of \mbox{\mbox{ $G$}}, then \mbox{\mbox{ $G$}}
 is said to be \emph{locally exponential}.
\end{definition}

\begin{lemma}
 \label{lem:interchangeOfActionsOnGroupAndAlgebra} If \mbox{\mbox{ $G$}}
 and \mbox{\mbox{ $G'$}} are locally convex Lie groups with exponential
 function, then for each morphism \mbox{\mbox{ $\alpha :G\to G'$}} of
 Lie groups and the induced morphism \mbox{\mbox{
 $d\alpha (e):\fg \to \fg'$}} of Lie algebras, the diagram
 \[
  \begin{CD}
   G @>\alpha >>G'\\
   @AA\exp_{G}A  @AA\exp_{G'}A\\
   \fg@>d\alpha (e)>> \fg'
  \end{CD}
 \]
 commutes.
\end{lemma}

\begin{remark}
 \label{rem:banachLieGroupsAreLocallyExponential} The Fundamental
 Theorem of Calculus for locally convex spaces (cf.\ \cite[Thm.\
 1.5]{Glockner02Infinite-dimensional-Lie-groups-without-completeness-restrictions})
 yields that a locally convex Lie group \mbox{\mbox{ $G$}} can have at
 most one exponential function (cf.\ \cite[Lem.\
 II.3.5]{Neeb06Towards-a-Lie-theory-of-locally-convex-groups}).

 Typical examples of locally exponential Lie groups are Banach-Lie
 groups (by the existence of solutions of differential equations and the
 inverse mapping theorem, cf.\
 \cite{Lang99Fundamentals-of-differential-geometry}) and groups of
 smooth and continuous mappings from compact manifolds into locally
 exponential groups (\cite[Sect.\
 3.2]{Glockner02Lie-group-structures-on-quotient-groups-and-universal-complexifications-for-infinite-dimensional-Lie-groups},
 \cite{Wockel06Smooth-extensions-and-spaces-of-smooth-and-holomorphic-mappings}).
 However, diffeomorphism groups of compact manifolds are never locally
 exponential (cf.\ \cite[Ex.\
 II.5.13]{Neeb06Towards-a-Lie-theory-of-locally-convex-groups}) and
 direct limit Lie groups not always (cf.\ \cite[Rem.\
 4.7]{Glockner05Fundamentals-of-direct-limit-Lie-theory}). For a
 detailed treatment of locally exponential Lie groups and their
 structure theory we refer to \cite[Sect.\
 IV]{Neeb06Towards-a-Lie-theory-of-locally-convex-groups}.
\end{remark}

\begin{remark}\label{rem:mackey-complete-space}
 Let $X$ be a locally convex space. Then $X$ is said to be
 Mackey-complete if each Mackey-Cauchy sequence converges in $X$ (cf.\
 \cite[Sect.\
 I.2]{KrieglMichor97The-Convenient-Setting-of-Global-Analysis}). In
 particular, sequentially complete spaces are Mackey-complete. The main
 reason for working with this weaker concept of completeness is that it
 ensures the existence of (weak) integrals of smooth curves (cf.\
 \cite[Thm.\
 I.2.14]{KrieglMichor97The-Convenient-Setting-of-Global-Analysis}), even
 for non-complete spaces. Moreover, it implies the existence of
 integrals for smooth functions on cubes and standard simplices (cf.\
 \cite[Rem.\
 I.4.4]{Neeb06Towards-a-Lie-theory-of-locally-convex-groups}).
\end{remark}

\def\polhk#1{\setbox0=\hbox{#1}{\ooalign{\hidewidth
  \lower1.5ex\hbox{`}\hidewidth\crcr\unhbox0}}} \def\cprime{$'$}


\begin{thebibliography}{BCSS07}
\providecommand{\url}[1]{\texttt{#1}}
\providecommand{\urlprefix}{URL }
\providecommand{\eprint}[2][]{\url{#2}}

\bibitem[AM08]{AgoreMilitaru08Crossed-product-of-groups.-Applications}
Agore, A.~L. and Militaru, G.
\newblock \emph{{C}rossed product of groups. {A}pplications} 2008.
\newblock \href{http://arxiv.org/abs/0805.4077}{\texttt{arXiv:0805.4077}}

\bibitem[Bae45]{Baer45The-homomorphism-theorems-for-loops}
Baer, R.
\newblock \emph{The homomorphism theorems for loops}.
\newblock Amer. J. Math. \textbf{67} (1945):450--460

\bibitem[BBF05]{BlancoBullejosFaro05Categorical-non-abelian-cohomology-and-the-Schreier-theory-of-groupoids}
Blanco, V., Bullejos, M. and Faro, E.
\newblock \emph{Categorical non-abelian cohomology and the {S}chreier theory of
  groupoids}.
\newblock Math. Z. \textbf{251} (2005)(1):41--59.
\newblock \href{http://arxiv.org/abs/math/0410202}{\texttt{arXiv:math/0410202}}

\bibitem[BC04]{BaezCrans04Higher-dimensional-algebra.-VI.-Lie-2-algebras}
Baez, J.~C. and Crans, A.~S.
\newblock \emph{Higher-dimensional algebra. {VI}. {L}ie 2-algebras}.
\newblock Theory Appl. Categ. \textbf{12} (2004):492--538 (electronic).
\newblock \href{http://arxiv.org/abs/math/0307263}{\texttt{arXiv:math/0307263}}

\bibitem[BCSS07]{BaezCransStevensonSchreiber07From-loop-groups-to-2-groups}
Baez, J.~C., Crans, A.~S., Stevenson, D. and Schreiber, U.
\newblock \emph{From loop groups to 2-groups}.
\newblock Homology, Homotopy Appl. \textbf{9} (2007)(2):101--135.
\newblock \href{http://arxiv.org/abs/math/0504123}{\texttt{arXiv:math/0504123}}

\bibitem[BH08]{BaezHoffnung08Convenient-Categories-of-Smooth-Spaces}
Baez, J.~C. and Hoffnung, A.~E.
\newblock \emph{{C}onvenient {C}ategories of {S}mooth {S}paces} 2008.
\newblock To appear in Trans. Amer. Math. Soc., \href{http://arxiv.org/abs/0807.1704}{\texttt{arXiv:0807.1704}}

\bibitem[BHR10]{BaezHoffnungRogers10Categorified-symplectic-geometry-and-the-classical-string}
Baez, J.~C., Hoffnung, A.~E. and Rogers, C.~L.
\newblock \emph{Categorified symplectic geometry and the classical string}.
\newblock Comm. Math. Phys. \textbf{293} (2010)(3):701--725.
\newblock \href{http://arxiv.org/abs/0808.0246}{\texttt{arXiv:0808.0246}}

\bibitem[BL04]{BaezLauda04Higher-dimensional-algebra.-V.-2-groups}
Baez, J.~C. and Lauda, A.~D.
\newblock \emph{Higher-dimensional algebra. {V}. 2-groups}.
\newblock Theory Appl. Categ. \textbf{12} (2004):423--491 (electronic).
\newblock \href{http://arxiv.org/abs/math.QA/0307200}{\texttt{arXiv:math.QA/0307200}}

\bibitem[Blo08]{Blohmann08Stacky-Lie-groups}
Blohmann, C.
\newblock \emph{Stacky {L}ie groups}.
\newblock Int. Math. Res. Not. IMRN  (2008):Art. ID rnn 082, 51

\bibitem[BM93]{BrylinskiMcLaughlin93A-geometric-construction-of-the-first-Pontryagin-class}
Brylinski, J.-L. and McLaughlin, D.
\newblock \emph{A geometric construction of the first {P}ontryagin class}.
\newblock In \emph{Quantum topology}, \emph{Ser. Knots Everything}, vol.~3, pp.
  209--220 (World Sci. Publ., River Edge, NJ, 1993)

\bibitem[Bou98]{Bourbaki98Lie-groups-and-Lie-algebras.-Chapters-1--3}
Bourbaki, N.
\newblock \emph{Lie groups and {L}ie algebras. {C}hapters 1--3}.
\newblock Elements of Mathematics (Berlin) (Springer-Verlag, Berlin, 1998).
\newblock Translated from the French, Reprint of the 1989 English translation

\bibitem[BR08]{BourginRobart08An-infinite-dimensional-approach-to-the-third-fundamental-theorem-of-Lie}
Bourgin, R.~D. and Robart, T.~P.
\newblock \emph{An infinite dimensional approach to the third fundamental
  theorem of {L}ie}.
\newblock SIGMA Symmetry Integrability Geom. Methods Appl. \textbf{4}
  (2008):Paper 020, 10

\bibitem[Bre92]{Breen92Theorie-de-Schreier-superieure}
Breen, L.
\newblock \emph{Th\'eorie de {S}chreier sup\'erieure}.
\newblock Ann. Sci. \'Ecole Norm. Sup. (4) \textbf{25} (1992)(5):465--514

\bibitem[Bry93]{Brylinski93Loop-spaces-characteristic-classes-and-geometric-quantization}
Brylinski, J.-L.
\newblock \emph{Loop spaces, characteristic classes and geometric
  quantization}, \emph{Progress in Mathematics}, vol. 107 (Birkh{\"a}user
  Boston Inc., Boston, MA, 1993)

\bibitem[Car30]{Cartan30Le-troisieme-theoreme-fondamental-de-Lie}
Cartan, E.
\newblock \emph{{Le troisi\`eme th\'eor\`eme fondamental de Lie}}.
\newblock C. R. Math. Acad. Sci. Paris, \textbf{190} (1930):914--916

\bibitem[Car36]{Cartan36La-topologie-des-groupes-de-Lie}
Cartan, E.
\newblock \emph{La topologie des groupes de {L}ie}.
\newblock Actualit{\'e}s. Sci. Indust. \textbf{358} (1936)(?)

\bibitem[CF03]{CrainicFernandes03Integrability-of-Lie-brackets}
Crainic, M. and Fernandes, R.~L.
\newblock \emph{Integrability of {L}ie brackets}.
\newblock Ann. of Math. (2) \textbf{157} (2003)(2):575--620

\bibitem[DL66]{DouadyLazard66Espaces-fibres-en-algebres-de-Lie-et-en-groupes}
Douady, A. and Lazard, M.
\newblock \emph{Espaces fibr\'es en alg\`ebres de {L}ie et en groupes}.
\newblock Invent. Math. \textbf{1} (1966):133--151

\bibitem[EK64]{EstKorthagen64Non-enlargible-Lie-algebras}
van Est, W.~T. and Korthagen, T.~J.
\newblock \emph{Non-enlargible {L}ie algebras}.
\newblock Nederl. Akad. Wetensch. Proc. Ser. A 67=Indag. Math. \textbf{26}
  (1964):15--31

\bibitem[EM47]{EilenbergMacLane47Algebraic-cohomology-groups-and-loops}
Eilenberg, S. and MacLane, S.
\newblock \emph{Algebraic cohomology groups and loops}.
\newblock Duke Math. J. \textbf{14} (1947):435--463

\bibitem[Est62a]{Est62Local-and-global-groups.-I}
van Est, W.~T.
\newblock \emph{Local and global groups. {I}}.
\newblock Nederl. Akad. Wetensch. Proc. Ser. A 65 = Indag. Math. \textbf{24}
  (1962):391--408

\bibitem[Est62b]{Est62Local-and-global-groups.-II}
van Est, W.~T.
\newblock \emph{Local and global groups. {II}}.
\newblock Nederl. Akad. Wetensch. Proc. Ser. A 65 = Indag. Math. \textbf{24}
  (1962):409--425

\bibitem[FB02]{Forrester-Barker02Group-Objects-and-Internal-Categories}
Forrester-Barker, M.
\newblock \emph{{G}roup {O}bjects and {I}nternal {C}ategories} 2002.
\newblock \href{http://arxiv.org/abs/math/0212065}{\texttt{arXiv:math/0212065}}

\bibitem[Gl{\"o}02a]{Glockner02Infinite-dimensional-Lie-groups-without-completeness-restrictions}
Gl{\"o}ckner, H.
\newblock \emph{Infinite-dimensional {L}ie groups without completeness
  restrictions}.
\newblock In \emph{Geometry and Analysis on Finite- and Infinite-Dimensional
  {L}ie Groups (B\polhk edlewo, 2000)}, \emph{Banach Center Publ.}, vol.~55,
  pp. 43--59 (Polish Acad. Sci., Warsaw, 2002)

\bibitem[Gl{\"o}02b]{Glockner02Lie-group-structures-on-quotient-groups-and-universal-complexifications-for-infinite-dimensional-Lie-groups}
Gl{\"o}ckner, H.
\newblock \emph{Lie group structures on quotient groups and universal
  complexifications for infinite-dimensional {L}ie groups}.
\newblock J. Funct. Anal. \textbf{194} (2002)(2):347--409

\bibitem[Gl{\"o}05]{Glockner05Fundamentals-of-direct-limit-Lie-theory}
Gl{\"o}ckner, H.
\newblock \emph{Fundamentals of direct limit {L}ie theory}.
\newblock Compos. Math. \textbf{141} (2005)(6):1551--1577

\bibitem[Hen08]{Henriques08Integrating-Lsb-infty-algebras}
Henriques, A.
\newblock \emph{Integrating {$L\sb \infty$}-algebras}.
\newblock Compos. Math. \textbf{144} (2008)(4):1017--1045.
\newblock \href{http://arxiv.org/abs/math/0603563}{\texttt{arXiv:math/0603563}}

\bibitem[HM98]{HofmannMorris98The-structure-of-compact-groups}
Hofmann, K.~H. and Morris, S.~A.
\newblock \emph{The structure of compact groups}, \emph{de Gruyter Studies in
  Mathematics}, vol.~25 (Walter de Gruyter \& Co., Berlin, 1998).
\newblock A primer for the student---a handbook for the expert

\bibitem[HS90]{HofmannStrambach90Topological-and-analytic-loops}
Hofmann, K.~H. and Strambach, K.
\newblock \emph{Topological and analytic loops}.
\newblock In \emph{Quasigroups and loops: theory and applications}, \emph{Sigma
  Ser. Pure Math.}, vol.~8, pp. 205--262 (Heldermann, Berlin, 1990)

\bibitem[Igl95]{Iglesias95La-trilogie-du-moment}
Igl{\'e}sias, P.
\newblock \emph{La trilogie du moment}.
\newblock Ann. Inst. Fourier (Grenoble) \textbf{45} (1995)(3):825--857

\bibitem[Igl08]{Iglesias08The-moment-maps-in-diffeology}
Igl{\'e}sias, P.
\newblock \emph{The moment maps in diffeology}.
\newblock Memoirs of the AMS (AMS, 2008).
\newblock To appear

\bibitem[Jac62]{Jacobson62Lie-algebras}
Jacobson, N.
\newblock \emph{Lie algebras}.
\newblock Interscience Tracts in Pure and Applied Mathematics, No. 10
  (Interscience Publishers (a division of John Wiley \& Sons), New York-London,
  1962)

\bibitem[KK04]{KiechleKinyon04Infinite-simple-Bol-loops}
Kiechle, H. and Kinyon, M.~K.
\newblock \emph{Infinite simple {B}ol loops}.
\newblock Comment. Math. Univ. Carolin. \textbf{45} (2004)(2):275--278

\bibitem[KM97]{KrieglMichor97The-Convenient-Setting-of-Global-Analysis}
Kriegl, A. and Michor, P.~W.
\newblock \emph{The {C}onvenient {S}etting of {G}lobal {A}nalysis},
  \emph{Mathematical Surveys and Monographs}, vol.~53 (American Mathematical
  Society, Providence, RI, 1997)

\bibitem[Lan99]{Lang99Fundamentals-of-differential-geometry}
Lang, S.
\newblock \emph{Fundamentals of differential geometry}, \emph{Graduate Texts in
  Mathematics}, vol. 191 (Springer-Verlag, New York, 1999)

\bibitem[Lem97]{Lempert97The-problem-of-complexifying-a-Lie-group}
Lempert, L.
\newblock \emph{The problem of complexifying a {L}ie group}.
\newblock In \emph{Multidimensional complex analysis and partial differential
  equations ({S}\~ao {C}arlos, 1995)}, \emph{Contemp. Math.}, vol. 205, pp.
  169--176 (Amer. Math. Soc., Providence, RI, 1997)

\bibitem[Lod82]{Loday82Spaces-with-finitely-many-nontrivial-homotopy-groups}
Loday, J.-L.
\newblock \emph{Spaces with finitely many nontrivial homotopy groups}.
\newblock J. Pure Appl. Algebra \textbf{24} (1982)(2):179--202

\bibitem[Mac63]{MacLane63Homology}
MacLane, S.
\newblock \emph{Homology}.
\newblock Die Grundlehren der mathematischen Wissenschaften, Bd. 114 (Academic
  Press Inc., Publishers, New York, 1963)

\bibitem[Mic87]{Mickelsson87Kac-Moody-groups-topology-of-the-Dirac-determinant-bundle-and-fermionization}
Mickelsson, J.
\newblock \emph{Kac-{M}oody groups, topology of the {D}irac determinant bundle,
  and fermionization}.
\newblock Comm. Math. Phys. \textbf{110} (1987)(2):173--183

\bibitem[MN03]{MaierNeeb03Central-extensions-of-current-groups}
Maier, P. and Neeb, K.-H.
\newblock \emph{Central extensions of current groups}.
\newblock Math. Ann. \textbf{326} (2003)(2):367--415

\bibitem[MS03]{MurrayStevenson03Higgs-fields-bundle-gerbes-and-string-structures}
Murray, M.~K. and Stevenson, D.
\newblock \emph{Higgs fields, bundle gerbes and string structures}.
\newblock Comm. Math. Phys. \textbf{243} (2003)(3):541--555

\bibitem[Nee02]{Neeb02Central-extensions-of-infinite-dimensional-Lie-groups}
Neeb, K.-H.
\newblock \emph{Central extensions of infinite-dimensional {L}ie groups}.
\newblock Ann. Inst. Fourier (Grenoble) \textbf{52} (2002)(5):1365--1442

\bibitem[Nee06]{Neeb06Towards-a-Lie-theory-of-locally-convex-groups}
Neeb, K.-H.
\newblock \emph{Towards a {L}ie theory of locally convex groups}.
\newblock Jpn. J. Math. \textbf{1} (2006)(2):291--468

\bibitem[NS02]{NagyStrambach02Loops-in-group-theory-and-Lie-theory}
Nagy, P.~T. and Strambach, K.
\newblock \emph{Loops in group theory and {L}ie theory}, \emph{de Gruyter
  Expositions in Mathematics}, vol.~35 (Walter de Gruyter \& Co., Berlin, 2002)

\bibitem[NSW11]{NikolausSachseWockel11A-Smooth-Model-for-the-String-Group}
Nikolaus, T., Sachse, C. and Wockel, C.
\newblock \emph{{A} {S}mooth {M}odel for the {S}tring {G}roup} 2011.
\newblock \href{http://arxiv.org/abs/1104.4288}{\texttt{arXiv:1104.4288}}

\bibitem[NW11]{NikolausWaldorf11Four-Equivalent-Versions-of-Non-Abelian-Gerbes}
Nikolaus, T. and Waldorf, K.
\newblock \emph{{F}our {E}quivalent {V}ersions of {N}on-{A}belian {G}erbes}
  2011.
\newblock \href{http://arxiv.org/abs/1103.4815}{\texttt{arXiv:1103.4815}}

\bibitem[Olv96]{Olver96Non-associative-local-Lie-groups}
Olver, P.~J.
\newblock \emph{Non-associative local {L}ie groups}.
\newblock J. Lie Theory \textbf{6} (1996)(1):23--51

\bibitem[Pfl90]{Pflugfelder90Quasigroups-and-loops:-introduction}
Pflugfelder, H.~O.
\newblock \emph{Quasigroups and loops: introduction}, \emph{Sigma Series in
  Pure Mathematics}, vol.~7 (Heldermann Verlag, Berlin, 1990)

\bibitem[Por08]{Porst08Strict-2-Groups-are-Crossed-Modules}
Porst, S.~S.
\newblock \emph{{S}trict 2-{G}roups are {C}rossed {M}odules} 2008.
\newblock \href{http://arxiv.org/abs/0812.1464}{\texttt{arXiv:0812.1464}}

\bibitem[Pra68]{Pradines68Troisieme-theoreme-de-Lie-les-groupoi-des-differentiables}
Pradines, J.
\newblock \emph{Troisi{\`e}me th{\'e}or{\`e}me de {L}ie les groupo\"\i des
  diff{\'e}rentiables}.
\newblock C. R. Acad. Sci. Paris S{\'e}r. A-B \textbf{267} (1968):A21--A23

\bibitem[PS86]{PressleySegal86Loop-groups}
Pressley, A. and Segal, G.
\newblock \emph{Loop groups}.
\newblock {O}xford Mathematical Monographs (The Clarendon Press Oxford
  University Press, New York, 1986).
\newblock {O}xford Science Publications

\bibitem[PW]{PorstWockel08Higher-conneced-covers-of-topological-groups-via-categorified-central-extensions}
Porst, S.~S. and Wockel, C.
\newblock \emph{Higher connected covers of topological groups via
  categorification}.
\newblock In preparation

\bibitem[Sch08]{Schreiberpersonal-communicaltion}
Schreiber, U.
\newblock \emph{personal communicaltion} 2008.
\newblock
  \urlprefix\url{http://www.math.uni-hamburg.de/home/schreiber/Lausanne.pdf}

\bibitem[Sch11]{Schreiber11Differential-Cohomology-in-a-Cohesive-infty-Topos}
Schreiber, U.
\newblock \emph{Differential cohomology in a cohesive $\infty$-topos}.
\newblock published online 2011.
\newblock
  \urlprefix\url{http://nlab.mathforge.org/schreiber/files/cohesivedocumentv028.pdf}

\bibitem[Smi50]{Smith50Some-Notions-Connected-with-a-Set-of-Generators}
Smith, P.~A.
\newblock \emph{Some notions connected with a set of generators}.
\newblock In \emph{Proceedings of the International Congress of
  Mathematicians}, vol.~2, pp. 436--441 (1950)

\bibitem[SP11]{Schommer-Pries10Central-Extensions-of-Smooth-2-Groups-and-a-Finite-Dimensional-String-2-Group}
Schommer-Pries, C.
\newblock \emph{{C}entral extensions of smooth 2-groups and a
  finite-dimensional string 2-group}.
\newblock Geom. Topol. \textbf{15} (2011):609---676.
\newblock \href{http://arxiv.org/abs/0911.2483}{\texttt{arXiv:0911.2483}}

\bibitem[ST04]{StolzTeichner04What-is-an-elliptic-object}
Stolz, S. and Teichner, P.
\newblock \emph{What is an elliptic object?}
\newblock In \emph{Topology, geometry and quantum field theory}, \emph{London
  Math. Soc. Lecture Note Ser.}, vol. 308, pp. 247--343 (Cambridge Univ. Press,
  Cambridge, 2004)

\bibitem[Sto96]{Stolz96A-conjecture-concerning-positive-Ricci-curvature-and-the-Witten-genus}
Stolz, S.
\newblock \emph{A conjecture concerning positive {R}icci curvature and the
  {W}itten genus}.
\newblock Math. Ann. \textbf{304} (1996)(4):785--800

\bibitem[SW08]{SchreiberWaldorf08Connections-on-non-abelian-Gerbes-and-their-Holonomy}
Schreiber, U. and Waldorf, K.
\newblock \emph{{C}onnections on non-abelian {G}erbes and their {H}olonomy}
  2008.
\newblock \href{http://arxiv.org/abs/0808.1923}{\texttt{arXiv:0808.1923}}

\bibitem[TW87]{TuynmanWiegerinck87Central-extensions-and-physics}
Tuynman, G.~M. and Wiegerinck, W. A. J.~J.
\newblock \emph{Central extensions and physics}.
\newblock J. Geom. Phys. \textbf{4} (1987)(2):207--258

\bibitem[TZ06]{TsengZhu06Integrating-Lie-algebroids-via-stacks}
Tseng, H.-H. and Zhu, C.
\newblock \emph{Integrating {L}ie algebroids via stacks}.
\newblock Compos. Math. \textbf{142} (2006)(1):251--270

\bibitem[Wal10]{Waldorf10Multiplicative-bundle-gerbes-with-connection}
Waldorf, K.
\newblock \emph{Multiplicative bundle gerbes with connection}.
\newblock Differential Geom. Appl. \textbf{28} (2010)(3):313--340

\bibitem[Woc06]{Wockel06Smooth-extensions-and-spaces-of-smooth-and-holomorphic-mappings}
Wockel, C.
\newblock \emph{Smooth extensions and spaces of smooth and holomorphic
  mappings}.
\newblock J. Geom. Symmetry Phys. \textbf{5} (2006):118--126

\bibitem[Woc09]{Wockel06A-Generalisation-of-Steenrods-Approximation-Theorem}
Wockel, C.
\newblock \emph{{A} {G}eneralisation of {S}teenrod's {A}pproximation
  {T}heorem}.
\newblock Arch. Math. (Brno) \textbf{45} (2009)(2):95--104.
\newblock \href{http://arxiv.org/abs/math/0610252}{\texttt{arXiv:math/0610252}}

\bibitem[Woc11]{Wockel09Principal-2-bundles-and-their-gauge-2-groups}
Wockel, C.
\newblock \emph{{P}rincipal 2-bundles and their gauge 2-groups}.
\newblock Forum Math. \textbf{23} (2011):565--610.
\newblock \href{http://arxiv.org/abs/0803.3692}{\texttt{arXiv:0803.3692}}

\bibitem[WX91]{WeinsteinXu91Extensions-of-symplectic-groupoids-and-quantization}
Weinstein, A. and Xu, P.
\newblock \emph{Extensions of symplectic groupoids and quantization}.
\newblock J. Reine Angew. Math. \textbf{417} (1991):159--189

\bibitem[ZW]{ZhuWockel10Integrating-central-extensions-of-Lie-algebras-via-group-stacks}
Zhu, C. and Wockel, C.
\newblock \emph{Integrating central extensions of lie algebras via group
  stacks}.
\newblock In preparation

\end{thebibliography}
\end{document}